\documentclass[10pt]{article}
\usepackage[a4paper]{geometry}
\usepackage{graphicx,version,color,epstopdf,enumitem,latexsym,amsmath,amsfonts,amsthm,setspace,mathtools}
\usepackage[title]{appendix}
\usepackage{caption}
\graphicspath{{figs/}}
\newtheorem{remark}{Remark}[section]
\newtheorem{theorem}{Theorem}[section]
\newtheorem{lemma}{Lemma}[section]

\newtheorem{corollary}{Corollary}[section]

\numberwithin{equation}{section}
\numberwithin{figure}{section}

\allowdisplaybreaks

\usepackage{hyperref}
\hypersetup{colorlinks=true,citecolor=blue,linkcolor=blue,anchorcolor=blue}

\title{Pointwise error estimates and local superconvergence of Jacobi expansions}

\author{Shuhuang Xiang\thanks{School of Mathematics and Statistics, INP-LAMA, Central South University, Changsha 410083,  P. R. China, xiangsh@csu.edu.cn} \and Desong Kong\thanks{School of Mathematics and Statistics, Central South University, Changsha 410083, P. R. China, desongkong@csu.edu.cn} \and Guidong Liu\thanks{Corresponding author, School of  Statistics and Mathematics, Nanjing Audit University, Nanjing 211815, P. R. China, csu\_guidongliu@163.com} \and Li-Lian Wang\thanks{Division of Mathematical Sciences, School of Physical and Mathematical Sciences, Nanyang Technological University, 637371, Singapore,  lilian@ntu.edu.sg}}

\begin{document}

\maketitle

\begin{abstract} As one myth of polynomial interpolation and quadrature, Trefethen \cite{Trefethen} revealed that the Chebyshev
interpolation of $|x-a|$ (with $|a|<1 $) at  the Clenshaw-Curtis points exhibited a much smaller error than the best polynomial approximation (in the maximum norm)
in about $95\%$ range of $[-1,1]$ except for a small neighbourhood near the singular point $x=a.$
In this paper, we rigorously show that the Jacobi expansion for a more general class of   $\Phi$-functions also enjoys such a local convergence behaviour.   Our assertion  draws on the pointwise error estimate using  the reproducing kernel of Jacobi polynomials
and  the Hilb-type formula on the asymptotic of the Bessel transforms.  We also study the local superconvergence
and show the gain in order and the subregions it occurs. As a by-product of this new argument, the undesired $\log n$-factor in the pointwise error estimate for the Legendre expansion recently stated in
Babu\u{s}ka and  Hakula \cite{Babuska2019} can be removed.  Finally, all these estimates are extended to the functions with boundary singularities. We provide ample numerical evidences to demonstrate the optimality and sharpness of the estimates.
\end{abstract}

\vspace{0.05in}


\section{Introduction}
\label{introduction}
Approximation by polynomials  plays a fundamental role in algorithm development and  numerical analysis of  many computational methods. It is known that for a given continuous function $f(x)$ defined on $[-1,1]$, the best polynomial approximation  of $f(x)$  in the maximum norm is  a unique  polynomial $p_n^*\in {\cal P}_n$ (denotes the set of polynomials of degree  at most $n$) that minimizes
\begin{equation*}
	\|f-p_n^*\|_{\infty}=\min_{p\in {\cal P}_n}\|f-p\|_{\infty}.
\end{equation*}
The best polynomial approximation $p_n^*(x)$ is optimal, but its computation is  nontrivial for a general nonlinear  function $f(x)$\,\cite{Trefethen1}. In fact, Trefethen \cite{Trefethen} pointed out that for $f(x)=|x-\frac{1}{4}|$, the pointwise   error $|f(x)-p_n(x)|$ by the  polynomial interpolant  $p_n$  at  the Clenshaw-Curtis points $\{x_j=\cos(\frac{j\pi}{n})\}_{j=0}^n$ is much smaller than that of the best polynomial: $|f(x)-p_n^*(x)|$  for most values of $x,$ except for a small subinterval centred around the singular point $x=\frac{1}{4}$  (see Fig.\! \ref{figure11} (left) for an illustration of  $n=100$).

\begin{figure}[hpbt]
\centerline{\includegraphics[height=4.5cm,width=14.5cm]{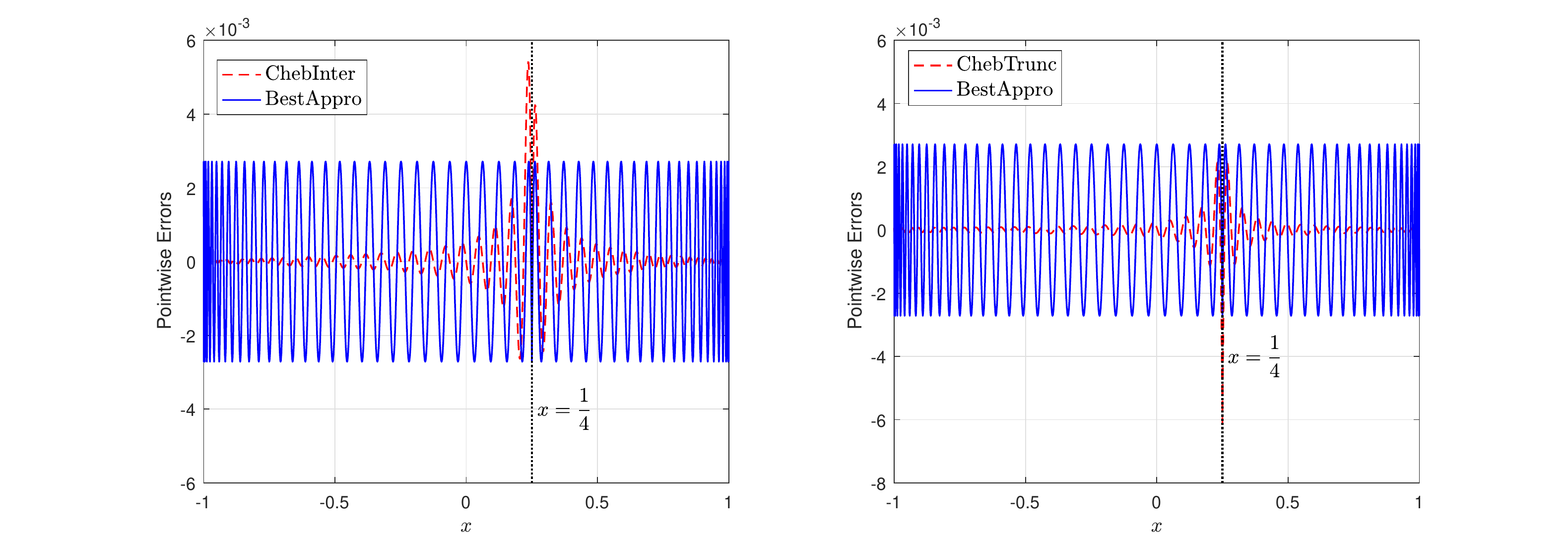}}
\vspace*{-10pt}
\caption{\small Pointwise error curves of the best polynomial approximation $f(x)-p_n^*(x)$, Chebyshev interpolation $f(x)-p_n(x)$ (left), and Chebyshev truncation $f(x)-S_{n}^{(-\frac{1}{2},-\frac{1}{2})}[f](x)$ (right), where $n=100$.}\label{figure11}
\end{figure}

Needless to say,  the pointwise error is a very useful  indication of  the approximability and approximation quality  of a numerical tool in solving partial differential equations \cite{BaGuo3,CSYZ,CSYZ2,CZZ,CCS,LZ}.  In the past several decades,  the error estimates of spectral approximation in Sobolev norms   has been intensively studied and well-documented in e.g., \cite{BaGuo1,BaGuo2,LWL,STW,Trefethen1,Tuan}. However, whenever possible,  one would wish to estimate
the pointwise error of the approximation \cite{Babuska2019,LWW,Tuan,Wang1,Zhang,ZhangZ,ZZ}, though it is usually more challenging.

Compared with the aforementioned Chebyshev interpolation $p_n(x)$,   the pointwise error $f(x)-S^{(-\frac 1 2,-\frac 12)}_n[f](x)$ of  the Chebyshev spectral projection $S^{(-\frac 1 2,-\frac 12)}_n[f]$ of $f(x)=|x-\frac{1}{4}|$  is much smaller than the best polynomial approximation $f(x)-p_n^*(x)$ except for the subinterval near  the  singular point  $x=\frac{1}{4}$, and it is more localized than $f(x)-p_n(x)$ near the singularity ( see  Fig.\! \ref{figure11} (right)). This  implies some  underlying local superconvergence at the points slightly away from  the singularity.

\begin{figure}[hpbt]
\centerline{\includegraphics[height=4.5cm,width=14.5cm]{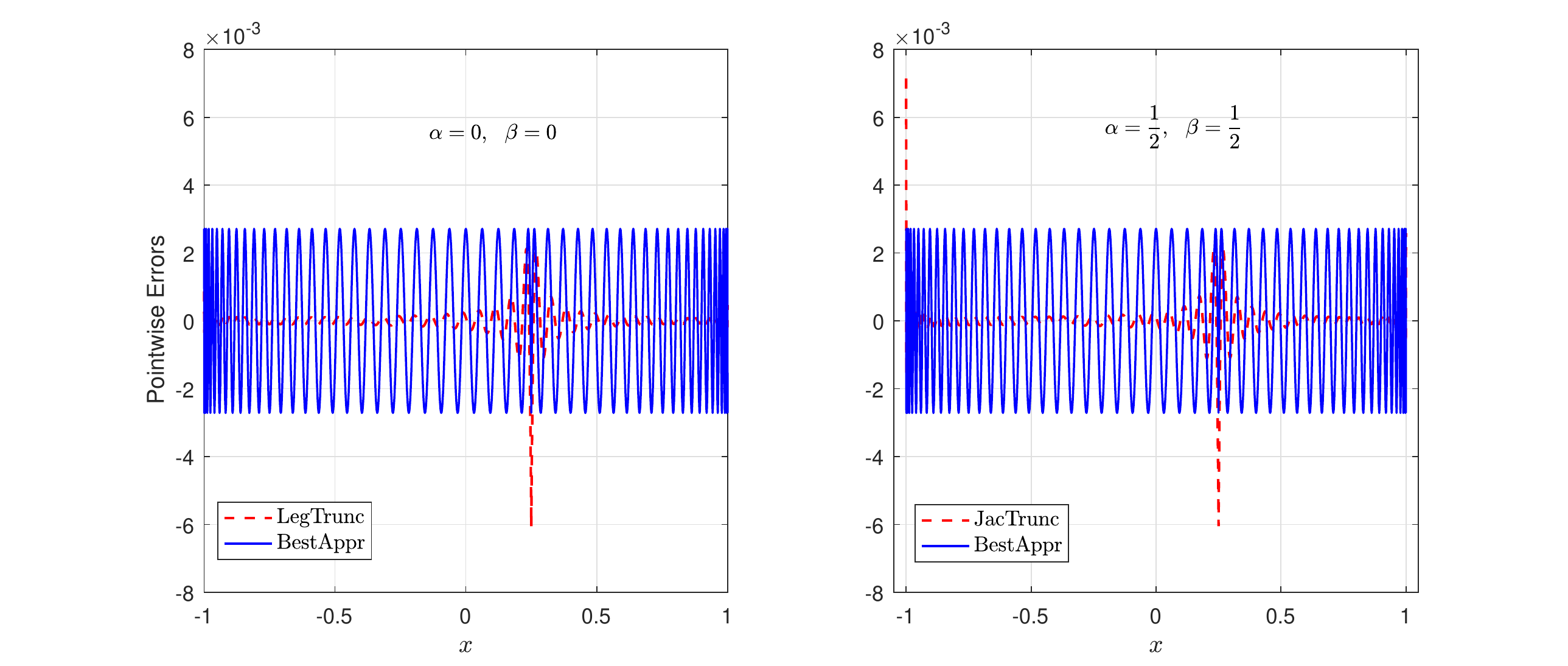}}
\vspace*{-10pt}
\caption{\small  Pointwise error curves of the best polynomial approximation  $f(x)-p_{n}^{*}(x)$ and the Jacobi truncation $f(x)-S_{n}^{(\alpha,\beta)}[f](x)$  with $\alpha=\beta=0$ (left) and $\alpha=\beta=\frac 12$  (right),  where $n=100$.}\label{figure12}
\end{figure}

Interestingly, such a superconvergence phenomenon also occurs in the Legendre and more general Jacobi expansions (except for
an additional  small neighbourhood of the endpoints $x=\pm 1$, see Fig.\!\! \ref{figure12}). In fact, we are not the first to unfold this  convergence behaviour.  In a recent work,  Babu\u{s}ka and  Hakula \cite{Babuska2019} provided deep insights into this phenomenon for the Legendre expansion of the class of $\Phi$-functions defined by

\begin{equation}\label{pfun}
	f(x)=(x-a)_+^\lambda=\begin{dcases}
	0,&-1\le x<a,\\
	(x-a)^{\lambda},&a<x\le 1,
	\end{dcases}
	\quad\,\lambda>-1,\quad a\in(-1,1),
\end{equation}
which  appears frequently in various applications \cite{Babuska2019}. More precisely, define the truncated Legendre series and the pointwise error as
\begin{equation}\label{elpfun}
    S_n^{(0,0)}[f](x)=\sum_{k=0}^na_kP_k(x),\quad e_f(n,x)=\big|f(x)-S_n^{(0,0)}[f](x)\big|,
\end{equation}
where $P_k(x)$ is the Legendre polynomial of degree $k$ as in \cite{Szego} and  $\{a_k\}$ are the Legendre expansion coefficients.
Taking into account  the convergence rates on the piecewise analytic functions in Saff and Totik \cite{st}, Babu\u{s}ka and  Hakula \cite{Babuska2019} derived the following estimates.
\begin{theorem}[see \cite{Babuska2019}]\label{Thm11}
    Let $f(x)$  be a $\Phi$-function defined by \eqref{pfun} with $\lambda=0$, i.e., a step function.
    \begin{itemize}
    \item[{\rm(i)}]  For $x\in (-1,a)\cup (a,1)$, we have $e_f(n,x)\le C(x)n^{-1}$, where  $C(x)$ is independent of $n$ and has the behaviours near $x=a,\pm 1$ as follows
    \begin{equation}\label{Cpmxi}
        C(-1+ \xi)\le D(-1)\,\xi^{-\frac 14},\quad C(1- \xi)\le D(1)\,\xi^{-\frac 1 4},\quad  C(a\pm \xi)\le D(a)\,\xi^{-1},
    \end{equation}
    for $0<\xi\le \delta,$ where $D(\pm 1), D(a)>0$ and $\delta>0$ are independent of $n.$
    \smallskip
    \item[{\rm (ii)}]  At  $x=\pm 1, a$,  we have
    \begin{equation}\label{Cpmapm1}
        e_f(n,\pm 1)\le Cn^{-\frac{1}{2}},\quad  e_f(n, a)\le Cn^{-1}.
    \end{equation}
    \end{itemize}
\end{theorem}

     In \eqref{Cpmapm1} and what follows,  we denote by  $C$   a generic positive constant independent of $n$ which may have a different value  in a different context.

 Following Wahlbin \cite{Wahlbin} and Bary \cite{Bary}, Babu\u{s}ka and  Hakula \cite{Babuska2019}  further obtained the following estimates  for
  $\lambda\not=0$.
 \smallskip
\begin{theorem}[see \cite{Babuska2019}]\label{Thm12}
Let $f(x)$  be a $\Phi$-function defined by \eqref{pfun} with $\lambda>-1$ but $\lambda\not=0$.
\begin{itemize}
\item[{\rm(i)}]  For $x\in (-1,a)\cup (a,1)$, we have
\begin{equation}\label{logfactor}
e_f(n,x)\le C(x)n^{-\lambda-1}\log n,
\end{equation}
where $C(x)>0$ is independent of $n,$ and has the same behaviour as $C(x)$ in \eqref{Cpmxi}.
\smallskip
\item[{\rm(ii)}]  At $x=\pm1, a$,  we have
\begin{equation}\label{Cpmapmlbda}
e_f(n,\pm 1)\le Cn^{-\lambda-\frac{1}{2}}\log n, \quad  e_f(n, a)\le
\begin{dcases}
  Cn^{-\lambda-1}\log n, & \mbox{$\lambda$ {\rm even},}\\[2pt]
  Cn^{-\lambda}\log n,   & \mbox{$\lambda>0$ {\rm and non-even}}.
\end{dcases}
\end{equation}
\end{itemize}
\end{theorem}

\noindent Some remarks are in order.
\smallskip
\begin{itemize}
\item From ample delicate  numerical experiments, Babu\u{s}ka and  Hakula \cite{Babuska2019} conjectured  that the multiplicative factor $\log n$ in Theorem \ref{Thm12} seems to be a defect  of the analysis technique employed in the proof, and Theorem \ref{Thm12}
should   hold without the $\log n$ factor.  This  was  stated  as a hypothesis  and  claimed  ``in spite of many attempts, the hypothesis underlines the need for new theory" in \cite{Babuska2019}.
\smallskip
\item  It is worthy of mentioning that  Kruglov  extended the numerical study in \cite{Babuska2019} for the Legendre expansions
to the more general Jacobi polynomial cases %{\cred by numerical illustrarion}
in the master thesis \cite{Kruglov}, but the $\log$-term remained as a conjecture in the results therein.
\smallskip
\item It is seen from Theorem \ref{Thm12} that if $\lambda$ is not an even integer,  we  have the superconvergence
\begin{equation}\label{newsuper}
 e_f(n,x)\le C(x) n^{-1}\log n\,  \|f-S_n^{(0,0)}[f]\|_\infty, %\quad x\in (-1,a)\cup (a,1),}
\end{equation}
with  a gain of convergence rate   ${\cal O} (n^{-1}\log n)$ %higher except a factor $\log n$}
on any closed subinterval that excludes $x=a, \pm 1.$
\end{itemize}

The main purposes of this paper are twofold. Firstly, using a new technique, we shall show that the log-factor can be removed. Secondly,  we shall conduct  the optimal  pointwise convergence  and  superconvergence analysis for the Jacobi expansions of the following  generalised $\Phi$-functions
\begin{equation}\label{gpfun}
f(x)=z(x)\cdot\begin{dcases}
0,&-1\le x<a,\\
(x-a)^{\lambda},&a<x\le 1,\end{dcases} \quad a\in(-1,1),\quad\lambda>-1,
\end{equation}
and
\begin{equation}\label{gafun}
  f(x)=|x-a|^{\lambda}z(x) \;\;\;  (\lambda>-1\;\; \text{is not an even integer}),
\end{equation}
where we set in (\ref{gpfun}) $f(a)=0$ for $\lambda>0$ and $f(a)=\frac{z(a)}{2}$ for $\lambda=0$, and the given function $z(x)$ involved is assumed to be smooth with $z(a)\neq0$.
Denote the Jacobi expansion of $f(x)$ in (\ref{gpfun}) or (\ref{gafun}) and the pointwise error  by
\begin{equation}\label{sepAfun}
S_n^{(\alpha,\beta)}[f](x)=\sum_{k=0}^n a_{k}^{(\alpha,\beta)}P_k^{(\alpha,\beta)}(x),\;\;\;\;  e_f^{(\alpha,\beta)}(n,x)=|f(x)-S_n^{(\alpha,\beta)}[f](x)|,
\end{equation}
respectively, where $P_k^{(\alpha,\beta)}(x)$ is the Jacobi polynomial of degree $k$ and
\begin{equation}\label{eq:jacexpcoeffs0}
\begin{split}
  & a_{k}^{(\alpha,\beta)}=\frac{1}{\sigma_{k}^{(\alpha,\beta)}}\int_{-1}^{1}\!f(x)P_k^{(\alpha,\beta)}(x) \omega^{(\alpha,\beta)}(x)\mathrm{d}x,
 \quad  \omega^{(\alpha,\beta)}(x)=(1-x)^\alpha(1+x)^\beta,
  \\& \sigma_{k}^{(\alpha,\beta)}=\frac{2^{\alpha+\beta+1}\,\Gamma(k+\alpha+1)\Gamma(k+\beta+1)}{k!(2k+\alpha+\beta+1)\Gamma(k+\alpha+\beta+1)}.
  \end{split}
\end{equation}

Using the reproducing kernel of Jacobi polynomials, together with the Hilb-type formula and van der Corput-type Lemma on the asymptotic  of the Bessel transforms, we are able to  derive the following
main results.

\smallskip
\begin{theorem}\label{Thm13}
Let $f(x)$  be a generalised $\Phi$-function  defined in \eqref{gpfun} or \eqref{gafun}. Then for $\alpha,\beta >-1$ and  $\lambda>-1$, we have the following pointwise error estimates.
\begin{itemize}
\item[{\rm(i)}]  For $x\in (-1,a)\cup (a,1)$, we have
\begin{equation}\label{PerrJac100}
e_f^{(\alpha,\beta)}(n,x)\le C(x)n^{-\lambda-1},
\end{equation}
where  $C(x)$ is independent of $n$ and has the behaviours near $x=a,\pm 1$ as follows
\begin{equation}\label{PerrJac1}
\begin{split}
& C(-1+ \xi)\le D(-1)\,\xi^{-\max\{\frac \beta2+\frac 14,0\}},\quad C(1- \xi)\le D(1)\,\xi^{-\max\{\frac \alpha2+\frac 14,0\}},\\
&  C(a\pm \xi)\le D(a)\,\xi^{-1},\quad 0<\xi\le \delta.
\end{split}
\end{equation}
%for $0<\xi\le \delta,$
Here  $D(\pm 1), D(a)>0$ and $\delta>0$ are  independent of $n.$
\smallskip
\item[{\rm (ii)}]  At  $x=\pm 1$,  we have
\begin{equation}\label{PerrJac2}
e_f^{(\alpha,\beta)}(n,1)\le C n^{-\lambda+\alpha-\frac{1}{2}},\quad e_f^{(\alpha,\beta)}(n,-1)\le C n^{-\lambda+\beta-\frac{1}{2}}.
\end{equation}

\smallskip
\item[{\rm (iii)}]  At $x=a$ and for $\lambda>0$,  we have
\begin{equation}\label{PerrJac3}
\begin{split}
& e_f^{(\alpha,\beta)}(n, a)\le\begin{cases}
Cn^{-\lambda-1},&\lambda\ \text{even},\\
 Cn^{-\lambda}, &\text{otherwise},  \end{cases}\quad \mbox{for $f(x)$ defined by \eqref{gpfun}}; \\[2pt]
& e_f^{(\alpha,\beta)}(n, a)\le Cn^{-\lambda}\quad  \mbox{for $f(x)$ defined by  \eqref{gafun} and  non-even $\lambda$}.
\end{split}
 \end{equation}
\end{itemize}
\end{theorem}

We emphasize  that all the above estimates are optimal in the sense that the convergence order can not be improved, which will be illustrated numerically in Section \ref{sec3}. As a special case, the multiplicative factor $\log n$ in Theorem \ref{Thm12} for the Legendre expansion is removed.  The asymptotic behavior of the pointwise error around the endpoints is described clearly. Indeed, we infer from (\ref{PerrJac1}) and (\ref{PerrJac2}) that $e_{f}^{(\alpha,\beta)}(n,x)$ achieves the best convergence rate around $x=\pm1$ when $\alpha,\beta\leq-\frac{1}{2}$. As an example, we consider  $f(x)=|x-a|$. It is known that the pointwise error of the best polynomial approximation equally oscillates  $N\ge n+2$ times and converges linearly as $n\to\infty$, i.e., there exist at least $N\ge n+2$ distinct  points $x_{1}, x_{2},\cdots, x_{N}$ on $[-1,1]$ such that
\begin{equation*}
  f(x_{i})-p_{n}^{*}(x_{i})=\varepsilon(-1)^{i}\|f-p_{n}^{*}\|_{\infty},\quad \|f-p_{n}^{*}\|_{\infty}\sim\sigma\sqrt{1-a^{2}}\,n^{-1},
\end{equation*}
where $\varepsilon=\pm1$ and $\sigma\approx1/2\sqrt{\pi}$ is the Bernstein constant (see \cite{Bernstein1,Trefethen1}). As a comparison, $e_{f}^{(\alpha,\beta)}(n,x)$ shares the same order of convergence $\mathcal{O}(n^{-1})$ at $x=a$, but somehow worse in magnitude than that of the best polynomial approximation. Nevertheless, superconvergence appears when $x\in(-1,a)\cup(a,1)$, where it follows from Theorem \ref{Thm13} that $e_{f}^{(\alpha,\beta)}(n,x)=\mathcal{O}(n^{-2})$ (also see   Fig.\! \ref{figure12}).

Incidentally, from the viewpoint of the maximum norm (i.e., the worst-case behavior of $e_{f}^{(\alpha,\beta)}(n,x)$), the Jacobi truncation $S_{n}^{(\alpha,\beta)}[f](x)$ ($\alpha,\beta\leq\frac{1}{2}$) performs as excellent as the best polynomial approximation in the sense of asymptotic rate when $n\to\infty$ for functions defined in (\ref{gpfun}) and (\ref{gafun}) where $\lambda>0$, that is,
\begin{equation}\label{maxer1}
 \|f-S_{n}^{(\alpha,\beta)}[f]\|_{\infty} =\begin{dcases}
 {\cal O}\big(n^{\max\{\alpha-\frac{1}{2},\beta-\frac{1}{2}\}-\lambda}\big),&  {\rm if}\; \max\{\alpha,\beta\}>\frac{1}{2}, \\[4pt]
{\cal O}\left(n^{-\lambda}\right),& {\rm if}\; \max\{\alpha,\beta\}\le\frac{1}{2}. \end{dcases}
\end{equation}

Taking the local behavior of $e_{f}^{(\alpha,\beta)}(n,x)$ around the boundaries $x=\pm1$ and singularity $x=a$ into consideration, we consider  a new weighted pointwise error function
\begin{equation}\label{weightedpointerror}
	\hat{e}_{f}^{(\alpha, \beta)}(n,x)=(1-x)^{\max\{\frac{\alpha}{2}+\frac{1}{4},0\}}(1+x)^{\max\{\frac{\beta}{2}+\frac{1}{4},0\}}(x-a)\,e_{f}^{(\alpha, \beta)}(n,x).
\end{equation}
Then we deduce from Theorem \ref{Thm13}  the uniform convergence order
\begin{equation}\label{weightedpointerrormax}
  \|\hat{e}_{f}^{(\alpha,\beta)}\|_{\infty}= {\cal O}(n^{-\lambda-1}),
\end{equation}
which also testifies to  the optimality of the estimates on $C(x)$ in (\ref{PerrJac1}). As a result, a global superconvergence is attained by $\hat{e}_{f}^{(\alpha,\beta)}(x)$, which gains one order higher in convergence rate than the best polynomial approximation (see Fig.\! \ref{figure13}).

\begin{figure}
\centerline{\includegraphics[height=5cm,width=16cm]{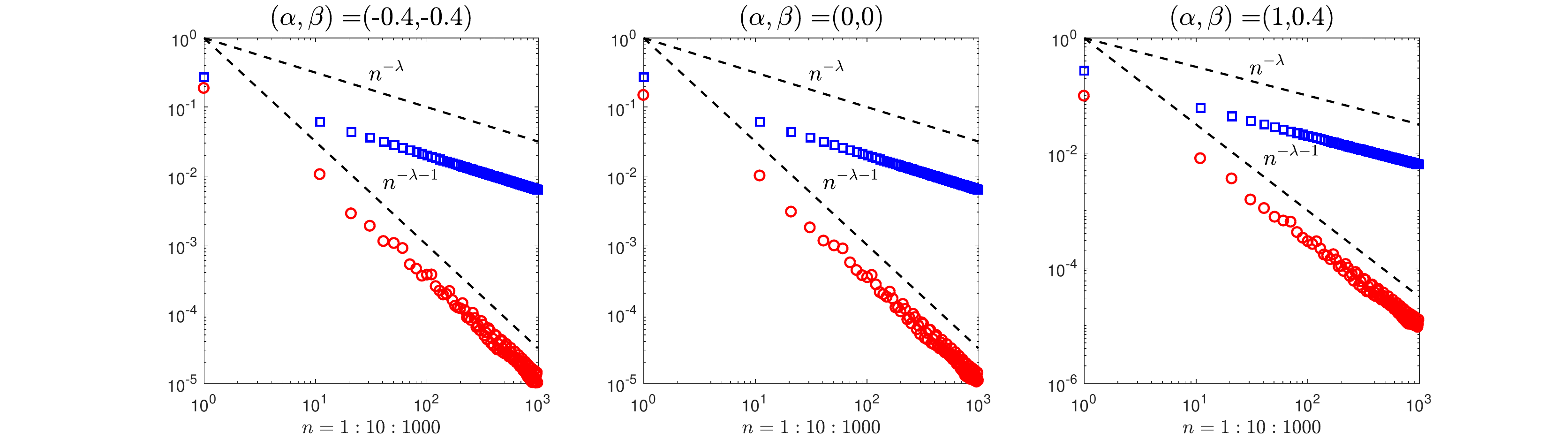}}
%\vspace*{-10pt}
\caption{\small Comparison of Jacobi expansion and the optimal polynomial approximation to $f(x)=(x-1/4)_{+}^{1/2}$ by the convergence rates of $\|\hat{e}_{f}^{(\alpha,\beta)}\|_{\infty}$  {\rm(}red{\rm)} versus $\|f-p_{n}^{*}\|_{\infty}$ {\rm(}blue{\rm)}.}\label{figure13}	
\end{figure}

The rest of this paper is largely devoted to the proof of the results stated in Theorem \ref{Thm13}. In Section \ref{sec2}, we present the pointwise error formula and some asymptotic results on the Jacobi  polynomials.
Applying the Hilb-type formula and van der Corput-type Lemma for Bessel transforms, in Section \ref{sec3}, we prove the optimal pointwise error estimates for the Jacobi truncation on functions defined in (\ref{gpfun})-(\ref{gafun}).  Errors in maximum norm and weighted maximum norm (superconvergence analysis) are considered in Section \ref{sec4}. Finally we extend the analysis to functions with boundary singularities and conclude the paper with some remarks  in Section \ref{sec5}.

\section{The reproducing kernel and pointwise error formula}\label{sec2}

Let ${\rm d}\omega(x)$ be a given distribution in  the Stieltjes sense.
 Assume that $\left\{p_k\right\}_{k=0}^{\infty}$ with ${\rm deg}(p_k)=k$ is the  set of orthonormal polynomials
associated with ${\rm d}\omega(x):$
$$
\int_{-1}^{1}p_j(x)\,p_k(x)\,{\rm d} \omega(x)=\delta_{jk},\quad j,k=0,1,2,\cdots,
$$
where $\delta_{jk}$ is the Kronecker Delta symbol.
In view of the Christoffel-Darboux formula, the reproducing kernel $K_n(x,y)$ is defined by
\begin{equation}\label{kernel}
K_n(x,y)=\sum_{k=0}^n p_k(x)p_k(y)= \frac{\kappa_n}{\kappa_{n+1}}\frac{p_{n+1}(x)p_n(y)-p_{n}(x)p_{n+1}(y)}{x-y},\quad \lim_{n\rightarrow\infty} \frac{\kappa_n}{\kappa_{n+1}}=\frac{1}{2}
\end{equation}
(see \cite[Theorem 3.22]{Szego} and Lubinsky \cite{Lubinsky}), where $\kappa_n$ is the leading coefficient of $p_n(x)$. It is easy to verify from  the orthogonality  that
\begin{equation}\label{kernel2}
\int_{-1}^1K_n(x,y)q(y) \,{\rm d}\omega(y)=q(x), \quad \forall\, q\in {\mathcal P}_n.
\end{equation}

We intend to estimate the pointwise error of the Jacobi orthogonal projection $e_f^{(\alpha,\beta)}(n,x)$ defined in  \eqref{sepAfun}.
According to  Szeg\"{o} \cite[(4.5.2)]{Szego} and Hesthaven, Gottlieb and Gottlieb \cite[Theorem 4.4]{Hesthaven}, the reproducing kernel of the Jacobi polynomials can be represented as follows
\begin{equation}\label{Jkernel}
\begin{split}
& K_n(x,y) ={\displaystyle\sum_{k=0}^n \frac{1}{ \sigma_{k}^{(\alpha,\beta)}}P_k^{(\alpha,\beta)}(x)P_k^{(\alpha,\beta)}(y)}
={\displaystyle \rho_n^{(\alpha, \beta)}
\frac{P_{n+1}^{(\alpha,\beta)}(x)P_n^{(\alpha,\beta)}(y)-P_n^{(\alpha,\beta)}(x)P_{n+1}^{(\alpha,\beta)}(y)}{x-y}},
\end{split}
\end{equation}
where
\begin{equation}\label{rhoformula}
 \rho_n^{(\alpha, \beta)}:=\frac{2^{-\alpha-\beta}}{2n+\alpha+\beta+2}\cdot\frac{\Gamma(n+2)\Gamma(n+\alpha+\beta+2)}{\Gamma(n+\alpha+1)\Gamma(n+\beta+1)}.
\end{equation}
From \eqref{kernel2} with $q(x)\equiv 1$ and \eqref{Jkernel}, we can derive  the following pointwise error formula, which  plays a fundamental role in the error analysis.
\smallskip
\begin{theorem}\label{thm1} Let  $f(x)$ be a suitably smooth function on $[-1,1].$ For every $x\in [-1,1],$  we denote the Jacobi expansion coefficients of the following quotient in $y$  by
\begin{equation}\label{galpbet}
 a_{n}^{(\alpha,\beta)}(x;g):=\frac{1}{\sigma_{k}^{(\alpha,\beta)}}\int_{-1}^{1}\!g(x,y) P_n^{(\alpha,\beta)}(y) \omega^{(\alpha,\beta)}(y)\mathrm{d}y,\quad  g(x,y):=\frac{f(x)-f(y)}{x-y}.
\end{equation}
Then the pointwise error of  the Jacobi expansion has the compact representation
\begin{equation}\label{Jtruerr}
\begin{split}
f(x)- S_n^{(\alpha,\beta)}[f](x)
&=A_n^{(\alpha, \beta)} {a_{n}^{(\alpha,\beta)}(x;g)P_{n+1}^{(\alpha,\beta)}(x)}- B_n^{(\alpha, \beta)} {a_{n+1}^{(\alpha,\beta)}(x;g)P_n^{(\alpha,\beta)}(x)},
\end{split}
\end{equation}
where
\begin{equation}\label{haveA}
A_n^{(\alpha, \beta)}=\displaystyle \frac{2(n+1)(n+\alpha+\beta+1)}{(2n+\alpha+\beta+2)(2n+\alpha+\beta+1)},\;\;
B_n^{(\alpha, \beta)}=  \frac{2(n+\alpha+1)(n+\beta+1)}{(2n+\alpha+\beta+2)(2n+\alpha+\beta+3)}.
\end{equation}
\end{theorem}
\begin{proof}
From (\ref{eq:jacexpcoeffs0}) and (\ref{kernel})-(\ref{kernel2}), we find readily that
\[
\begin{split}
 S_n^{(\alpha,\beta)}[f](x)
={\displaystyle\sum_{k=0}^n\frac{P_k^{(\alpha,\beta)}(x)}{ \sigma_{k}^{(\alpha,\beta)}}\int_{-1}^1f(y)P_k^{(\alpha,\beta)}(y) \omega^{(\alpha,\beta)}(y) \mathrm{d}y}
={\displaystyle\int_{-1}^1K_n(x,y)f(y) \omega^{(\alpha,\beta)}(y) {\rm d} y},
\end{split}
\]
and
$$ f(x)=\int_{-1}^1K_n(x,y)f(x) \omega^{(\alpha,\beta)}(y){\rm d}y.$$
Thus, we obtain from \eqref{eq:jacexpcoeffs0} and  \eqref{Jkernel} that
\[
\begin{split}
f(x) &- S_n^{(\alpha,\beta)}[f](x)=\int_{-1}^1K_n(x,y)[f(x)-f(y)]\omega^{(\alpha,\beta)}(y) {\rm d}y\\
&=
{\displaystyle  \rho_n^{(\alpha,\beta)}\int_{-1}^1\frac{f(x)-f(y)}{x-y}\big[P_{n+1}^{(\alpha,\beta)}(x)P_n^{(\alpha,\beta)}(y)-P_n^{(\alpha,\beta)}(x)P_{n+1}^{(\alpha,\beta)}(y)\big]} \omega^{(\alpha,\beta)}(y) {\rm d}y\\
&=\rho_n^{(\alpha,\beta)} \big[\sigma_n^{(\alpha,\beta)}   a_{n}^{(\alpha,\beta)}(x;g)P_{n+1}^{(\alpha,\beta)}(x) -
\sigma_{n+1}^{(\alpha,\beta)}   a_{n+1}^{(\alpha,\beta)}(x;g)P_{n}^{(\alpha,\beta)}(x)\big].
\end{split}
\]
Then, the identity  (\ref{Jtruerr})  follows from  directly working out the constants $A_n^{(\alpha, \beta)}=\rho_n^{(\alpha,\beta)}\sigma_n^{(\alpha,\beta)}$ and $B_n^{(\alpha, \beta)}=\rho_n^{(\alpha,\beta)}\sigma_{n+1}^{(\alpha,\beta)}$
by using  \eqref{eq:jacexpcoeffs0} and
\eqref{rhoformula}.
\end{proof}

\smallskip

Now, we take $f(x)$ in the above to be the generalised $\Phi$-function (\ref{gpfun}),
and obtain from \eqref{galpbet} that
\begin{equation}\label{Jcoe}
    a_{n}^{(\alpha,\beta)}(x;g)= \frac{1}{\sigma_{n}^{(\alpha,\beta)}}\cdot\begin{cases}
    	\displaystyle\int_{a}^{1}g_{1}(x,y)P_{n}^{(\alpha,\beta)}(y) \omega^{(\alpha,\beta)}(y)\mathrm{d}y,\quad &x<a,\\[11pt]
	    	\displaystyle\int_{a}^{1}g_{2}(y)P_{n}^{(\alpha,\beta)}(y) \omega^{(\alpha,\beta)}(y) \mathrm{d}y,\quad &x=a,\\[11pt]
     	    	\displaystyle\int_{-1}^{1}g_{3}(x,y)P_{n}^{(\alpha,\beta)}(y)\omega^{(\alpha,\beta)}(y) \mathrm{d}y,\quad &x>a,
    \end{cases}
\end{equation}
where $g=g_i, i=1,2,3$ are given by
\begin{equation}\label{Jcoe2}\begin{split}
	g_{1}(x,y)&=\frac{z(y)(y-a)^{\lambda}}{y-x}; \quad g_{2}(y)=g_1(a,y)=z(y)(y-a)^{\lambda-1},\quad\lambda>0;\\[4pt]
	g_{3}(x,y)&=\begin{cases}
		\displaystyle\frac{z(x)(x-a)^{\lambda}}{x-y},\quad &y\leq a,\\[11pt]
		\displaystyle\frac{z(x)(x-a)^{\lambda}-z(y)(y-a)^{\lambda}}{x-y},\quad &y>a.
	\end{cases}
\end{split}
\end{equation}

\smallskip
It is important to point out  that in  \eqref{Jtruerr}, the convergence rate  of $e_f^{(\alpha,\beta)}(n,x)$  depends on both $a_{n}^{(\alpha,\beta)}(x;g)$ and  $P_n^{(\alpha,\beta)}(x)$.
 Moreover,  the two terms  in the right hand side of (\ref{Jtruerr}) do not cancel each other for almost all $x$,
except for  the function (\ref{gpfun})  with $\lambda$ being even.
Accordingly, we can estimate   $a_{n}^{(\alpha,\beta)}(x;g)$ and  $P_n^{(\alpha,\beta)}(x)$ separately.
The roadmap for the pointwise error analysis is as follows.
\begin{itemize}
\item[(i)]  We shall bound $|P_n^{(\alpha,\beta)}(x)|$ pointwisely using Theorem \ref{Thm23} and Corollary \ref{Bineua}  below.
\item[(ii)]  We shall estimate $|a_{n}^{(\alpha,\beta)}(x;g)|$ by using the Hilb-type formula in Theorem \ref{Thm22} and the van der Corput-type  Lemma to be presented in  Section \ref{sec3}, together with the regularity analysis of the underlying function $g(x,y)$ in $y$ given in
\eqref{Jcoe2}.
\end{itemize}

We first recall that Darboux \cite{Darboux} and Szeg\"{o}  \cite[Theorem 8.21.12]{Szego} introduced the  following Hilb-type formula on the asymptotics of  $P_n^{(\alpha,\beta)}(x)$ in terms of  a highly oscillatory Bessel function.
 \smallskip
\begin{theorem}[see \cite{Darboux,Szego}]\label{Thm22}
For $\alpha,\beta>-1$ and  $n\gg 1,$ we have
  \begin{equation}\label{Hilb}
  \begin{split}
     \theta^{-\frac{1}{2}}\sin^{\alpha+\frac{1}{2}}&\Big(\frac{\theta}{2}\Big)\cos^{\beta+\frac{1}{2}}
     \Big(\frac{\theta}{2}\Big)P_n^{(\alpha,\beta)}(\cos\theta)\\
    &=\frac{\Gamma(n+\alpha+1)}{\sqrt{2}n!{\tilde N}^{\alpha}}J_{\alpha}({\tilde N}\theta)
     +\begin{dcases}
     \theta^{\frac{1}{2}}{\cal O}\big({\tilde N}^{-\frac{3}{2}}\big),& cn^{-1}\le \theta\le \pi-\epsilon,\\[4pt]
     \theta^{\alpha+2}{\cal O}\big({\tilde N}^{\alpha}\big),& 0< \theta\le cn^{-1},
     \end{dcases}
     \end{split}
  \end{equation}
  where ${\tilde N}=n+(\alpha+\beta+1)/2$, $c$ and $\epsilon$ are fixed positive numbers, and $J_{\alpha}(z)$ is the first kind of Bessel function of order $\alpha$. The constants in the ${\cal O}$-terms depend on $\alpha$, $\beta$, $c$, and $\epsilon$, but do not depend on $n.$
\end{theorem}
\smallskip

Using \cite[Theorem 7.32.2]{Szego} and $P_n^{(\alpha,\beta)}(-x)=(-1)^nP_n^{(\beta,\alpha)}(x)$, Muckenhoupt \cite{Muckenhoupt} derived the following pointwise bound.
\begin{theorem}[{see \cite[(2.6)-(2.7)]{Muckenhoupt}}] \label{Thm23} Let $\alpha,\beta>-1$ and $d$ be a fixed integer. Then for $n\ge \max\{0,-d\}$, we have
  \begin{equation}\label{Jasy1}
\big|P_{n+d}^{(\alpha,\beta)}(x)\big|\le CE_n^{(\alpha,\beta)}(x),
 \end{equation}
where $C$ is a positive  constant independent of $n$ and $x$, and
\begin{equation}\label{Jasy2}
  E_n^{(\alpha,\beta)}(x)=
  \begin{dcases}
    (n+1)^{\alpha},  & 1-(n+1)^{-2}\le x\le 1,\\[4pt]
    (n+1)^{-\frac{1}{2}}(1-x)^{-\frac{\alpha}{2}-\frac{1}{4}}, & 0\le x\le 1-(n+1)^{-2},\\[4pt]
    (n+1)^{-\frac{1}{2}}(1+x)^{-\frac{\beta}{2}-\frac{1}{4}},  & -1+(n+1)^{-2}\le x\le 0,\\[4pt]
    (n+1)^{\beta}, & -1\le x\le -1+(n+1)^{-2}.
  \end{dcases}
\end{equation}
\end{theorem}

\smallskip
As  a direct consequence of this theorem,  we have the following useful pointwise upper bound.
\begin{corollary}\label{Bineua} For $\alpha,\beta>-1$, we have
  \begin{equation}\label{Jasy}
\big|P_{n}^{(\alpha,\beta)}(x)\big|\le C_0 (n+1)^{-\frac{1}{2}}(1-x)^{-\max\{\frac{\alpha}{2}+\frac{1}{4},0\}}(1+x)^{-\max\{\frac{\beta}{2}+\frac{1}{4},0\}},\,\,\;\;  x\in [-1,1],
 \end{equation}
where  $C_0=2^{\max\{\frac{\alpha}{2}+\frac{1}{4},\frac{\beta}{2}+\frac{1}{4},0\}}C$ and $C$ is the same as in \eqref{Jasy1} with $d=0.$
\end{corollary}
\begin{proof}  It is evident that if  $-1<\alpha\le -1/2,$
then $\max\{\frac{\alpha}{2}+\frac{1}{4},0\}=0,$  and by \eqref{Jasy2},
\begin{equation*}%\label{Jasy5}
\begin{dcases}
  (n+1)^{\alpha} \le (n+1)^{-\frac{1}{2}}(1-x)^{-\max\{\frac{\alpha}{2}+\frac{1}{4},0\}}, & 1-(n+1)^{-2}\le x\le 1,\\[2pt]
  (n+1)^{-\frac{1}{2}}(1-x)^{-\frac{\alpha}{2}-\frac{1}{4}}\le (n+1)^{-\frac{1}{2}}(1-x)^{-\max\{\frac{\alpha}{2}+\frac{1}{4},0\}}, & 0\le x\le 1-(n+1)^{-2}.
  \end{dcases}
\end{equation*}
If  $\alpha>-1/2,$ then we have $\max\{\frac{\alpha}{2}+\frac{1}{4},0\}=\frac{\alpha}{2}+\frac{1}{4}>0,$ so apparently \eqref{Jasy2} is valid.
Therefore, for   $x\in [0,1],$ and  $\alpha,\beta>-1,$ we have
\begin{equation*}\label{Jasy5}
\begin{split}
 E_n^{(\alpha,\beta)}(x)& \le (n+1)^{-\frac{1}{2}}(1-x)^{-\max\{\frac{\alpha}{2}+\frac{1}{4},0\}}\\
 & \le  2^{\max\left\{\frac{\beta}{2}+\frac{1}{4},0\right\}}
  (n+1)^{-\frac{1}{2}}(1-x)^{-\max\{\frac{\alpha}{2}+\frac{1}{4},0\}}(1+x)^{-\max\{\frac{\beta}{2}+\frac{1}{4},0\}}.
  \end{split}
 \end{equation*}
Following the same lines as above,  we can show that for $x\in [-1,0]$ and  $\alpha,\beta>-1,$
 \begin{equation*}\label{Jasy50}
\begin{split}
 E_n^{(\alpha,\beta)}(x)& \le (n+1)^{-\frac{1}{2}}(1+x)^{-\max\{\frac{\beta}{2}+\frac{1}{4},0\}}\\
 & \le 2^{\max\left\{\frac{\alpha}{2}+\frac{1}{4},0\right\}}
  (n+1)^{-\frac{1}{2}}(1-x)^{-\max\{\frac{\alpha}{2}+\frac{1}{4},0\}}(1+x)^{-\max\{\frac{\beta}{2}+\frac{1}{4},0\}}.
  \end{split}
 \end{equation*}
This completes the proof.
  \end{proof}

\begin{remark}  In some special cases,  the constant $C$  in \eqref{Jasy} is explicitly known. Indeed, we find from
\cite[18.14.3]{Olver} that for $-\frac{1}{2}\le \alpha,\,\beta\le \frac{1}{2}$, we have
  \begin{equation*}
  \Big(\frac{1-x}{2}\Big)^{\frac{\alpha}{2}+\frac{1}{4}}\Big(\frac{1+x}{2}\Big)^{\frac{\beta}{2}+\frac{1}{4}}\big|P_{n}^{(\alpha,\beta)}(x)\big|\le
  \frac{\Gamma(\max(\alpha,\beta)+n+1)}{\pi^{\frac{1}{2}}n!\big(n+\frac{\alpha+\beta+1}{2}\big)^{\max(\alpha,\beta)+\frac{1}{2}}},\quad x\in[-1,1].
   \end{equation*}
Moreover, F\"{o}rster \cite{Forster} stated the bound for the Gegenbauer polynomials  with $\alpha\geq1$, that is,
 \begin{equation*}\label{Gegen}
  (1-x^{2})^{\frac{\alpha}{2}}\big|C_{n}^{(\alpha)}(x)\big|\leq\frac{ (2\alpha-1)\Gamma\left(\frac{n}{2}+\alpha\right)}{\Gamma(\alpha)\Gamma\left(\frac{n}{2}+1\right)},\quad x\in[-1,1],
  \end{equation*}
  where
  \begin{equation*}\label{}
    C_{n}^{(\alpha)}(x)=\frac{\Gamma(\alpha+\frac{1}{2})\Gamma(n+2\alpha)}{\Gamma(2\alpha)\Gamma(n+\alpha+\frac{1}{2})}P_n^{(\alpha-\frac{1}{2},\alpha-\frac{1}{2})}(x).
  \end{equation*}
\end{remark}

\begin{remark}
	The above bound of  $P_n^{(\alpha,\beta)}(x)$ % in Theorem \ref{Thm24}
	can precisely  characterise its behaviour  near the endpoints, which allows us to %  Based on \eqref{Jasy},
 describe   $C(1-\xi)$ and $C(-1+\xi)$ for  $\xi\in (0,\delta)$ in \eqref{PerrJac1}.
\end{remark}

\section{Pointwise error estimate for the Jacobi expansion of  the generalised $\Phi$-function}
 \label{sec3}

 This section is devoted to the asymptotic analysis of the Jacobi expansion coefficient $a_{n}^{(\alpha,\beta)}(x;g)$ of $g(x,y)$  in \eqref{Jcoe}-\eqref{Jcoe2}. With this and the preparations in Section \ref{sec2}, we shall be able to prove the main results stated in Theorem \ref{Thm13}.

\subsection{Useful lemmas}
A critical tool for the analysis is the following asymptotic formulas involving a highly oscillatory Bessel functions related to the Jacobi polynomials in \eqref{Hilb},
 which extend the classical van der Corput Lemma on
the Fourier transform \cite[pp. 332-334]{Stein} to the Bessel transform. They are  therefore dubbed as the generalized van der Corput-type Lemmas.

Let  $\Omega=(a, b) \subset \mathbb{R}$  be a finite open interval. Denote by   ${\rm  AC}(\bar \Omega)$  the space of
absolutely continuous functions on $\bar \Omega$.  We further introduce the space
$$W_{\rm AC}(\Omega)=\big\{\psi\,:\,\psi\in {\rm AC}(\bar \Omega),\;\; \psi'\in L^1(\Omega)\big\},$$
equipped with the norm
\begin{equation*}\label{}
  \|\psi\|_{W_{\rm AC}(\Omega)}=\|\psi\|_{L^{\infty}(\Omega)}+\|\psi^{\prime}\|_{L^{1}(\Omega)}.
\end{equation*}
 Indeed, according to Stein and Shakarchi \cite[pp.\! 130]{Stein2003} and Tao \cite[pp.\! 143-145]{Tao}, we have the integral representation for any $\psi \in W_{\rm AC}(\Omega):$
\begin{equation*}\label{}
\psi(x)=\psi(a)+\int_a^x\psi'(t){\rm d}t,
\end{equation*}
and any continuous function of bounded variation on $\Omega$ which maps each set of measure zero into a set of measure zero is also absolutely continuous.

It is also noteworthy that the AC-type space with different regularity on the highest derivative, e.g., BV (functions of bounded variation) has been used in Trefethen  \cite{Trefethen1} in the context of Chebyshev polynomial approximation of singular functions.

\begin{lemma}\label{vdc}
  Given $\alpha+\nu>-1$ and $\beta>-1$, the following asymptotic estimates hold for $\omega\gg 1.$
  \begin{itemize}
    \item[{\rm (i)}] For $\psi(x)\in W_{\rm AC}(0,b)$, we have
    \begin{equation}\label{vdczero}
      \int_{0}^{b}\!x^{\alpha}(b-x)^{\beta}J_{\nu}(\omega x)\psi(x)\,\mathrm{d}x=\|\psi\|_{W_{\rm AC}(0,b)}\cdot\mathcal{O}\big(\omega^{-\min\{\alpha+1,\beta+\frac{3}{2},\frac{3}{2}\}}\big).
    \end{equation}
 \item[{\rm (ii)}]  For $b>c>0$ and $\psi(x)\in W_{\rm AC}(c,b)$, we have
    \begin{equation}\label{vdcnonzero}
       \int_{c}^{b}\!(b-x)^{\beta}J_{\nu}(\omega x)\psi(x)\,\mathrm{d}x=\|\psi\|_{W_{\rm AC}(c,b)}\cdot\mathcal{O}\big(\omega^{-\min\{\beta+\frac{3}{2},\frac{3}{2}\}}\big).
    \end{equation}
    \end{itemize}
Here, the constant in the Big $\mathcal O$ is independent of $\omega$ and $\psi$.
\end{lemma}
 \begin{proof}
  The estimate \eqref{vdczero} is a special case of  \cite[Lemma 2.5]{Xiang2020}.  Now we show the improved estimate
  (\ref{vdcnonzero}) on the closed subinterval $[c,b]\subset [0,b]$. From the asymptotic property  of the Bessel function  \cite[pp. 362]{Abram}:
  \begin{equation*}
    J_{\nu}(z)=\sqrt{\frac{2}{\pi z}}\cos(z-\nu\pi/2-\pi/4)+\mathcal{O}\big(z^{-\frac{3}{2}}\big),\quad z\to\infty,
  \end{equation*}
  we have by using $\cos(\omega x-\nu\pi/2-\pi/4)=\cos(-\omega x+\nu\pi/2+\pi/4)$ that
  \begin{equation}\label{vdcpf1}
    \begin{aligned}
    \int_{c}^{b} &(b-x)^{\beta}\psi(x)J_{\nu}(\omega x)\,\mathrm{d}x\\
    & =\sqrt{\frac{2}{\pi \omega}}\int_{c}^{b}\!(b-x)^{\beta}
      \cos(\omega x-\nu\pi/2-\pi/4)x^{-\frac{1}{2}}\psi(x)\,\mathrm{d}x+\mathcal{O}(\omega^{-\frac{3}{2}})\\
    & =\sqrt{\frac{2}{\pi\omega}} \,\Re\Big\{e^{{\rm i} (\nu\pi/2+\pi/4)}\int_{c}^{b}\!(b-x)^{\beta}
       e^{-{\rm i}\omega x}x^{-\frac{1}{2}}\psi(x)\,\mathrm{d}x\Big\}+\mathcal{O}(\omega^{-\frac{3}{2}})\\
    & =\sqrt{\frac{2}{\pi\omega}}\,  \Re\Big\{e^{{\rm i} (\nu\pi/2+\pi/4-b\omega)}\int_{0}^{b-c}\! u^{\beta}
     e^{{\rm i}\omega u}(b-u)^{-\frac{1}{2}}\psi(b-u)\,\mathrm{d}u\Big\}+\mathcal{O}(\omega^{-\frac{3}{2}}),\\
     \end{aligned}
  \end{equation}
  where $\Re\{z\}$ denotes the real part of $z$. Setting $F(x)=\int_{0}^{x}\!u^{\beta}e^{{\rm i}\omega u}\,\mathrm{d}u$ and applying the integration by parts, we obtain
  \begin{equation}\label{vdcpf2}
    \begin{aligned}
    &\Big|\int_{0}^{b-c}\!u^{\beta}e^{{\rm i}\omega u}(b-u)^{-\frac{1}{2}}\psi(b-u)\,\mathrm{d}u\Big|=\Big| \int_{0}^{b-c}\!(b-u)^{-\frac{1}{2}}\psi(b-u) {\rm d} F(u) \Big|\\
    &\;\; =\Big|\Big[(b-u)^{-\frac{1}{2}}\psi(b-u)F(u)\Big]_{0}^{b-c}-\int_{0}^{b-c}\!F(u)\,\big((b-u)^{-\frac{1}{2}}\psi(b-u)\big)' \mathrm{d} u\Big|\\
    &\;\; \leq \Big(\frac{|\psi(c)|}{\sqrt{c}}+\int_{c}^{b}\big|\big(x^{-\frac{1}{2}}\psi(x)\big)^{\prime}\big|\,\mathrm{d}x\Big) \max_{u\in[0,b-c]}|F(u)|
    \\
    &\;\; \leq C\|\psi\|_{W_{\rm AC}(c,b)}\max_{x\in[0,b-c]}|F(x)|,
    \end{aligned}
  \end{equation}
  where $C$ is a constant independent of $\omega$. Finally, we use an asymptotic behavior of $F(x)$ in terms of the hypergeometric function ${}_{1}{\rm F}_{1}(\cdot)$ in \cite[(13.5.1)]{Abram}  to claim that
  \begin{equation*}\label{}
    |F(x)|=\frac{x^{\beta+1}}{\beta+1}\big|{}_{1}{\rm F}_{1}(\beta+1;\beta+2;{\rm i}\omega x)\big|=\mathcal{O}\left(\omega^{-1}+\omega^{-\beta-1}\right),
  \end{equation*}
  which, together with (\ref{vdcpf1}) and (\ref{vdcpf2}), leads to  (\ref{vdcnonzero}).
\end{proof}

With Lemma \ref{vdc} at our disposal, we now associate the Bessel function in \eqref{vdczero}-\eqref{vdcnonzero} with the Jacobi polynomial
through the  Hilb-type formula  \eqref{Hilb} and derive the  asymptotic estimates in Lemma  \ref{van der Corput Jacobi}  and Lemma \ref{JEC} below.
These allow us to deal with  the  integrals involved the Jacobi polynomials in $a_{n}^{(\alpha,\beta)}(x;g)$.
\begin{lemma}\label{van der Corput Jacobi} Let  $\alpha,\beta,\gamma,\delta>-1.$
If $\psi(x)\in W_{\rm AC}(a,1)$ with $a\in (-1,1)$, then for $n\gg 1,$ we have
   \begin{equation}\label{vdcorputJaca}
        \int_{a}^{1}\!(x-a)^{\gamma}(1-x)^{\delta}P_{n}^{(\alpha,\beta)}(x)\psi(x)\,\mathrm{d}x=\|\psi\|_{W_{\rm AC}(a,1)} \cdot \mathcal{O}\big(n^{-\min\{2\delta-\alpha+2,\gamma+\frac{3}{2},\frac{3}{2}\}}\big).
     \end{equation}
If $\psi(x)\in W_{\rm AC}(a,b)$  with  $-1<a<b<1$, then  for $n\gg 1,$ we have
    \begin{equation}\label{vdcorputJacb}
      \int_{a}^{b}\!(x-a)^{\gamma}P_{n}^{(\alpha,\beta)}(x)\psi(x)\,\mathrm{d}x=\|\psi\|_{W_{\rm AC}(a,b)} \cdot \mathcal{O}\big(n^{-\min\{\gamma+\frac{3}{2},\frac{3}{2}\}}\big).
     \end{equation}
\end{lemma}
\begin{proof} We make a change of variable $x=\cos\theta$  and denote $\theta_{0}=\arccos{a}.$ Then it follows from
 the Hilb-type formula (\ref{Hilb}) that
  \begin{equation}\label{Ja1}
    \begin{aligned}
    &\int_{a}^{1}\!(x-a)^{\gamma}(1-x)^{\delta}P_{n}^{(\alpha,\beta)}(x)\psi(x)\,\mathrm{d}x \\
    &\;\; =\int_{0}^{\theta_{0}}\!2^{\delta+1}\sin^{2\delta+1}\!\!\Big(\frac{\theta}{2}\Big)\cos\!\Big(\frac{\theta}{2}\Big)(\cos\theta-\cos\theta_{0})^{\gamma}P_{n}^{(\alpha,\beta)}(\cos\theta)\psi(\cos\theta)\,\mathrm{d}\theta\\
    &\;\;=\frac{\Gamma(n+\alpha+1)}{n!\tilde{N}^{\alpha}}\int_{0}^{\theta_{0}}\!\theta^{2\delta+1-\alpha}(\theta_{0}-\theta)^{\gamma}J_{\alpha}(\tilde{N}\theta)\Psi(\theta)\,\mathrm{d}\theta+\mathcal{O}(n^{-3/2}),
  \end{aligned}
  \end{equation}
  where $\Psi(\theta)=h(\theta)\psi(\cos\theta)$ and
  \begin{equation*}\label{Bel}
    h(\theta)= 2^{\alpha-\delta}\Big(\frac{\sin(\theta/2)}{\theta/2}\Big)^{2\delta-\alpha+1/2}\cos^{1/2-\beta}
    \Big(\frac{\theta}{2}\Big)\Big(\frac{\cos\theta-\cos\theta_{0}}{\theta_{0}-\theta}\Big)^{\gamma}.
  \end{equation*}
  One verifies readily that  $\Psi(\theta)$ is absolutely continuous on $[0,\theta_{0}]$ and $\Psi^{\prime}(\theta)\in L^{1}(0,\theta_{0})$.
  Then using \eqref{vdczero} in  Lemma \ref{vdc}  and the asymptotic property of the Gamma function (see \cite{Abram}):
  \begin{equation*}
    \lim_{n\to\infty}\frac{\Gamma(n+\alpha+1)}{n!\tilde{N}^{\alpha}}=1,
  \end{equation*}
   we obtain from \eqref{Ja1} that
  \begin{equation}\label{Phasmp}
  \int_{a}^{1}\!(x-a)^{\gamma}(1-x)^{\delta}P_{n}^{(\alpha,\beta)}(x)\psi(x)\,\mathrm{d}x=\|\Psi\|_{{W}_{\rm AC}(0,\theta_{0})} \cdot \mathcal{O}\big(n^{-\min\{2\delta-\alpha+2,\gamma+\frac{3}{2},\frac{3}{2}\}}\big).
  \end{equation}
 As
\begin{equation*}
 |\Psi(\theta)|\le C_0 |\psi(\cos\theta)|,\quad  |\Psi'(\theta)|\le C_1(|\psi(\cos\theta)|+|\psi'(\cos\theta)|\sin\theta),\quad \forall\, \theta\in [0,\theta_{0}],
  \end{equation*}
 for some constants $C_0, C_1$ independent of $\theta$, we derive from direct calculation that
\begin{equation}\label{WAC}
\begin{aligned}
\|\Psi\|_{{W}^{\rm AC}(0,\theta_{0})}&\le C\Big(\|\psi(\cos\theta)\|_{L^\infty(0,\theta_0)}+\int_0^{\theta_0}|\psi(\cos\theta)|{\rm d}\theta+\int_0^{\theta_0}|\psi'(\cos\theta)|\sin\theta {\rm d}\theta\Big)\\
&=C\Big(\|\psi\|_{L^\infty(a,1)}+\int_a^{1}|\psi(x)|\frac{{\rm d}x}{\sqrt{1-x^2}}+\int_a^{1}|\psi'(x)| {\rm d}x\Big)\\
&\le C\Big((\pi+1)\|\psi\|_{L^\infty(a,1)}+\int_a^{1}|\psi'(x)| {\rm d}x\Big)\\
& \le C(\pi+1)\|\psi\|_{W_{\rm AC}(a,1)}.
 \end{aligned}\end{equation}
Thus we claim \eqref{vdcorputJaca} from  \eqref{Phasmp} and \eqref{WAC}.

  Next for fixed $b<1$, we set $\theta_1=\arccos{b}.$ Using  \eqref{vdcnonzero} in  Lemma \ref{vdc} on $[\theta_1,\theta_0]$, we derive \eqref{vdcorputJacb} directly
   from the second identity in \eqref{Ja1} with $\delta=0$ and  $b=\cos \theta_1$, since %we have
  \begin{equation*}
    \int_{a}^{b}\!(x-a)^{\gamma}P_{n}^{(\alpha,\beta)}(x)\psi(x)\,\mathrm{d}x
    =\frac{\Gamma(n+\alpha+1)}{n!\tilde{N}^{\alpha}}\int_{\theta_1}^{\theta_{0}}\!(\theta_{0}-\theta)^{\gamma}J_{\alpha}(\tilde{N}\theta)\bar{\Psi}(\theta)
    \,\mathrm{d}\theta+\mathcal{O}(n^{-3/2}),
  \end{equation*}
  where
  \begin{equation*}
    \bar{\Psi}(\theta)=\sqrt{2\theta}\sin^{\frac{1}{2}-\alpha}\!\Big(\frac{\theta}{2}\Big)\cos^{\frac{1}{2}-\beta}\!\Big(\frac{\theta}{2}\Big)
    \Big(\frac{\cos\theta-\cos\theta_{0}}{\theta_{0}-\theta}\Big)^{\gamma}\psi(\cos\theta)\in W_{\rm AC}(\theta_1,\theta_{0}).
  \end{equation*}
Thus, following the same lines as the above for \eqref{WAC}, we can obtain \eqref{vdcorputJacb}.
\end{proof}

If $\psi(x)$ has more regularity, we denote $W_{\rm AC}^{m}(\Omega)$  for some positive integer $m$  as
\begin{equation*}
	W_{\rm AC}^{m}(\Omega)=\big\{\psi\,:\psi^{(k)}\in W_{\rm AC}(\Omega),\; k=0,\cdots, m\big\},
\end{equation*}
equipped with the norm
\begin{equation*}
	\|\psi\|_{W_{\rm AC}^{m}(\Omega)}=\sum_{k=0}^{m}\|\psi^{(k)}\|_{W_{\rm AC}(\Omega)}.
\end{equation*}
In particular, for $m=0,$ we have $W_{\rm AC}(\Omega)=W_{\rm AC}^{0}(\Omega).$
Then we can further derive the following estimate using Lemma \ref{van der Corput Jacobi}.

\begin{lemma}\label{JEC} Let $\alpha,\beta,\gamma>-1, a\in(-1,1)$ and $n\gg 1$. Denote by  $m=\lfloor\gamma\rfloor$  the greatest integer that is less than $\gamma$. If $\psi\in W_{\rm AC}^{m+1}(a,\frac{1+a}{2})\cap W_{\rm AC}^{m+2}(\frac{1+a}{2},1)$, then we have
  \begin{subequations}\label{JC2}
  \begin{equation}\label{JC2-1}
  \begin{split}
    &\int_{a}^{1}\!(x-a)^{\gamma}\psi(x) P_{n}^{(\alpha,\beta)}(x)\omega^{(\alpha,\beta)}(x)\,\mathrm{d}x\\
    &\qquad=\big(\|\psi\|_{W^{m+1}_{\rm AC}(a,\frac{1+a}{2})}+\|\psi\|_{W^{m+2}_{\rm AC}(\frac{1+a}{2},1)}\big)\cdot \mathcal{O}(n^{-\gamma-\frac32}).
    \end{split}
  \end{equation}
  If $\psi\in W_{\rm AC}^{m+1}(\frac{-1+a}{2},a)\cap W^{m+2}_{\rm AC}(-1,\frac{-1+a}{2})$, then
  \begin{equation}\label{JC2-2}
  \begin{split}
    &\int_{-1}^{a}\!(a-x)^{\gamma}\psi(x)P_{n}^{(\alpha,\beta)}(x)\omega^{(\alpha,\beta)}(x)\,\mathrm{d}x\\
    &\qquad=\big(\|\psi\|_{W_{\rm AC}^{m+1}(\frac{-1+a}{2},a)}+\|\psi\|_{W_{\rm AC}^{m+2}(-1,\frac{-1+a}{2})}\big)\cdot \mathcal{O}(n^{-\gamma-\frac32}).
  \end{split}
  \end{equation}
  \end{subequations}
 \end{lemma}
\begin{proof}
Recall the Rodrigues' formula of Jacobi polynomials (see \cite[pp.\! 94]{Szego}):
\begin{equation}\label{Rodrigues}
  \omega^{(\alpha,\beta)}(y)P_{n}^{(\alpha,\beta)}(y)=\frac{(-1)^{k}}{2^{k}(n)_{k}}\frac{\mathrm{d}^{k}}{\mathrm{d}y^{k}}
  \big(\omega^{(\alpha+k,\beta+k)}(y)P_{n-k}^{(\alpha+k,\beta+k)}(y)\big),\quad k\in\mathbb{N},
\end{equation}
where $(n)_{k}=n(n-1)\cdots(n-k+1)$ denotes the falling factorial.  Using \eqref{Rodrigues} and integration by parts, we
find from Lemma \ref{van der Corput Jacobi}  and the fact $P_{n}^{(\alpha,\beta)}(\frac {1+a} 2)=\mathcal{O}(n^{-\frac12})$
(see Theorem \ref{Thm23}) that
\begin{equation}\label{la3p}
\begin{split}
&\int_{a}^{1}\!(x-a)^{\gamma}\psi(x)P_{n}^{(\alpha,\beta)}(x)\omega^{(\alpha,\beta)}(x)\,\mathrm{d}x\\	
&\quad  =\frac{1}{2^{m+1}(n)_{m+1}}\int_{a}^{1}
\left[(x-a)^{\gamma}\psi(x)\right]^{(m+1)}P_{n-m-1}^{(\alpha+m+1,\beta+m+1)}(x)\omega^{(\alpha+m+1,\beta+m+1)}(x)\,\mathrm{d}x\\
& \quad =\frac{1}{2^{m+1}(n)_{m+1}}\int_{a}^{\frac{1+a}{2}}
\left[(x-a)^{\gamma}\psi(x)\right]^{(m+1)}P_{n-m-1}^{(\alpha+m+1,\beta+m+1)}(x)\omega^{(\alpha+m+1,\beta+m+1)}(x)\,\mathrm{d}x\\
&\quad\;\;  -\frac{1}{2^{m+2}(n)_{m+2}}\left[(x-a)^{\gamma}\psi(x)\right]^{(m+1)}P_{n-m-2}^{(\alpha+m+2,\beta+m+2)}(x)\omega^{(\alpha+m+2,\beta+m+2)}(x)\big|_{\frac{1+a}{2}}^1\\
&\quad\;\;  +\frac{1}{2^{m+2}(n)_{m+2}}\int_{\frac{1+a}{2}}^1\!
\left[(x-a)^{\gamma}\psi(x)\right]^{(m+2)}P_{n-m-2}^{(\alpha+m+2,\beta+m+2)}(x)
\omega^{(\alpha+m+2,\beta+m+2)}(x)\,\mathrm{d}x\\
&\quad  = \|\psi\|_{W_{\rm AC}^{m+1}(a,\frac{1+a}{2})} \cdot \mathcal{O}(n^{-\gamma-\frac32})+\left[(x-a)^{\gamma}\psi(x)\right]^{(m+1)}\big|_{x=\frac{1+a}2}\cdot  \mathcal{O}(n^{-m-\frac{5}{2}})
\\&\qquad
+\|\psi\|_{W_{\rm AC}^{m+2}(\frac{1+a}{2},1)}\cdot \mathcal{O}(n^{-m-\frac{7}{2}}).
\end{split}
 \end{equation}
This leads to (\ref{JC2-1})  by the fact that $\gamma+\frac{3}{2}\le  m+\frac 5 2$.

We can obtain  (\ref{JC2-2})  directly using $P_{n}^{(\alpha,\beta)}(-x)=(-1)^{n}P_{n}^{(\beta,\alpha)}(x)$.
\end{proof}

\subsection{Analysis of $a_{n}^{(\alpha,\beta)}(x;g)$}
We now turn to the optimal asymptotic estimates of $a_{n}^{(\alpha,\beta)}(x;g)$ in \eqref{Jcoe}-\eqref{Jcoe2}, which is of paramount importance in Theorem \ref{Thm13}.
\begin{theorem}\label{Thm32}
	For $\alpha,\beta >-1$  and $a_{n}^{(\alpha,\beta)}(x;g)$ defined in \eqref{Jcoe}, we have
	\begin{equation}\label{coe1}
        a_{n}^{(\alpha,\beta)}(x;g)
        =\begin{dcases}
        |x-a|^{-1}{\cal O}(n^{-\lambda-\frac12}),&x\in [-1,a)\cup (a,1], \;\; \lambda>-1, \\[6pt]
        {\cal O}(n^{-\lambda+\frac12}), &x=a, \;\; \lambda>0,
        \end{dcases}
\end{equation}
and
\begin{equation}\label{coe11}
 a_{n}^{(\alpha,\beta)}(x;g)={\cal O}(n^{-\lambda+\frac{1}{2}}) \quad\mbox{for $\lambda>0$ and $\forall\,x\in [-1,1]$},
\end{equation}
where the constants in $\mathcal{O}$-terms are independent of $x$.
\end{theorem}
\begin{proof}

For clarity, we carry out the proof in three cases: (i) $x<a$; (ii) $x=a$ and (iii) $x>a$. Below
we just provide the detailed  proof for the cases (i) and (ii), but sketch that of the third case in Appendix \ref{AppendixA}
to avoid unnecessary repetition.

{(i) $x<a$}: For each fixed $x<a$,  one verifies readily from  \eqref{JC2} and $\sigma_{n}^{(\alpha,\beta)}=\mathcal{O}(n^{-1})$ that
\begin{equation*}
\begin{aligned}
a_{n}^{(\alpha,\beta)}(x;g)&=\frac{1}{\sigma_{n}^{(\alpha,\beta)}}\int_{a}^{1}\!(y-a)^{\lambda}\frac{z(y)}{y-x}P_{n}^{(\alpha,\beta)}(y)
\omega^{(\alpha,\beta)}(y)\,\mathrm{d}y=\Big\|\frac{z(y)}{y-x}\Big\|_{W^{m+2}_{\rm AC}(a,1)}\cdot\mathcal{O}\big(n^{-\lambda-\frac12}\big),
\end{aligned}
\end{equation*}
where  $m=\lfloor \lambda\rfloor$ is defined as that in Lemma \ref{JEC}.  Although the asymptotic order $n^{-\lambda-\frac{1}{2}}$ in above is optimal, but the constant before it is not well controlled when $x$ is closed to $a$. To obtain \eqref{coe1} and \eqref{coe11}, for simplicity, we redefine $m$ as  $m=\lambda-1$ if $\lambda$ is an integer, otherwise $m=\lfloor \lambda\rfloor$.

Applying the Rodrigues' formula (\ref{Rodrigues}), we obtain
\begin{equation}\label{case1}\begin{aligned}
  a_{n}^{(\alpha,\beta)}(x;g)=\frac{1}{2^{m+1}(n)_{m+1}\sigma_{n}^{(\alpha,\beta)}}\int_{a}^{1}\!&(y-a)^{\lambda-m-1}\frac{\phi_{m+1}(x,y)}{y-x}\\
  &\times P_{n-m-1}^{(\alpha+m+1,\beta+m+1)}(y)\omega^{(\alpha+m+1,\beta+m+1)}(y)\,\mathrm{d}y,
  \end{aligned}
\end{equation}
and according to the Leibniz's rule, we have
\begin{equation}\label{g1}
\begin{aligned}
  &\partial_{y}^{m+1}g_{1}(x,y)=(y-a)^{\lambda-m-1}\frac{\phi_{m+1}(x,y)}{y-x}\\
  \phi_{m+1}(x,y)=\sum_{k=0}^{m+1}\sum_{j=0}^{k}&\frac{(-1)^{k-j}(m+1)!(\lambda)_{j}}{j!(m+1-k)!}z^{(m+1-k)}(y)(y-a)^{m+1-k}\Big(\frac{y-a}{y-x}\Big)^{k-j}.
\end{aligned}
\end{equation}
It is not difficult to verify that there exist two constants $C_{1}$ and $C_{2}$ independent of $x$ such that
\begin{equation*}
  \max_{y\in[a,\frac{1+a}{2}]}\left|\phi_{m+1}(x,y)\right|\leq C_{1},\quad\quad \max_{y\in[a,\frac{1+a}{2}]}\left|\partial_{y}\phi_{m+1}(x,y)\right|\leq \frac{C_{2}}{y-x}.
\end{equation*}
As a result, we can derive the estimate (\ref{coe1}) from Lemma \ref{JEC} since
$$\Big\|\frac{\phi_{m+1}}{y-x}\Big\|_{W_{\rm AC}(a,\frac{1+a}{2})}\leq C|x-a|^{-1}$$
for some constant $C$ independent of $x$, and $\frac{\phi_{m+1}(x,y)}{y-x}$ is smooth for $x<a$ and $\frac {a+1}{2}\le y\le 1$ (i.e. $\|\frac{\phi_{m+1}}{y-x}\|_{W^1_{\rm AC}(\frac{1+a}{2},1)}$ is uniformly bounded by a constant independent of $x$).

In order to obtain the uniformly estimate (\ref{coe11}) for $a_{n}^{(\alpha,\beta)}(x;g)$ for any $x\in[-1,a)$,
we conduct integration by parts till  $m$ instead of $m+1$  in \eqref{case1} and find
\begin{equation*}
\partial_{y}^{m} g_1(x, y)
= {(y-a)^{\lambda-m-1}} \Big(\frac{y-a}{y-x}\phi_m(x,y)\Big).
\end{equation*}
 Note that these two functions in $x$
 $$\frac{y-a}{y-x}\phi_m(x,y)\quad\text{and} \quad \int_a^1\Big|\partial_y\Big(\frac{y-a}{y-x}\phi_m(x,y)\Big)\Big|\,\mathrm{d}y$$
 are  uniformly bounded on $(x,y)\in[-1,a)\times(a,1)$, and $\|\frac{\phi_{m+1}}{y-x}\|_{W^1_{\rm AC}(\frac{1+a}{2},1)}$ is uniformly bounded  independent of $x$ too. Thus we obtain the uniform bound \eqref{coe11}.

\medskip

{(ii) $x=a$}: It  follows from  \eqref{JC2} directly that
\begin{equation*}\begin{aligned} a_{n}^{(\alpha,\beta)}(a;g)=\frac{1}{\sigma_{n}^{(\alpha,\beta)}}\int_{a}^{1}\!(y-a)^{\lambda-1}z(y)P_{n}^{(\alpha,\beta)}(y)\omega^{(\alpha,\beta)}(y)\,\mathrm{d}y
=\mathcal{O}(n^{-\lambda+\frac12}).
\end{aligned}\end{equation*}

{(iii) $x>a$}:\; See Appendix \ref{AppendixA} for a sketch.
\end{proof}

%%%%%%%%%%%%%%%%%%%%%%%%%%%%%%%%%%%%%%%%%%%%%%%%%%

\subsection{Proof of Theorem \ref{Thm13} for \eqref{gpfun}}

From (\ref{Jtruerr}) and  (\ref{coe1}),
 it follows that  for $ x\in [-1,a)\cup (a,1],$
\begin{equation}\label{re3}
\begin{split}
{\displaystyle e_f^{(\alpha,\beta)}(n,x)} &={\displaystyle \frac{2(n+1)(n+\alpha+\beta+1)}{(2n+\alpha+\beta+2)(2n+\alpha+\beta+1)}a_{n}^{(\alpha,\beta)}(x;g)P_{n+1}^{(\alpha,\beta)}(x)}\\
&\quad -{\displaystyle  \frac{2(n+\alpha+1)(n+\beta+1)}{(2n+\alpha+\beta+2)(2n+\alpha+\beta+3)}a_{n+1}^g(x,\alpha,\beta)P_n^{(\alpha,\beta)}(x)}\\
&={\displaystyle  |x-a|^{-1}{\displaystyle   {\cal O}(n^{-\lambda-\frac{1}{2}})}\big(|P_n^{(\alpha,\beta)}(x)|+|P_{n+1}^{(\alpha,\beta)}(x)|\big)},
\end{split}
 \end{equation}
while for $x=a,$
\begin{equation}\label{re4}
{\displaystyle e_f^{(\alpha,\beta)}(n,a)=\big(|P_n^{(\alpha,\beta)}(a)|+|P_{n+1}^{(\alpha,\beta)}(a)|\big) \, {\cal O}(n^{-\lambda+\frac{1}{2}}),
} \end{equation}
 which,
together with Theorem \ref{Thm23} on $P_k^{(\alpha,\beta)}(x)={\cal O}( n^{-\frac{1}{2}})$ ($k=n,n+1$) for
fixed $x\in (-1,1)$,
yields
$$e_f^{(\alpha,\beta)}(n,x)\le C(x)n^{-\lambda-1}\,\,(x\not=a),\quad e_f^{(\alpha,\beta)}(n,a)\le Cn^{-\lambda}.$$
Here,  $C$ and $C(x)$ are independent of $n$. Furthermore,  from Corollary \ref{Bineua}  on $P_n^{(\alpha,\beta)}(1\pm \xi)$,
we deduce that $C(x)$ behaves like  \eqref{PerrJac1} near $x=a,\pm 1$.
For $x=\pm 1$,   using the properties \cite[(7.32.2)]{Szego}:
$$P_n^{(\alpha,\beta)}(1)={\cal O}\big( n^{\alpha}\big),\quad P_n^{(\alpha,\beta)}(-1)={\cal O}\big( n^{\beta}\big),
$$
and \eqref{re3}, we obtain (\ref{PerrJac2}), i.e.,
$$e_f^{(\alpha,\beta)}(n,1)\le C n^{\alpha-\frac{1}{2}-\lambda},\quad e_f^{(\alpha,\beta)}(n,-1)\le C n^{\beta-\frac 1 2-\lambda},
$$
where  $C$ is a positive constant independent of $n$.

At this point, we have  completed the proof of Theorem \ref{Thm13}  for (\ref{gpfun}) except for the case where $x=a$ and $\lambda>0$ is an even integer.

 Indeed, if $\lambda$ is an even integer, we apply the Rodrigues formula \eqref{Rodrigues} $\lambda$ times, and use Lemma \ref{JEC}, leading to
	\begin{equation}\label{Intcase}
	\begin{split}
		a_{n}^{(\alpha,\beta)}(a;g)
		& =-\frac{g_{2}^{(\lambda-1)}(y)P_{n-\lambda}^{(\alpha+\lambda,\beta+\lambda)}(y)\omega^{(\alpha+\lambda,\beta+\lambda)}(y)}
{2^{\lambda}(n)_{\lambda}\sigma_{n}^{(\alpha,\beta)}}\bigg|_{a}^{1}\\
		&\quad +\frac{1}{2^{\lambda}(n)_{\lambda}\sigma_{n}^{(\alpha,\beta)}}\int_{a}^{1}\!g_{2}^{(\lambda)}(y)P_{n-\lambda}^{(\alpha+\lambda,\beta+\lambda)}(y)\omega^{(\alpha+\lambda,\beta+\lambda)}(y)\,\mathrm{d}y\\
		& =\frac{(\lambda-1)!(1-a)^{\alpha+\lambda}(1+a)^{\beta+\lambda}P_{n-\lambda}^{(\alpha+\lambda,\beta+\lambda)}(a)z(a)}{2^{\lambda}(n)_{\lambda}\sigma_{n}^{(\alpha,\beta)}}\\
		&\quad +\frac{1}{2^{\lambda}(n)_{\lambda}\sigma_{n}^{(\alpha,\beta)}}\int_{a}^{1}\!g_{2}^{(\lambda)}(y)
P_{n-\lambda}^{(\alpha+\lambda,\beta+\lambda)}(y)\omega^{(\alpha+\lambda,\beta+\lambda)}(y)\,\mathrm{d}y
\\
&=	\frac{(\lambda-1)!(1-a)^{\alpha+\lambda}(1+a)^{\beta+\lambda}P_{n-\lambda}^{(\alpha+\lambda,\beta+\lambda)}(a)z(a)}{2^{\lambda}(n)_{\lambda}\sigma_{n}^{(\alpha,\beta)}}
+\mathcal{O}(n^{-\lambda-\frac12}),
\end{split}
	\end{equation}
where we used $g_2$ is smooth for $y\in[a,1]$ and
\begin{equation*}	
\begin{split}	g_{2}^{(\lambda-1)}(a) & =\big[(y-a)^{\lambda-1}z(y)\big]^{(\lambda-1)}=\sum_{k=0}^{\lambda-1}\binom{\lambda-1}{k}(\lambda-1)_{k}(y-a)^{\lambda-1-k}z^{(\lambda-1-k)}(y)\Big|_{y=a}\\
& =(\lambda-1)!\,z(a).
\end{split}
	\end{equation*}
Then from (\ref{Intcase}) and (\ref{re3}) we have that   for $x=a$,
\begin{align}\label{even1}
&e_f^{(\alpha,\beta)}(n,a) \notag \\
&=A_n^{(\alpha,\beta)}\frac{(\lambda-1)! (1-a)^{\alpha+\lambda}(1+a)^{\beta+\lambda}z(a)}{2^{\lambda}}
\frac{P_{n-\lambda}^{(\alpha+\lambda,\beta+\lambda)}(a)P_{n+1}^{(\alpha,\beta)}(a)}{ (n)_{\lambda}\, \sigma_{n}^{(\alpha,\beta)}} \notag  \\
&\qquad -B_n^{(\alpha,\beta)}\frac{(\lambda-1)! (1-a)^{\alpha+\lambda}(1+a)^{\beta+\lambda}z(a)}{2^{\lambda}}
\frac{P_{n+1-\lambda}^{(\alpha+\lambda,\beta+\lambda)}(a)P_{n}^{(\alpha,\beta)}(a)}{ (n+1)_{\lambda}\, \sigma_{n+1}^{(\alpha,\beta)}}
+ {\cal O}(n^{-\lambda-1})\\
&=A_n^{(\alpha,\beta)}\frac{(\lambda-1)! (1-a)^{\alpha+\lambda}(1+a)^{\beta+\lambda}z(a)}{2^{\lambda}(n)_{\lambda}\, \sigma_{n}^{(\alpha,\beta)}}\Big[P_{n-\lambda}^{(\alpha+\lambda,\beta+\lambda)}(a)P_{n+1}^{(\alpha,\beta)}(a) \notag \\
&\qquad -P_{n+1-\lambda}^{(\alpha+\lambda,\beta+\lambda)}(a)P_{n}^{(\alpha,\beta)}(a)\Big]+ {\cal O}(n^{-\lambda-1}), \notag
\end{align}
 where we used the factor $P_{n+1-\lambda}^{(\alpha+\lambda,\beta+\lambda)}(a)P_{n}^{(\alpha,\beta)}(a)={\cal O}(n^{-1})$ and
 $$
B_n^{(\alpha,\beta)}=A_n^{(\alpha,\beta)}(1+{\cal O}(n^{-2})),\quad \frac{1}{ (n+1)_{\lambda}\, \sigma_{n+1}^{(\alpha,\beta)}}= \frac{1}{ (n)_{\lambda}\, \sigma_{n}^{(\alpha,\beta)}}\left(1+{\cal O}(n^{-1})\right).
$$
Moreover, we shall show that for even integer $\lambda>0,$
\begin{equation}\label{JacPoly}
P_{n-\lambda}^{(\alpha+\lambda,\beta+\lambda)}(a)P_{n+1}^{(\alpha,\beta)}(a)-P_{n+1-\lambda}^{(\alpha+\lambda,\beta+\lambda)}(a)P_{n}^{(\alpha,\beta)}(a)={\cal O}(n^{-2}),
 \end{equation}
which implies a cancellation happens and gains one order  higher in convergence rate. To this end, we use the asymptotic property in \cite[Theorem 8.21.8]{Szego}:
For real $\alpha$ and $\beta,$ we have
  \begin{equation}\label{Hilb2}
     P_n^{(\alpha,\beta)}(\cos\theta)
    =(n\pi)^{-\frac{1}{2}}\sin^{-\alpha-\frac{1}{2}}\Big(\frac{\theta}{2}\Big)\cos^{-\beta-\frac{1}{2}}\Big(\frac{\theta}{2}\Big)
     \cos({\tilde N}\theta+\gamma)
     +{\cal O}\big({n}^{-\frac{3}{2}}\big),\quad 0<\theta<\pi,
  \end{equation}
  where ${\tilde N}=n+(\alpha+\beta+1)/2$, $\gamma=-\frac{2\alpha+1}{4}\pi$. The bound for the error term holds uniformly in the interval
  $[\epsilon, \pi-\epsilon]$  for fixed positive number  $\epsilon$.
Thus by (\ref{Hilb2}),
\begin{equation}\label{JacPoly1}
\begin{split}
&\frac{P_{n-\lambda}^{(\alpha+\lambda,\beta+\lambda)}(a)P_{n+1}^{(\alpha,\beta)}(a)-P_{n+1-\lambda}^{(\alpha+\lambda,\beta+\lambda)}(a)P_{n}^{(\alpha,\beta)}(a)}  {(2\pi)^{-1} \sin^{-2\alpha-\lambda-1}\big(\frac{\theta_0}{2}\big)\cos^{-2\beta-\lambda-1}\big(\frac{\theta_0}{2}\big)} \\
& =
\frac{\cos(\theta_0+\lambda\pi/2)}{\sqrt{(n-\lambda)(n+1)}}-\frac{\cos(\theta_0+\lambda\pi/2)}{\sqrt{(n+1-\lambda)n}}
+\frac{\cos(\theta_1)}{\sqrt{(n-\lambda)(n+1)}}-\frac{\cos(\theta_1)}{\sqrt{(n+1-\lambda)n}}+{\cal O}\big({n}^{-2})
\end{split}
\end{equation}
with
$$\theta_1=(2n+2+\alpha+\beta)\theta_0-\frac{(2\alpha+\lambda+1)\pi}{2}.$$
Note  that if $\lambda$ is an even integer, the first four ${\cal O}(n^{-1})$ terms in \eqref{JacPoly1} can be cancelled. This yields
\eqref{JacPoly}.
Then we  derive from \eqref{even1} and \eqref{JacPoly1} that  for even integer  $\lambda,$
\begin{equation}\label{evencase}
{\displaystyle e_f^{(\alpha,\beta)}(n,a)= {\cal O}(n^{-\lambda-1})}.
 \end{equation}

%%%%%%%%%%%%%%%%%%%%%%%%%%%%%%%%%%%%%%%%%%
\subsection{Extension to (\ref{gafun})}
Let
\begin{equation}\label{invgpfun}
 f^*(x)=z(x)\cdot\begin{cases}
(a-x)^{\lambda},&-1\le x<a,\\[2pt]
0,&a<x\le 1,\end{cases} \quad a\in(-1,1),
\end{equation}
where  $\lambda>-1$, $z\in C^{\infty}[-1,1]$ and $f^*(a)=0$ for $\lambda>0$ and $ f^*(a)=\frac{z(a)}{2}$ for $\lambda=0$. Using an analogous argument as for (\ref{gpfun}), we can show that Theorem \ref{Thm13}  is also valid  for $f^*(x)$.
\smallskip
\begin{corollary} Given $f^*(x)$ in \eqref{invgpfun} and  $\alpha,\beta>-1$,  we  have the following pointwise error estimates.
\begin{itemize}
\item[{\rm(i)}]  For $x\in (-1,a)\cup (a,1)$, we have $e_{f^*}^{(\alpha,\beta)}(n,x)\le C(x)n^{-\lambda-1}$, where  $C(x)$ is independent of $n$ and has the behaviours near $x=a,\pm 1$ as follows
\begin{equation}\label{PerrJac11}
C(-1+ \xi)\le D(-1)\xi^{-\max\{\frac \beta2+\frac 14,0\}},\quad C(1- \xi)\le D(1)\xi^{-\max\{\frac \alpha2+\frac14,0\}},\quad  C(a\pm \xi)\le D(a)\xi^{-1},
\end{equation}
for $0<\xi\le \delta,$ where $D(\pm 1), D(a)>0$ and $\delta>0$ are independent of $n.$

\item[{\rm (ii)}]  At  $x=\pm 1$,  we have
\begin{equation}\label{PerrJac21}
 e_{f^*}^{(\alpha,\beta)}(n,1)\le C n^{-\lambda+\alpha-\frac{1}{2}};\quad e_{f^*}^{(\alpha,\beta)}(n,-1)\le C n^{-\lambda+\beta-\frac{1}{2}},
\end{equation}
where  $C$ is a positive constant independent of $n$.

\item[{\rm (iii)}] At $x=a$ and $ \lambda>0$,  we have
\begin{equation}\label{PerrJac31}
e_{f^*}^{(\alpha,\beta)}(n, a)\le\left\{\begin{aligned}
& Cn^{-\lambda-1},\quad &&\lambda\ \mathrm{is\ even,}\\
& Cn^{-\lambda},\quad &&\mathrm{otherwise,}
\end{aligned}\right.
\end{equation}
where  $C$ is a positive constant independent of $n$.
\end{itemize}
\end{corollary}

\smallskip

Notice that
$z(x)|x-a|^{\lambda}=f(x)+f^*(x)$, which implies Theorem \ref{Thm13} also holds for the function defined by (\ref{gafun}).

%%%%%%%%%%%%%%%%%%%%%%%%%%%%%%%%%%%%%%%%%%%%%%%%
\medskip
\begin{figure}[hpbt]
\centering
\includegraphics[height=7cm,width=14cm]{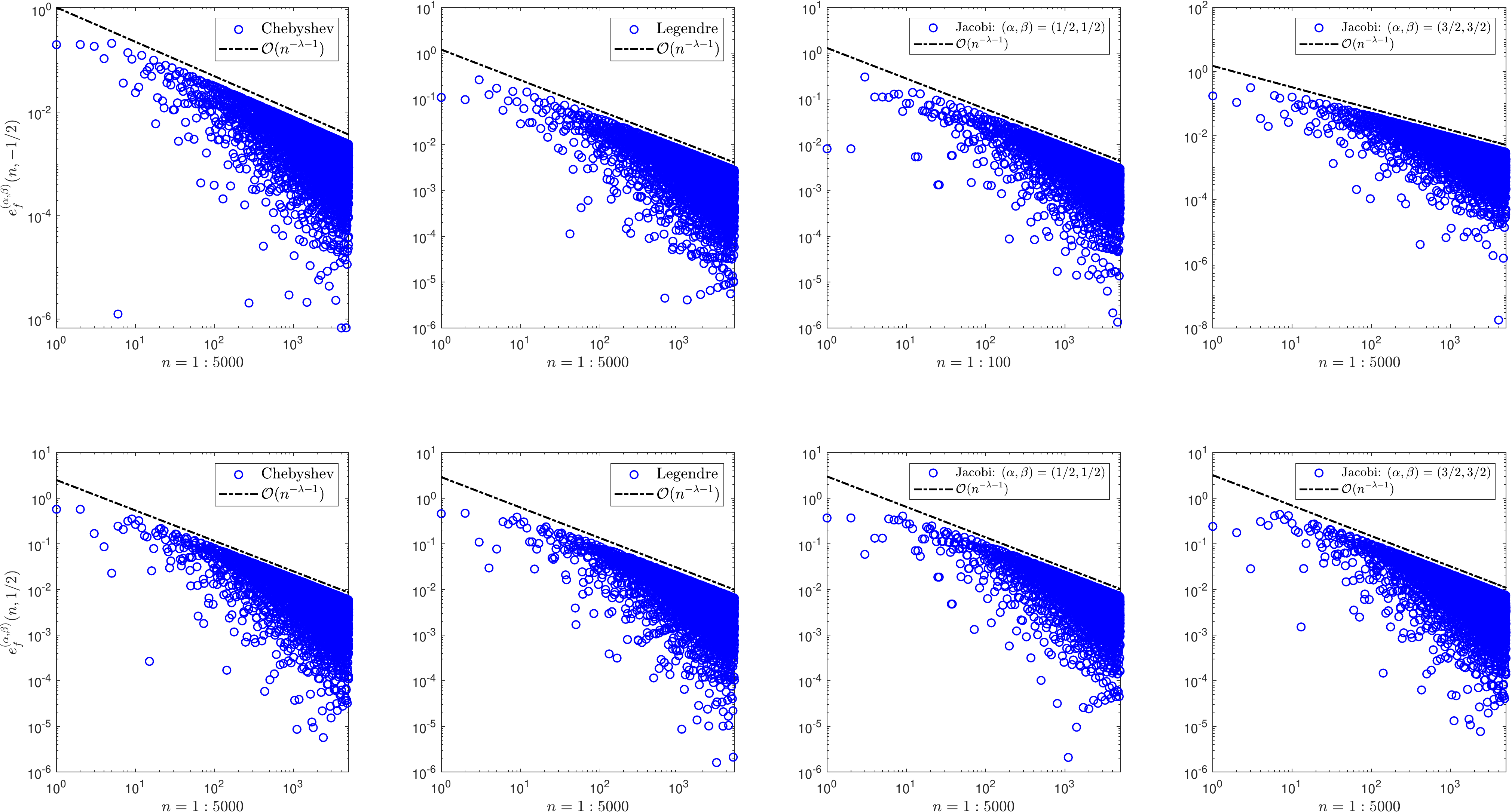}
%\vspace*{-10pt}
\caption{ Pointwise errors  $e_f^{(\alpha,\beta)}(n,x)$ for $f(x)=(x-\frac{1}{4})_+^{\lambda}$ with $\lambda=-\frac{1}{3}$ at $x=-\frac{1}{2}$ {\rm(}first row{\rm)} and $x=\frac{1}{2}$ {\rm(}second row{\rm)},  respectively, for $n=1:5000$. In the dashed dotted lines, the constants in $\mathcal O$ may have different values for different cases, and likewise for the figures hereinafter.}\label{figure14}
\end{figure}

\begin{figure}[hpbt]
\centering
\includegraphics[height=7cm,width=14cm]{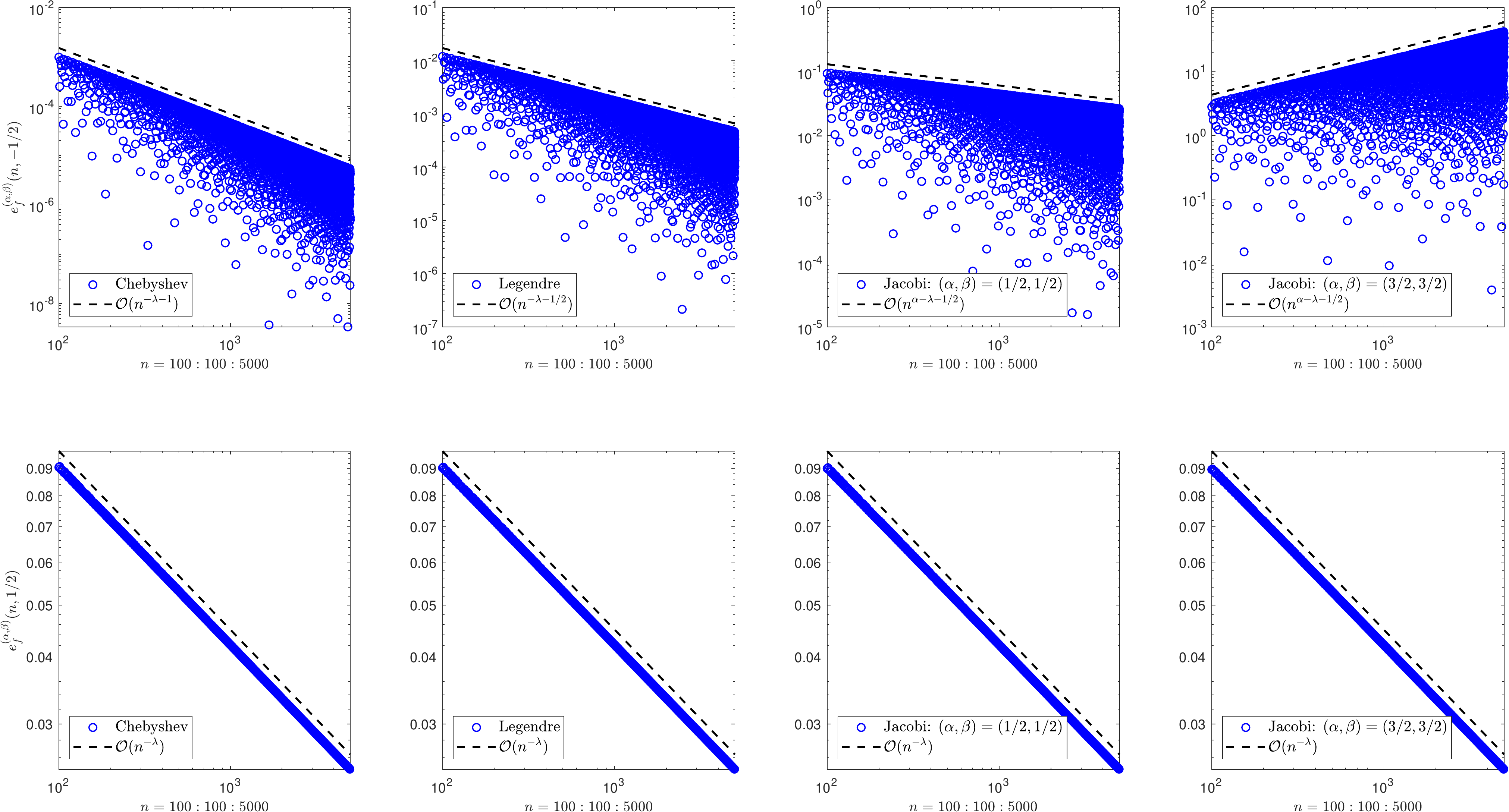}
%\vspace*{-5pt}
\caption{ Pointwise errors $e_f^{(\alpha,\beta)}(n,x)$ for $f(x)=(x-\frac{1}{4})_+^{\lambda}$ with $\lambda=\frac{1}{3}$ at endpoint $x=1$ {\rm(}first row{\rm)} and singular point $x=\frac{1}{4}$ {\rm(}second row{\rm)},  respectively, for $n=100:100:5000$.}\label{figure15}
\end{figure}

To check the error bounds (\ref{PerrJac100})-(\ref{PerrJac3})  numerically, we illustrate the convergence orders of the pointwise errors $e_{f}^{(\alpha,\beta)}(n,x)$ for functions $f(x)=(x-\frac{1}{4})_+^{\lambda}$ or $f(x)=|x-\frac{1}{4}|^{\lambda}$ with different values of $\lambda$ and $\alpha, \beta$. From Figs. \ref{figure14}-\ref{figure15}, we observe that these convergence orders are attainable and in accordance with the estimates stated in Theorem \ref{Thm13}, even for the divergent cases (see the last column of the first row in Fig. \ref{figure15} and Fig. \ref{figure16}).

Note that the function $C(x)$ near $x=\pm 1$ and $x=a$ is  described  in (\ref{PerrJac1}). We further demonstrate that
the estimate agrees well with the pointwise errors through a test on $f(x)=(x-a)_{+}^{\lambda}$. Pointwise errors around the point set $\left\{-1,a,1\right\}$ are plotted in Fig. \ref{figure44}, which implies the optimality on the estimates (\ref{PerrJac1}) in the sense that the orders on $\xi$ can not be improved.

\begin{figure}[!t]
\centering
\includegraphics[height=4cm,width=4.5cm]{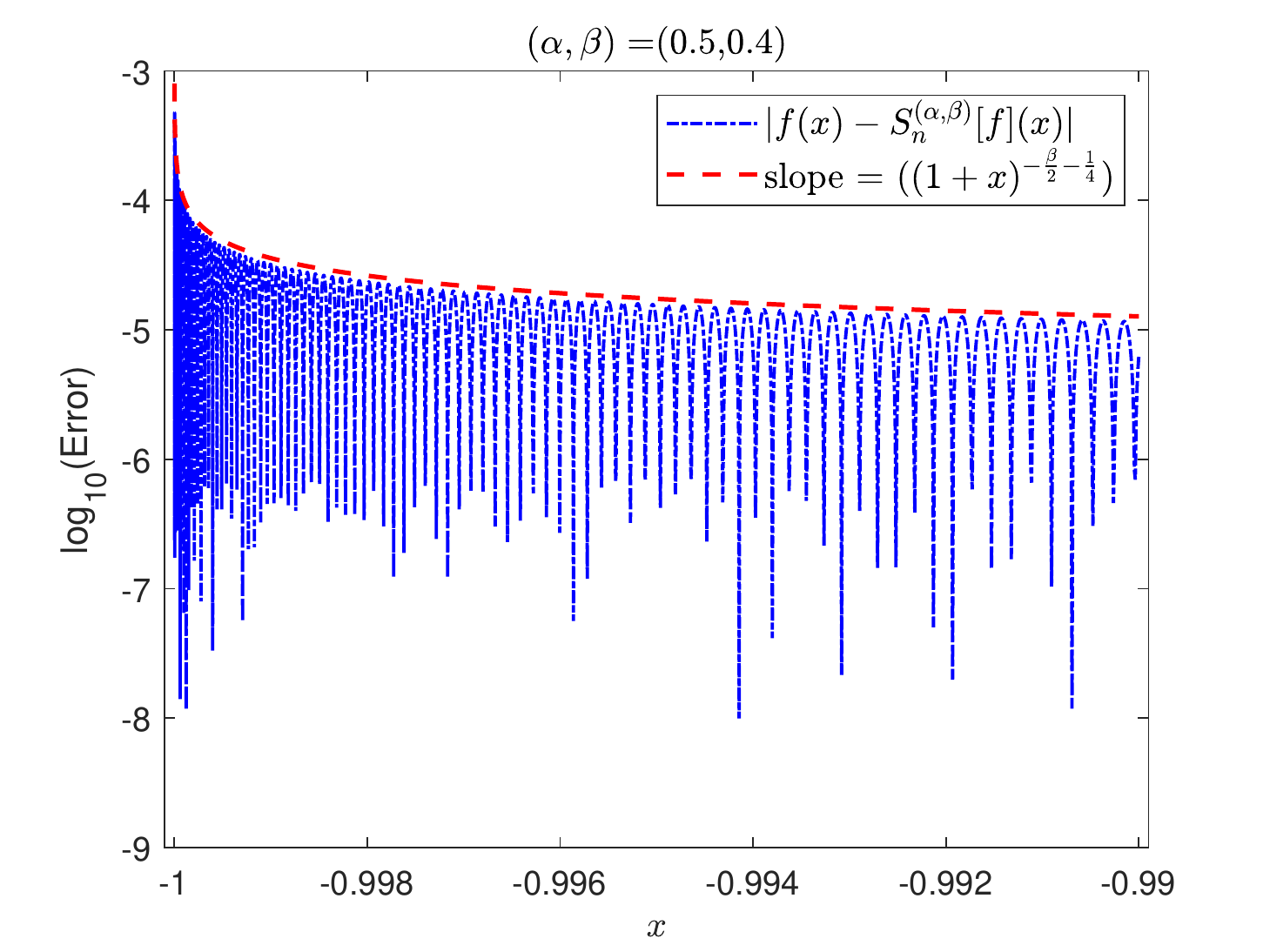}
\includegraphics[height=4cm,width=4.5cm]{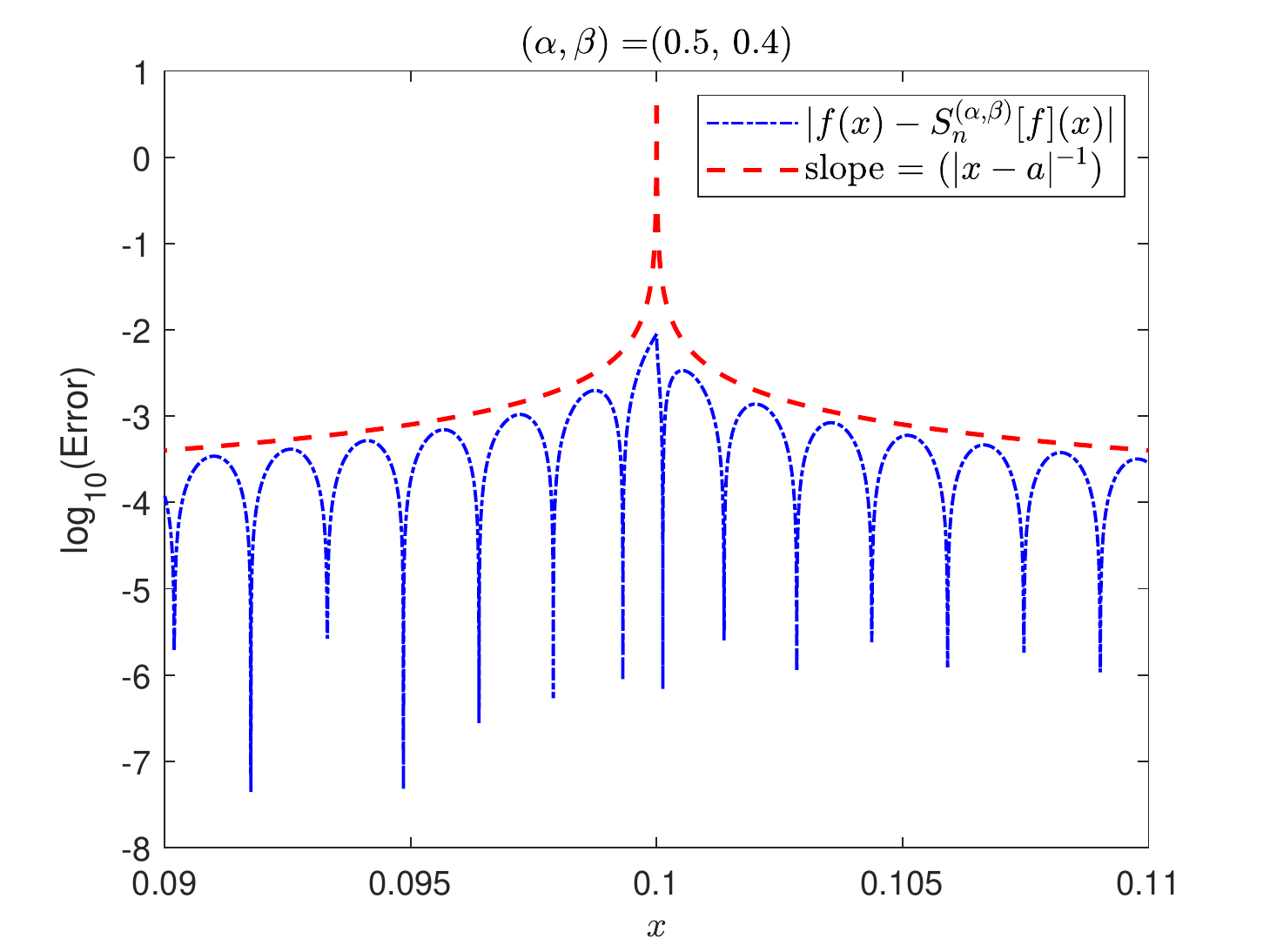}
\includegraphics[height=4cm,width=4.5cm]{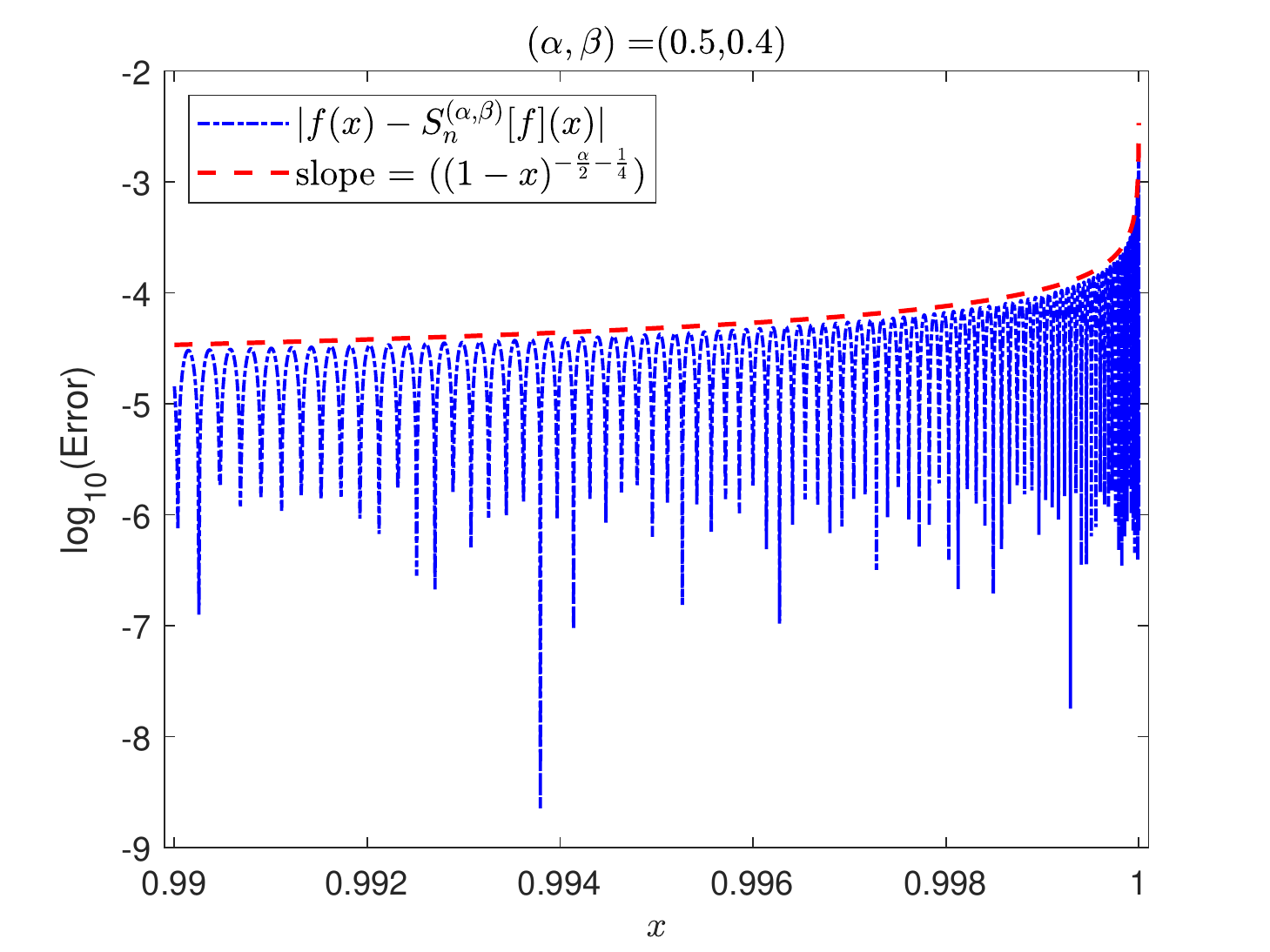}
\vspace*{-5pt}
\caption{ Pointwise error plots of $e_{f}^{(\alpha,\beta)}(n,x)$ around $x=-1$ (left), $x=a$ (middle) and $x=1$ (right), where $f(x)=\left(x-\frac{1}{10}\right)_{+}^{1/2}$, $\alpha=\frac{1}{2}$, $\beta=\frac{2}{5}$ and $n=2000$.}\label{figure44}
\end{figure}

%%%%%%%%%%%%%%%%%%%%%%%%%%%%%%%%%%%%%%%%%%%%%%%%%%%%
\section{Convergence rates in the maximum norm}\label{sec4}
 From the proof of Theorem \ref{Thm32} and  Theorem \ref{Thm13}, for  function (\ref{gpfun}) or (\ref{gafun}) with $\lambda>0$,
 we have the following asymptotic convergence rates.
\smallskip
\begin{corollary}\label{Thm14}
Suppose that $\lambda>0$ and $f(x)$ is a function defined in \eqref{gpfun} {\rm(}or \eqref{gafun} where $\lambda$ is not an even number{\rm)}, then for  $\alpha,\beta>-1$, \eqref{maxer1} holds, that is
\begin{equation*}
  \|f-S_{n}^{(\alpha,\beta)}[f]\|_{\infty} =\begin{dcases}
  {\cal O}\big(n^{\max\{\alpha-\frac{1}{2},\beta-\frac{1}{2}\}-\lambda}\big),&  {\rm if}\; \max\{\alpha,\beta\}>\frac{1}{2}, \\[4pt]
  {\cal O}\left(n^{-\lambda}\right),& {\rm if}\; \max\{\alpha,\beta\}\le\frac{1}{2}. \end{dcases}
\end{equation*}
\end{corollary}
\begin{proof}
We assume here that $f(x)$ is defined in (\ref{gpfun}), while for functions in (\ref{gafun}) a similar proof can be done by the fact $z(x)|x-a|^{\lambda}=f(x)+f^{*}(x)$.

If $x$ belongs to a closed subset of $[-1,a)\cup(a,1]$, that is $x\in [-1,a-\delta_0]\cup[a+\delta_0,1]$ for any fixed $\delta_{0}>0$ such that $a-\delta_{0}>-1$ and $a+\delta_{0}<1$, then $a_{n}^{(\alpha,\beta)}(x;g)$ will be uniformly bounded by $a_{n}^{(\alpha,\beta)}(x;g)=\mathcal{O}(n^{-\lambda-1/2})$. This further leads to
\begin{equation}\label{maxer1case1}
  e_{f}^{(\alpha,\beta)}(n,x)=\mathcal{O}(n^{-\lambda-\frac{1}{2}})\max\left\{\|P_n^{(\alpha,\beta)}\|_{\infty},\|P_{n+1}^{(\alpha,\beta)}\|_{\infty}\right\}
  =\mathcal{O}(n^{\max\left\{\alpha-\frac{1}{2},\beta-\frac{1}{2},-1\right\}-\lambda}).
\end{equation}

While if $x\in[a-\delta_{0},a+\delta_{0}]$, from the uniform estimate (\ref{coe11}) that $a_{n}^{(\alpha,\beta)}(x;g)=\mathcal{O}(n^{-\lambda+1/2})$ and $P_{n}^{(\alpha,\beta)}(x)=\mathcal{O}({n^{-1/2}})$ (see from Theorem \ref{Thm23}), we obtain
\begin{equation}\label{maxer1case2}
  e_{f}^{(\alpha,\beta)}(n,x)=\mathcal{O}(n^{-\lambda+\frac{1}{2}})\big(|P_{n}^{(\alpha,\beta)}(x)|+|P_{n+1}^{(\alpha,\beta)}(x)|\big)
  =\mathcal{O}(n^{-\lambda}),
\end{equation}
which together with (\ref{maxer1case1}) lead to (\ref{maxer1}).
\end{proof}

In Fig. \ref{figure16}, we demonstrate the convergence rates of the maximum error of $S_{n}^{(\alpha,\beta)}[f]$ and $p_{n}^{*}$ for $f(x)=|x-1/4|^{\lambda}$. Obviously, all these numerical results are consistent with the theoretical estimate (\ref{maxer1}).

\begin{figure}[t]
\centering
\includegraphics[width=.341\textwidth]{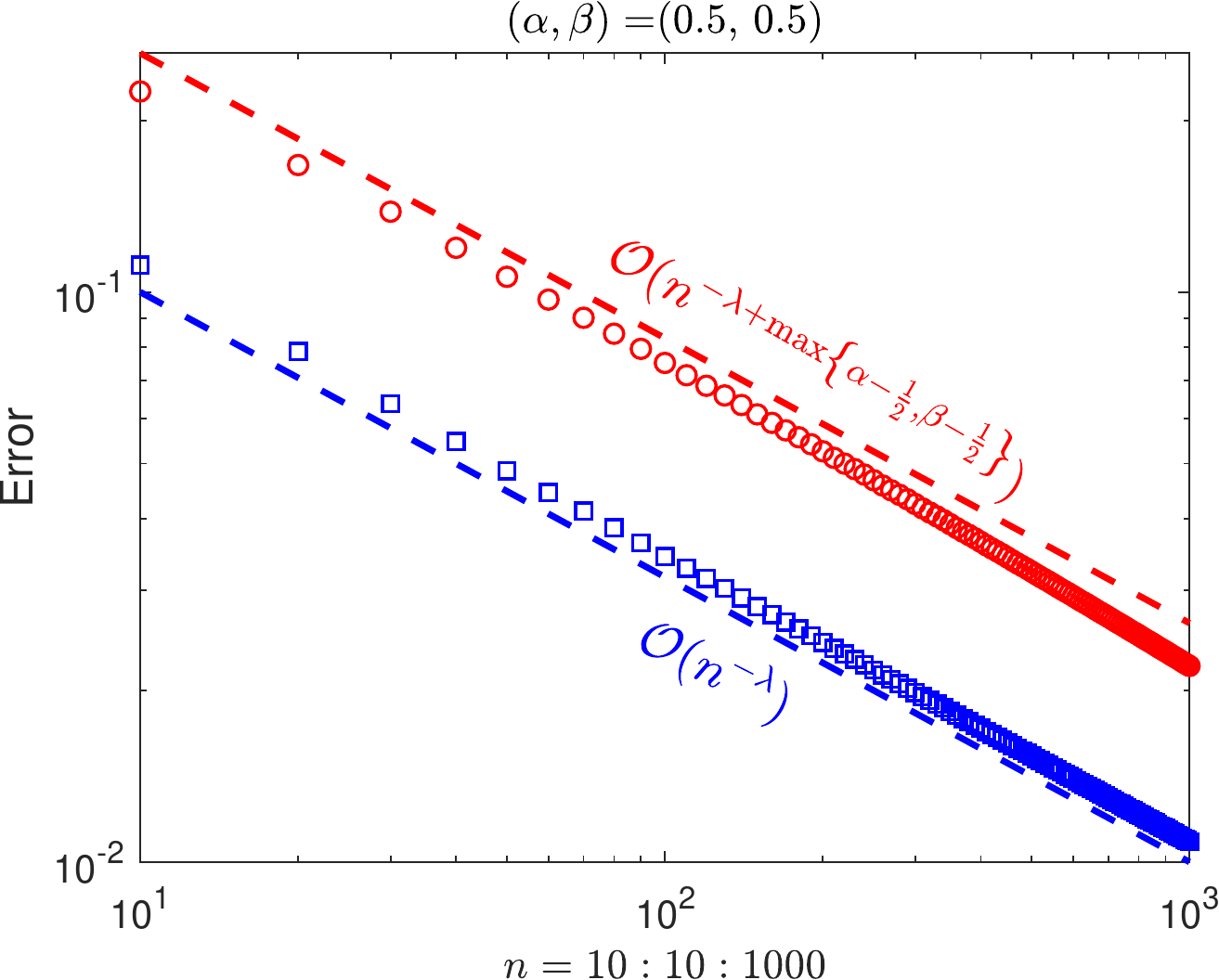}\hspace{-3pt}
\includegraphics[width=.325\textwidth]{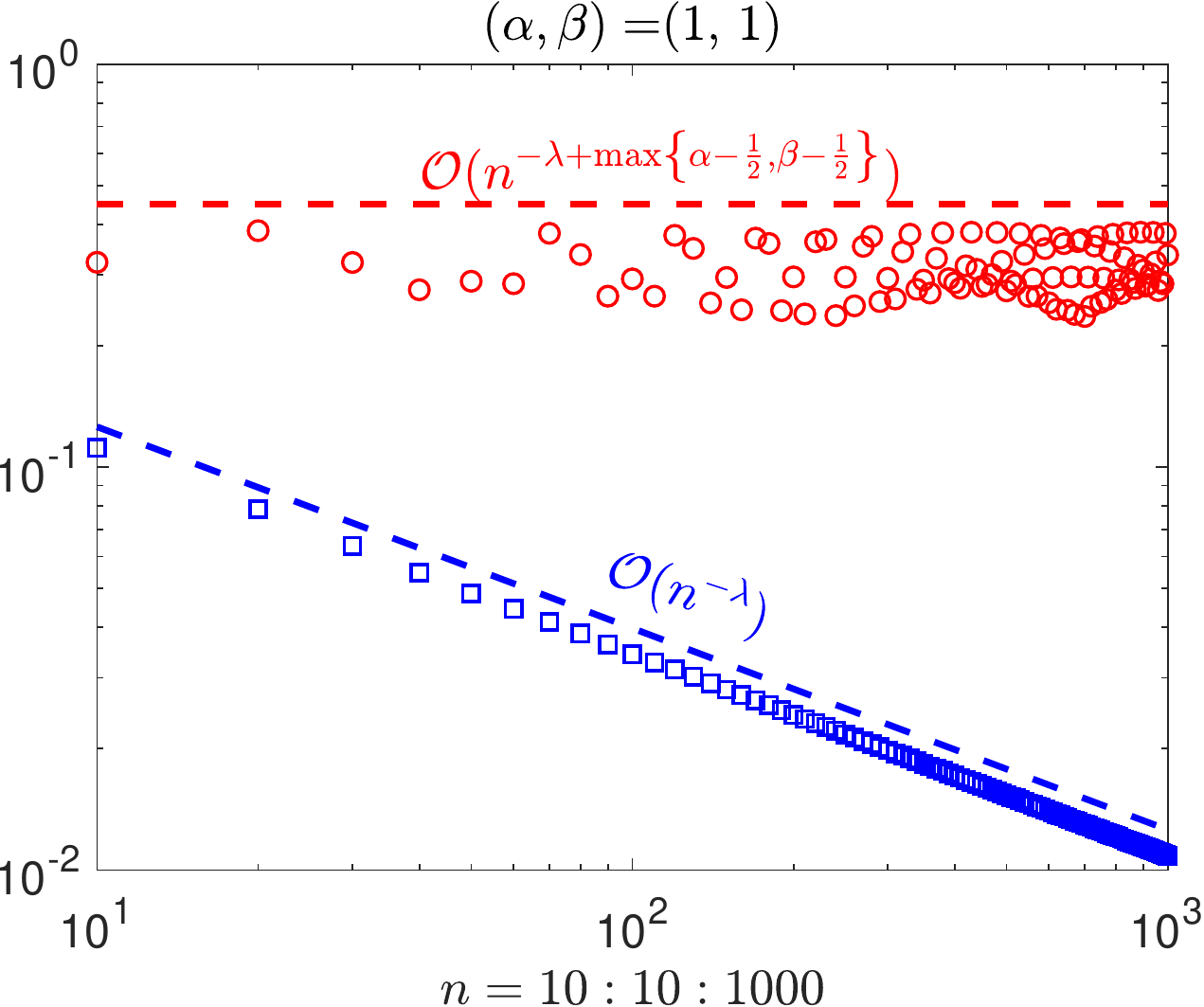}\hspace{-3pt}
\includegraphics[width=.325\textwidth]{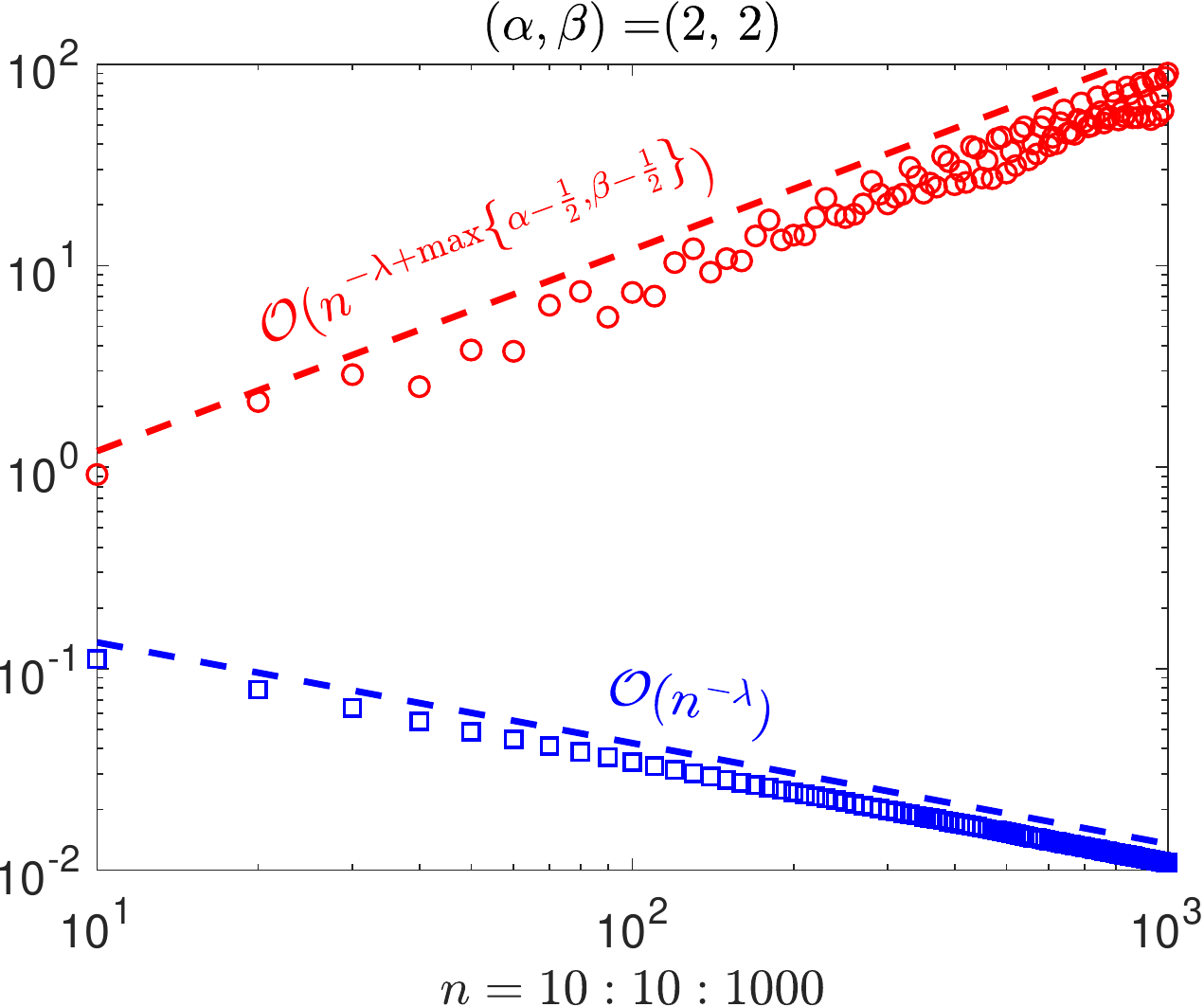} \\ [10pt]
\includegraphics[width=.341\textwidth]{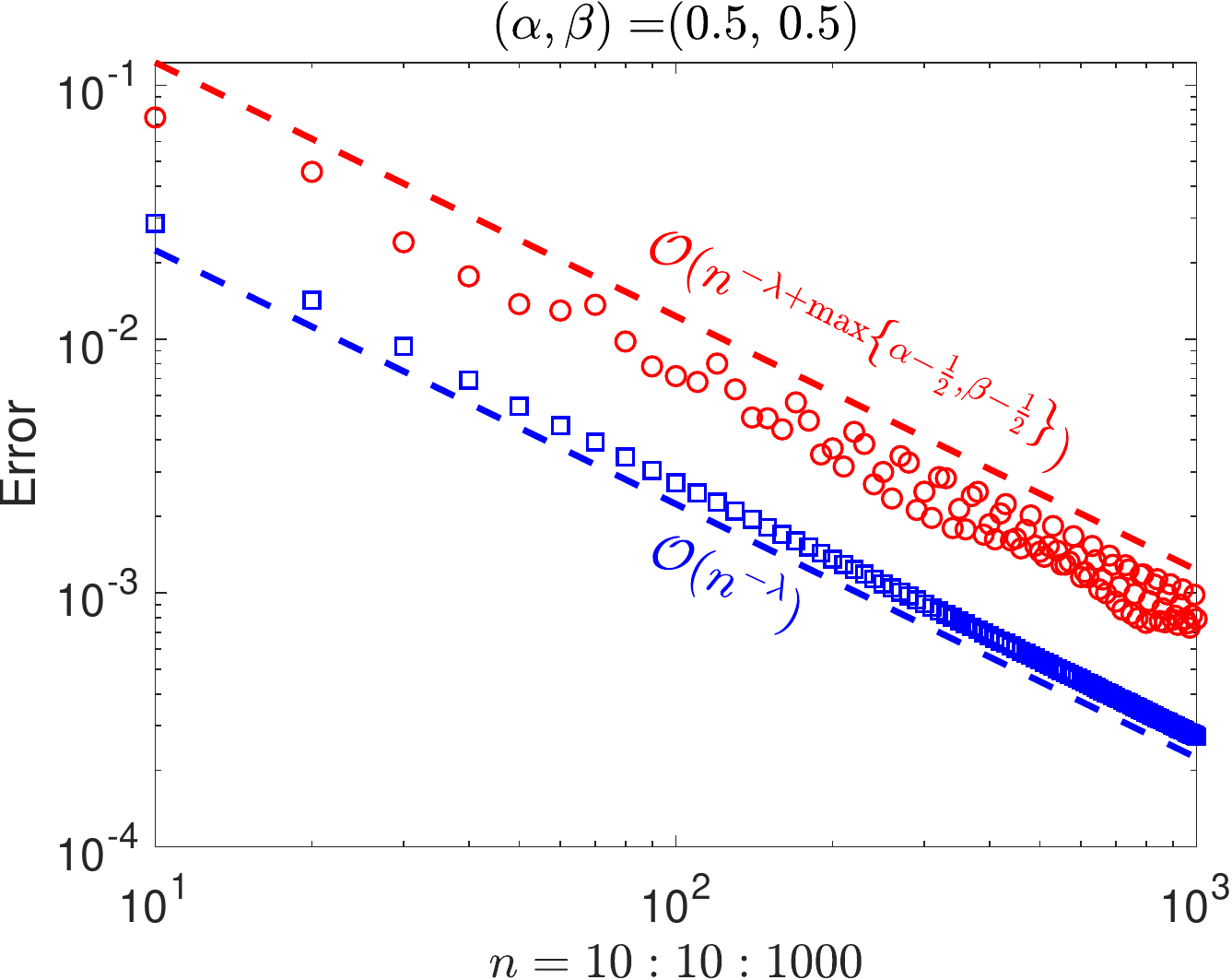}\hspace{-3pt}
\includegraphics[width=.325\textwidth]{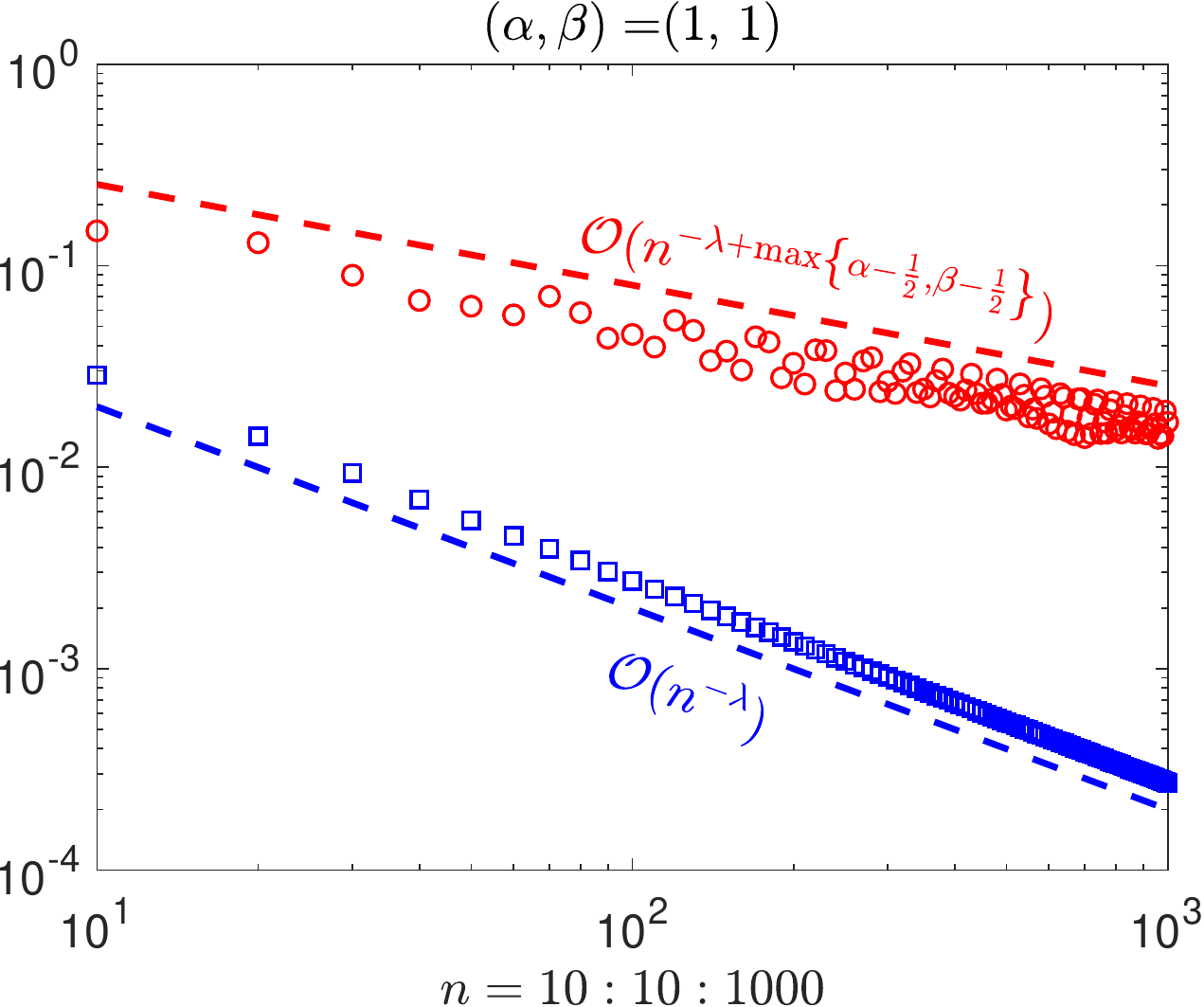}\hspace{-3pt}
\includegraphics[width=.325\textwidth]{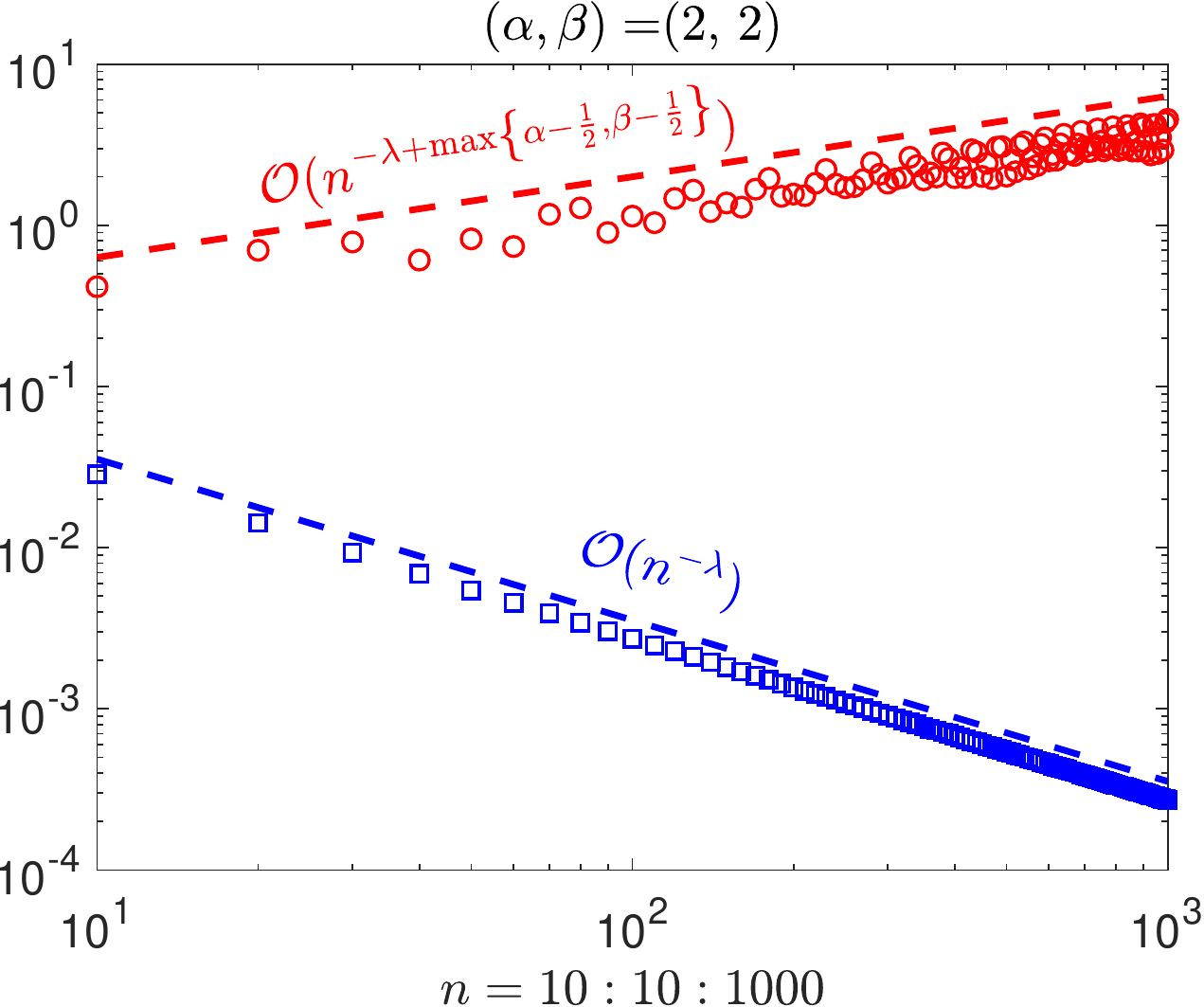}
\caption{ Maximum error $\|f-S_n^{(\alpha,\beta)}[f]\|_{\infty}$ {\rm(}red circle{\rm)} and $\|f-p_n^{*}\|_{\infty}$ {\rm(}blue square{\rm)} for $f(x)=|x-\frac{1}{4}|^{\lambda}$ with $\lambda=1/2$ {\rm(}first row{\rm)} and  $\lambda=1$ {\rm(}second row{\rm)}, respectively, for $n=10:10:10^3$.}\label{figure16}
\end{figure}

\medskip
Accordingly, it is interesting to examine the  weighted pointwise error defined in (\ref{weightedpointerror}), which is one order higher than the optimal polynomial approximation as stated below.

\begin{corollary}\label{Thm15}
Suppose that $\lambda>0$ and $f(x)$ is a function defined in \eqref{gpfun} {\rm(}or \eqref{gafun} where $\lambda$ is not an even number{\rm)}, then for  $\alpha,\beta>-1$, \eqref{weightedpointerrormax} holds, i.e., $\|\hat{e}_{f}^{(\alpha,\beta)}\|_{\infty}= {\cal O}(n^{-\lambda-1})$.
\end{corollary}
\begin{proof}
From \eqref{Jasy}, we have

\begin{equation}\label{Jacweiuniform}
  (1-x)^{\max\{\frac{\alpha}{2}+\frac{1}{4},0\}}(1+x)^{\max\{\frac{\beta}{2}+\frac{1}{4},0\}}P_n^{(\alpha,\beta)}(x)= {\cal O}(n^{-\frac{1}{2}}),
\end{equation}
which, together with (\ref{coe1}), leads to
$$
	\hat{e}_{f}^{(\alpha,\beta)}(n,x)=(1-x)^{\max\{\frac{\alpha}{2}+\frac{1}{4},0\}}(1+x)^{\max\{\frac{\beta}{2}+\frac{1}{4},0\}}(x-a)e_{f}^{(\alpha,\beta)}(n,x)= {\cal O}(n^{-\lambda-1}),
$$
when $x\neq a$, where the constant in $\mathcal{O}$-term is independent of $x$. While when $x=a$, it is obvious that $\hat{e}_{f}^{(\alpha,\beta)}(n,a)=0$. This completes the proof.
\end{proof}

Corollary \ref{Thm15} indicates that the weighted pointwise error of $S_{n}^{(\alpha,\beta)}[f]$ in the uniform norm is one order higher in convergence rate than the optimal polynomial approximation. However, if we consider the weighted pointwise error by removing the factor $(x-a):$
$$	\tilde{e}_{f}^{(\alpha,\beta)}(n,x)=(1-x)^{\max\{\frac{\alpha}{2}+\frac{1}{4},0\}}(1+x)^{\max\{\frac{\beta}{2}+\frac{1}{4},0\}}e_{f}^{(\alpha,\beta)}(n,x),
$$
then we can obtain similarly the following asymptotic estimate.
\begin{corollary}
Suppose that $\lambda>0$ and $f(x)$ is defined by \eqref{gpfun} {\rm{(}}or \eqref{gafun} where $\lambda$ is not an even number{\rm)}, then for  $\alpha,\beta>-1$, we have the following estimate
\begin{equation}
  \|\tilde{e}_{f}^{(\alpha,\beta)}\|_{\infty}={\cal O}(n^{-\lambda}).
\end{equation}
\end{corollary}
\begin{proof}
  This is obtained by the uniform estimate $a_{n}^{(\alpha,\beta)}(x;g)=\mathcal{O}(n^{-\lambda+1/2})$ and (\ref{Jacweiuniform}).
\end{proof}

\begin{figure}[hpbt]
\centering
\includegraphics[width=.341\textwidth]{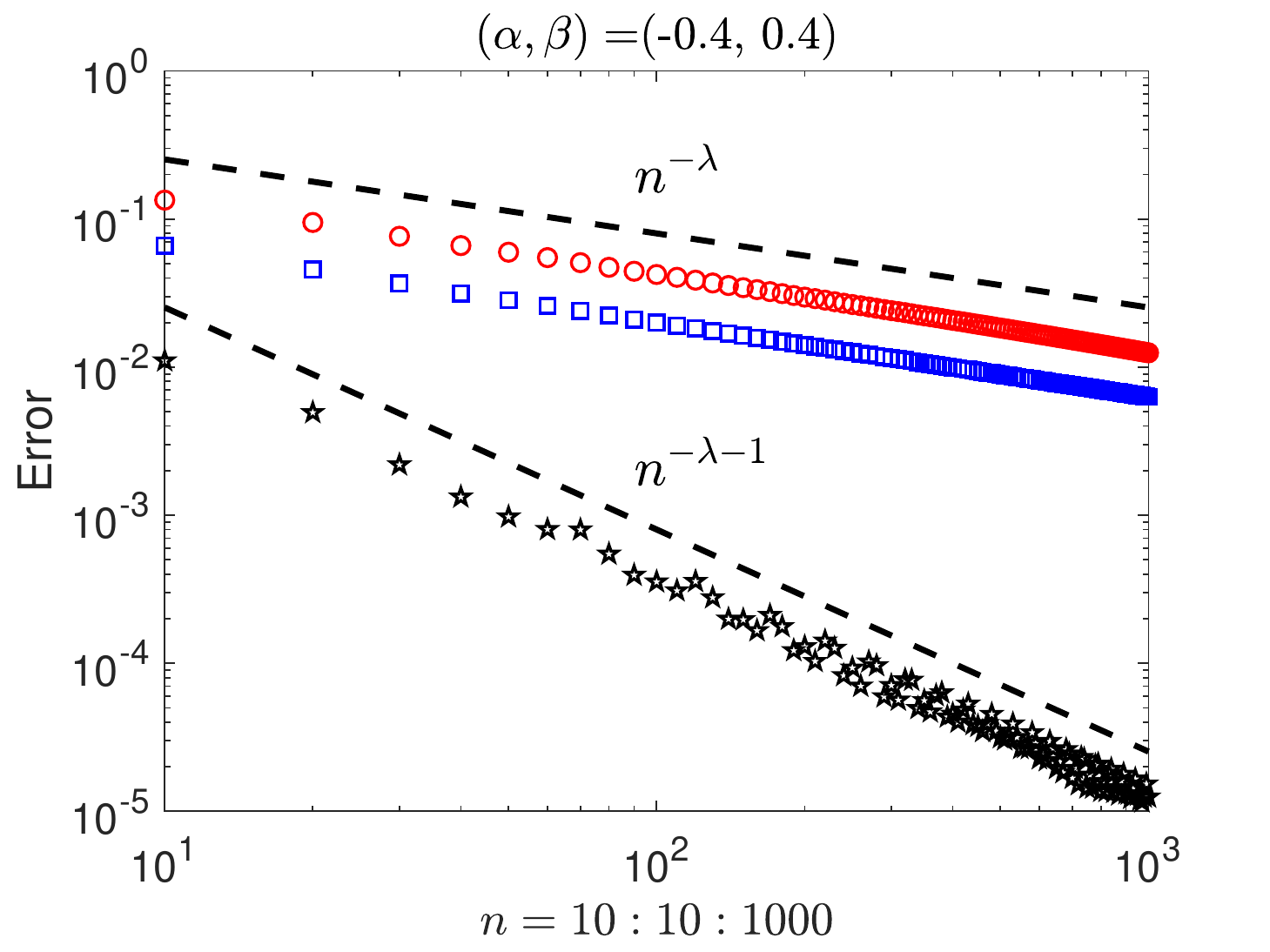}\hspace{-3pt}
\includegraphics[width=.325\textwidth]{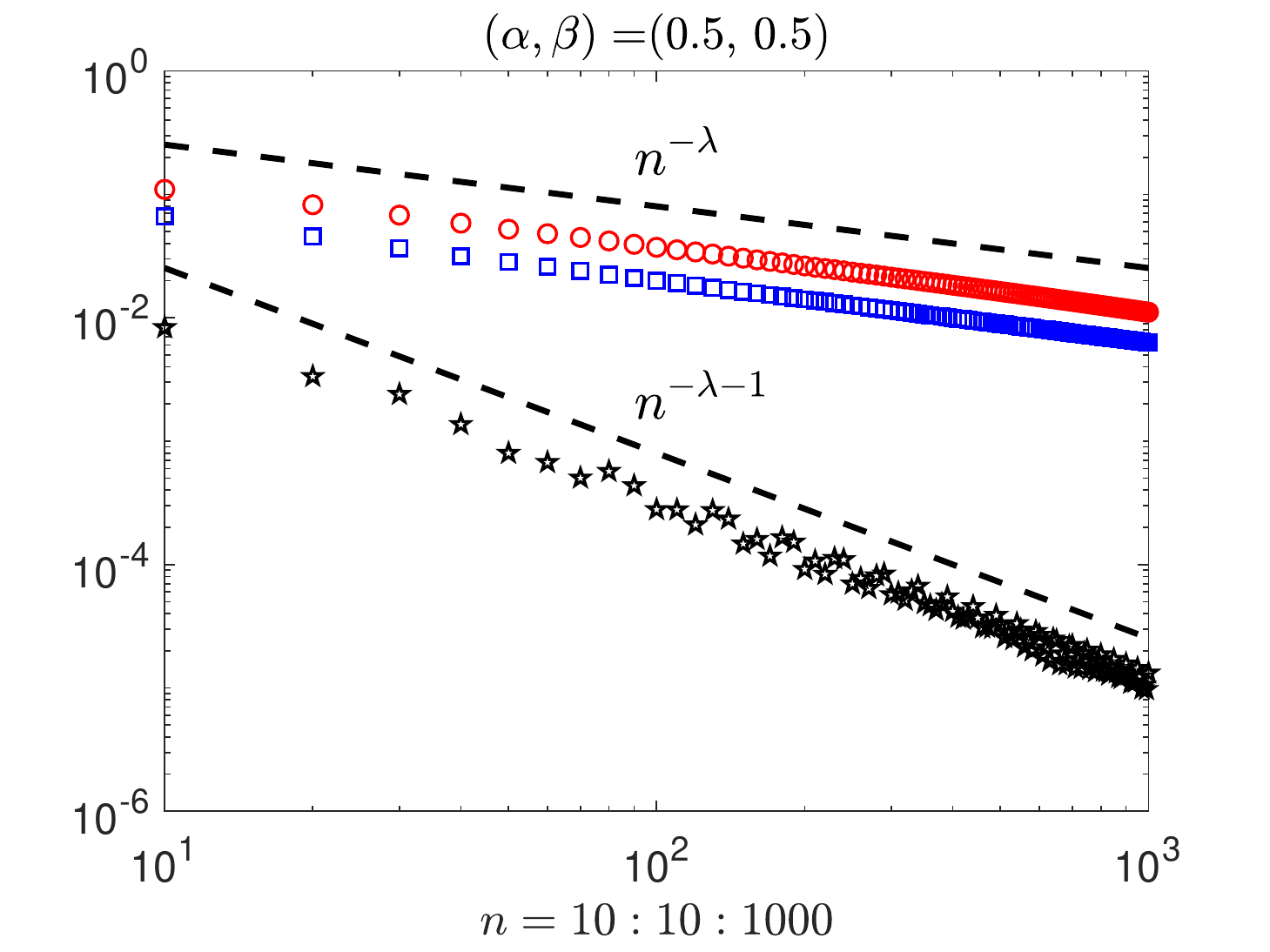}\hspace{-3pt}
\includegraphics[width=.325\textwidth]{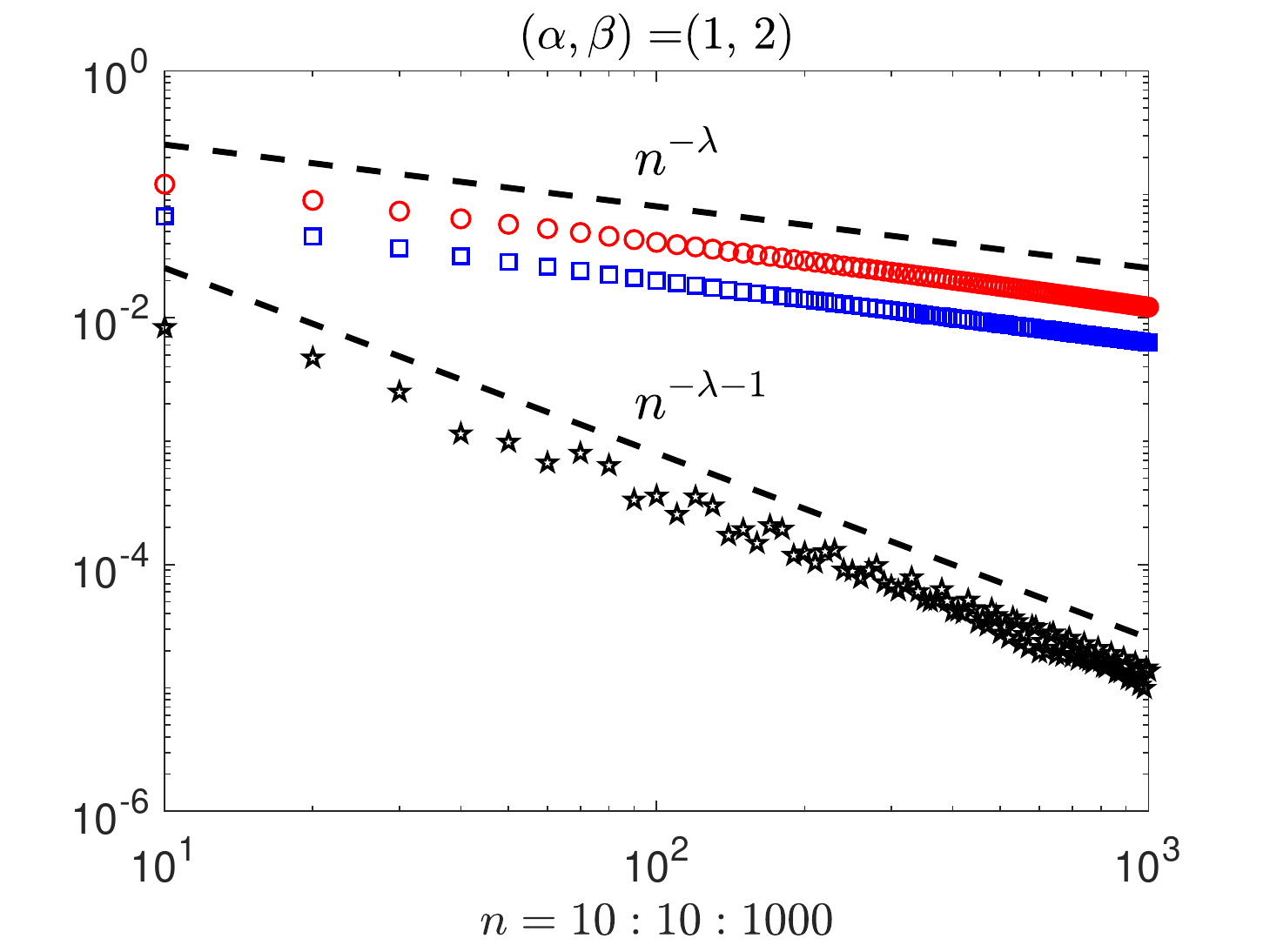} \\ [10pt]
\includegraphics[width=.341\textwidth]{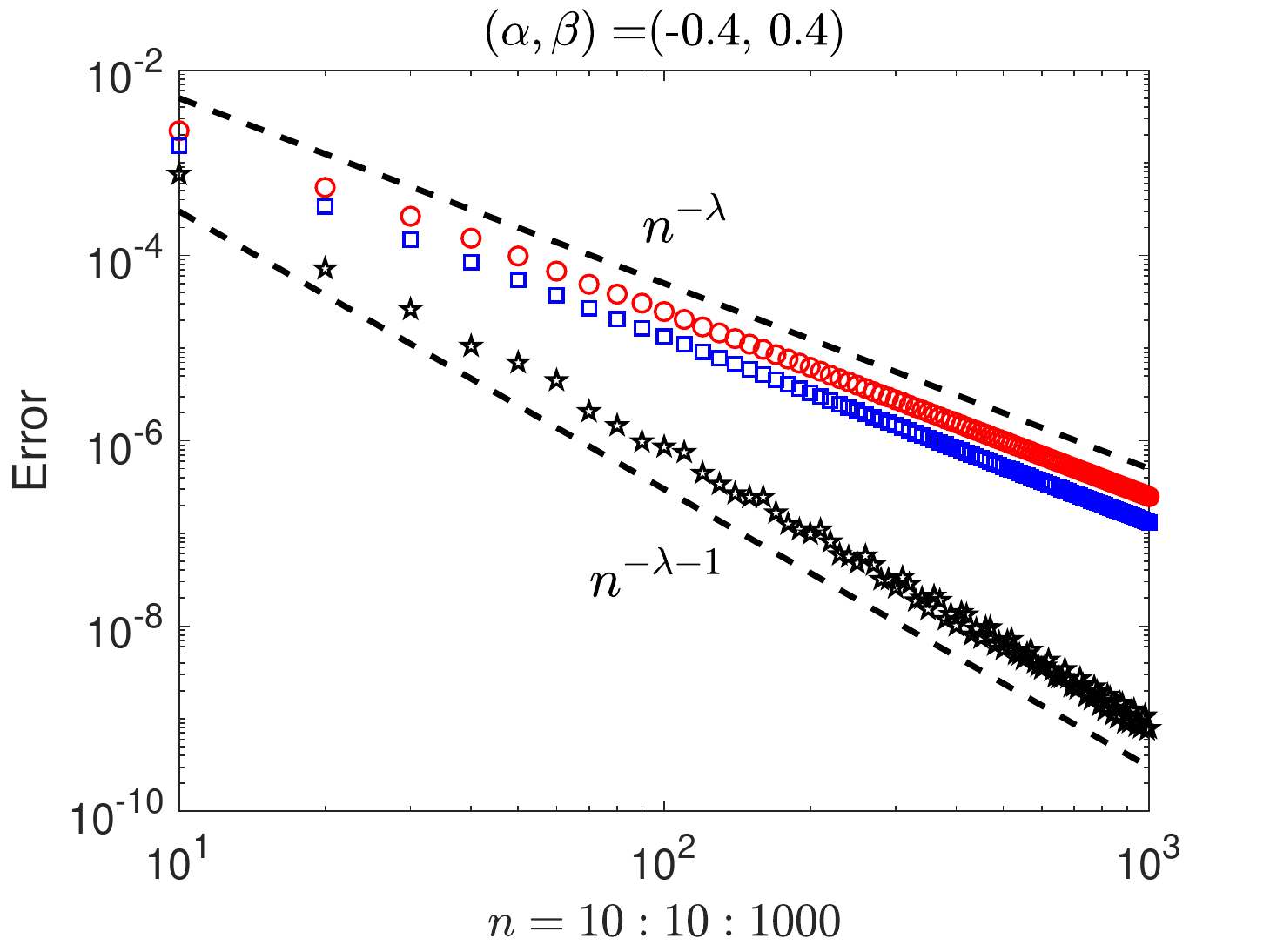}\hspace{-3pt}
\includegraphics[width=.325\textwidth]{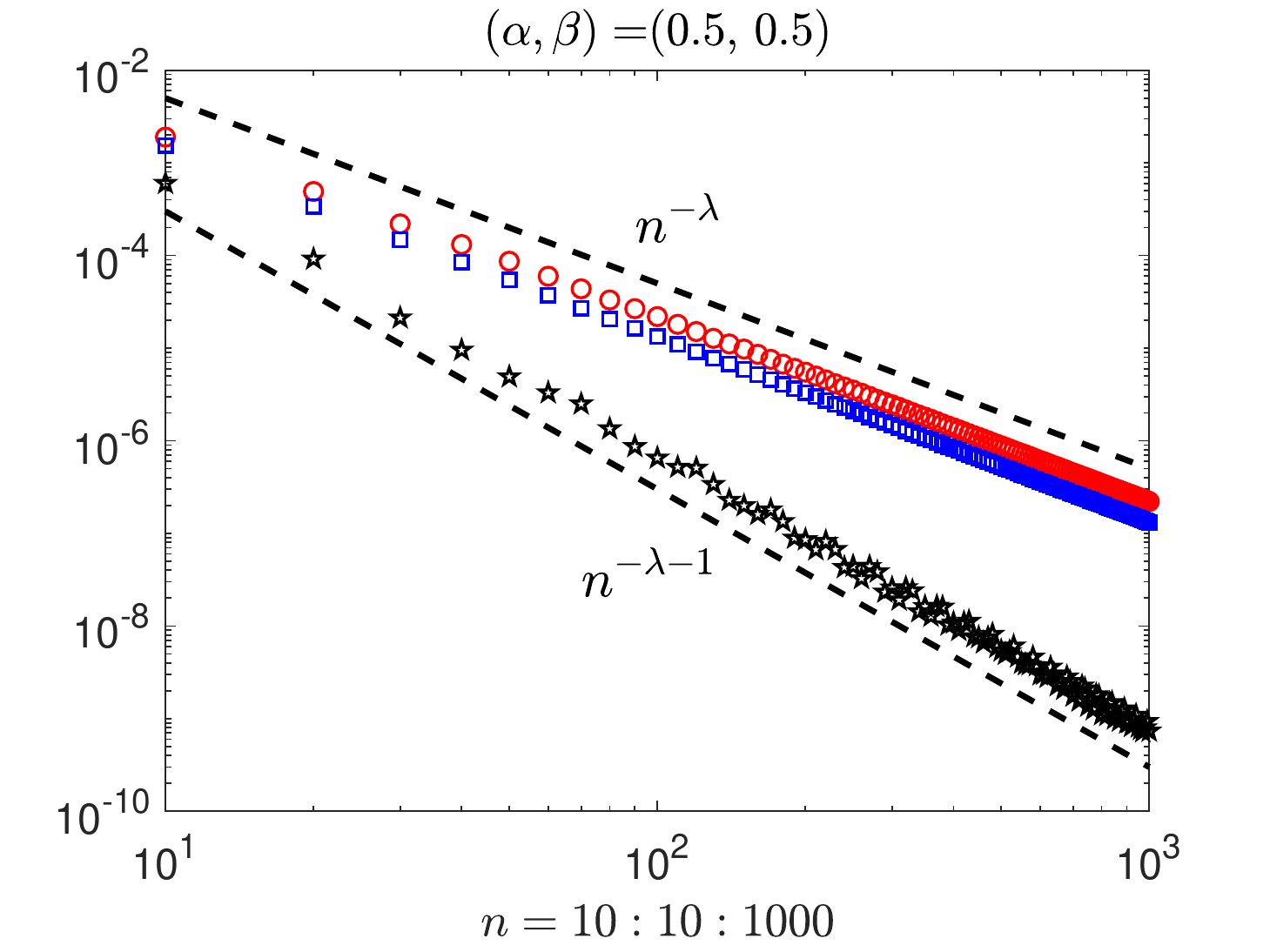}\hspace{-3pt}
\includegraphics[width=.325\textwidth]{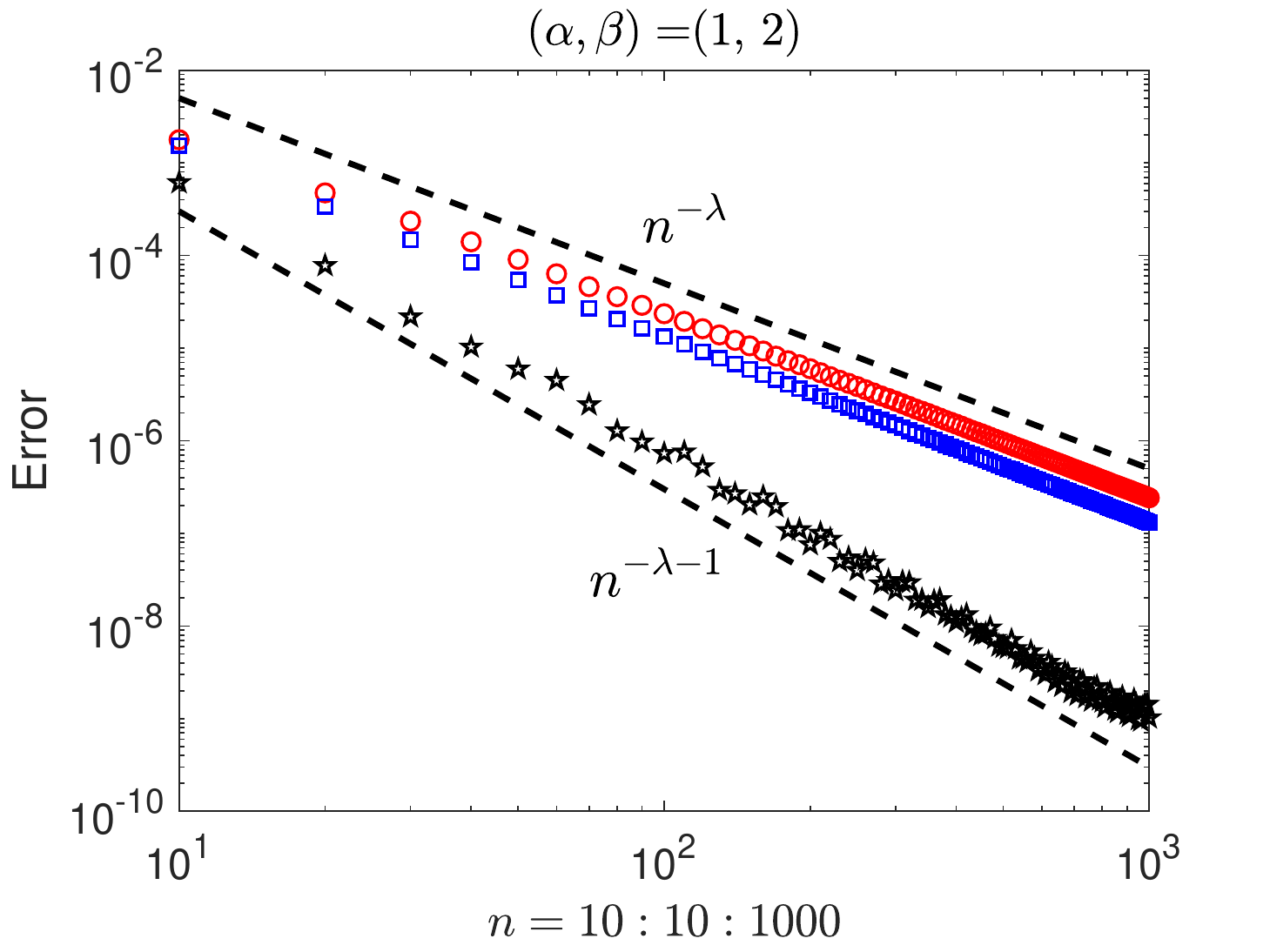}
\caption{Plots of $\|\hat{e}_{f}^{(\alpha,\beta)}\|_{\infty}$ {\rm(}black pentagram{\rm)}, $\|\tilde{e}_{f}^{(\alpha,\beta)}\|_{\infty}$ {\rm(}red circle{\rm)} and  the best approximation $\|f-p^{*}\|_{\infty}$ {\rm(}blue square{\rm)} for $f(x)=(x-1/4)_{+}^{\lambda}$ with $\lambda=1/2$ {\rm(}first row{\rm)} and $\lambda=2$ {\rm(}second row{\rm)}.}\label{figure51}
\end{figure}

Numerical results  in Fig.\! \ref{figure51}  illustrate the optimal estimates on these two kinds of weighted pointwise errors.
%%%%%%%%%%%%%%%%%%%%%%%%%%%%%%%%%%%%%%%%%%%%%%%%%%%%%%%%%%%%%

\section{Extensions to  functions with endpoint singularities}\label{sec5}
 In this section, we intend to study the pointwise error estimates and local superconvergence of Jacobi expansions to functions with endpoint singularities, which are stated in the following theorems.

In fact, the estimates in Lemma \ref{van der Corput Jacobi} can be generalized to the case with $a=-1$.
\begin{lemma}\label{van der boundary}
	Let $\alpha,\beta,\gamma,\delta>-1$ and $\psi(x)\in W_{\rm AC}(-1,1)$, then for $n\gg1$ we have
	\begin{equation}\label{vanbound}
		\int_{-1}^{1}\!(1+x)^{\gamma}(1-x)^{\delta}P_{n}^{(\alpha,\beta)}(x)\psi(x)\,\mathrm{d}x=\|\psi\|_{W_{\rm AC}(-1,1)}\cdot\mathcal{O}\big(n^{-\min\left\{2\delta-\alpha+2,2\gamma-\beta+2,\frac{3}{2}\right\}}\big).
	\end{equation}
\end{lemma}
\begin{proof}
	In order to obtain the estimate, we split the integral in (\ref{vanbound}) into two parts as follows
	\begin{equation*}		\int_{-1}^{0}\!(1+x)^{\gamma}P_{n}^{(\alpha,\beta)}(x)(1-x)^{\delta}\psi(x)\,\mathrm{d}x+\int_{0}^{1}\!(1-x)^{\delta}P_{n}^{(\alpha,\beta)}(x)(1+x)^{\gamma}\psi(x)\,\mathrm{d}x.
	\end{equation*}
	Using  Lemma \ref{van der Corput Jacobi}, we can derive the estimate
	\begin{equation*}\label{vanbound1}
\begin{split}
		\int_{0}^{1}\!(1-x)^{\delta}P_{n}^{(\alpha,\beta)}(x) (1+x)^{\gamma}\psi(x)\,\mathrm{d}x&=\|(1+x)^{\gamma}\psi(x)\|_{W_{\rm AC}(0,1)}\cdot\mathcal{O}\big(n^{-\min\{2\delta-\alpha+2,\frac{3}{2}\}}\big)
\\&=\|\psi\|_{W_{\rm AC}(0,1)}\cdot\mathcal{O}\big(n^{-\min\{2\delta-\alpha+2,\frac{3}{2}\}}\big).
\end{split}
	\end{equation*}
	Similarly, we can show that
	\begin{equation*}
		\int_{-1}^{0}\!(1+x)^{\gamma}P_{n}^{(\alpha,\beta)}(x)(1-x)^{\delta}\psi(x)\,\mathrm{d}x=\|\psi\|_{W_{\rm AC}(-1,0)}\cdot\mathcal{O}\big(n^{-\min\{2\gamma-\beta+2,\frac{3}{2}\}}\big).
	\end{equation*}
Then the estimate (\ref{vanbound}) follows immediately from the above.
\end{proof}

We have the following results on the pointwise estimates for the endpoint singularities.
\begin{theorem}\label{ThmB}
Define $f_1(x)=(1- x)^\lambda z(x)\; (\lambda+\alpha>-1)$ and $f_2(x)=(1+ x)^\lambda z(x)\; (\lambda+\beta>-1)$ where the given function $z(x)$ is smooth with
$z(\pm 1)\not= 0$. Then for $\lambda>-1$  not an integer, we have the following pointwise error estimates.
\begin{itemize}
\item[{\rm(i)}]  For $x\in (-1,1)$, we have
\begin{equation}\label{PerrJac1B-2}
e_{f_1}^{(\alpha,\beta)}(n,x)\le C_{1}(x)n^{-2\lambda-\alpha-\frac{3}{2}},\quad e_{f_2}^{(\alpha,\beta)}(n,x)\le C_{2}(x)n^{-2\lambda-\beta-\frac{3}{2}},
\end{equation}
where  $C_i(x)$ is independent of $n$ and has the behaviours near $x=\pm 1$ as follows for some $\delta>0$ and $0<\xi\le \delta$
\begin{equation}\label{PerrJac1B}
\begin{split}
&C_1(1- \xi)\le D(1)\,\xi^{-\max\{\frac \alpha2+\frac 14,0\}-1},\quad C_1(-1+ \xi)\le D(-1)\,\xi^{-\max\{\frac \beta2+\frac 14,0\}}, \\
&C_2(1- \xi)\le D(1)\,\xi^{-\max\{\frac \alpha2+\frac 14,0\}},\quad C_2(-1+ \xi)\le D(-1)\,\xi^{-\max\{\frac \beta2+\frac 14,0\}-1}
\end{split}
\end{equation}
with $D(\pm 1)$ independent of $\xi$ and $n$.
\smallskip
\item[{\rm (ii)}]  At  $x=\pm 1$,  we have
\begin{equation}\label{PerrJac2B}\begin{split}
e_{f_1}^{(\alpha,\beta)}(n,1)\le C n^{-2\lambda},\quad e_{f_2}^{(\alpha,\beta)}(n,-1)\le C n^{-2\lambda}.
\end{split}\end{equation}
%\smallskip
\item[{\rm (iii)}]  For the weighted pointwise error
\begin{equation}\label{PerrJac2C}
\begin{aligned}
	(1-x)^{\max\{\frac{\alpha}{2}+\frac{1}{4},0\}}(1+x)^{\max\{\frac{\beta}{2}+\frac{1}{4},0\}}(1-x)e_{f_1}^{(\alpha,\beta)}(n,x)= {\cal O}(n^{-2\lambda-\alpha-\frac{3}{2}}), \\ (1-x)^{\max\{\frac{\alpha}{2}+\frac{1}{4},0\}}(1+x)^{\max\{\frac{\beta}{2}+\frac{1}{4},0\}}(1+x)e_{f_2}^{(\alpha,\beta)}(n,x)= {\cal O}(n^{-2\lambda-\beta-\frac{3}{2}}),
\end{aligned}\end{equation}
where the $\mathcal{O}$-terms involved are independent of $x$.
\end{itemize}
\end{theorem}
\begin{proof}
For simplicity, we only consider the estimates for $f_1$. By $P_n^{(\alpha,\beta)}(-x)=(-1)^nP_n^{(\beta,\alpha)}(x)$ it leads to the desired results for $f_2$.

Notice that
\begin{equation}\label{bd1}\begin{aligned}
  a_{n}^{(\alpha,\beta)}(x;g)
&=\frac{1}{\sigma_{n}^{(\alpha,\beta)}}\left[\int_{-1}^{1}\!\frac{z(x)-z(y)}{x-y}(1-y)^{\lambda}P_{n}^{(\alpha,\beta)}(y)\omega^{(\alpha,\beta)}(y)\,\mathrm{d}y\right.\\
	&\qquad\qquad\left.+\int_{-1}^{1}\!\frac{(1-x)^{\lambda}-(1-y)^{\lambda}}{x-y}z(x)P_{n}^{(\alpha,\beta)}(y)\omega^{(\alpha,\beta)}(y)\,\mathrm{d}y\right]
\end{aligned}
\end{equation}
 and $\frac{z(x)-z(y)}{x-y}$ is smooth in $[-1,1]$ and satisfies for some $\xi$ between $x$ and $y$ that
\begin{equation}\label{Taylor}
\partial_y^{k}\left(\frac{z(x)-z(y)}{x-y}\right)=k!\frac{z(x)-\sum_{j=0}^k\frac{z^{(j)}(y)}{j!}(x-y)^j}{(x-y)^{k+1}}=\frac{1}{k+1}z^{(k+1)}(\xi)
\end{equation}
which is uniformly bounded independent of $x$ and $y$ for any fixed positive integer $k$.
Then it follows from the proof of \cite[Theorem 3.1]{Xiang2020} with $\mu=0$ that
\begin{equation}\label{bd2}
   \frac{1}{\sigma_{n}^{(\alpha,\beta)}}\int_{-1}^{1}\!\frac{z(x)-z(y)}{x-y}(1-y)^{\lambda}P_{n}^{(\alpha,\beta)}(y)\omega^{(\alpha,\beta)}(y)\,\mathrm{d}y={\cal O}(n^{-2\lambda-\alpha-1})
  \end{equation}
independent of $x$.

Now we turn to the second term in \eqref{bd1}. Firstly, take $m$ be a positive integer such that $\alpha+2\lambda+2-m\le \frac{3}{2}\le \beta+m+2$ and $m>\lambda$, then it follows from \cite{Xiang2020} that
\begin{equation}\label{bd3}\begin{split} &\frac{z(x)}{\sigma_{n}^{(\alpha,\beta)}}\int_{-1}^{1}\!\frac{(1-x)^{\lambda}-(1-y)^{\lambda}}{x-y}P_{n}^{(\alpha,\beta)}(y)\omega^{(\alpha,\beta)}(y)\,\mathrm{d}y\\
=&\frac{z(x)}{2^{m}(n)_{m}\sigma_{n}^{(\alpha,\beta)}}
\int_{-1}^{1}\!(1-y)^{\lambda-m}P_{n-m}^{(\alpha+m,\beta+m)}(y)\psi_{3}(x,y)\omega^{(\alpha+m,\beta+m)}(y)\,\mathrm{d}y
\end{split}\end{equation}
where
$${\displaystyle\psi_{3}(x,y)=(1-y)^{m-\lambda}\partial_y^{m}\Big(\frac{(1-x)^{\lambda}-(1-y)^{\lambda}}{x-y}\Big)}.$$
Analogously, $\psi_{3}(x,y)$ is continuous and for any integer $k>\lambda$, there exists $\xi_1$ between $x$ and $y$ such that
\begin{equation}\label{psi3}\begin{split}	(-1)^{\lfloor\lambda\rfloor+1}\partial_y^{k}\Big(\frac{(1-x)^{\lambda}-(1-y)^{\lambda}}{x-y}\Big)
=\frac{(-1)^{\lfloor\lambda\rfloor+1}}{k+1}(\lambda)_{k+1}(1-\xi_1)^{\lambda-k-1}(-1)^{k+1}>0.
\end{split}\end{equation}
In addition, by an  argument similar to the proof of the monotonicity of $\psi_{2}$ in Appendix \ref{AppendixA} but with $1-y$ in place of $y-a$ and $m$ in place  of $m+1$, we can readily prove that
\begin{equation*}
	(-1)^{\lfloor\lambda\rfloor+1}\partial_{y}\psi_{3}(x,y)>0.
\end{equation*}
This together with (\ref{psi3}) indicates that $\psi_{3}$ is positive and increasing w.r.t $y\in[-1,1]$ if $\lfloor\lambda\rfloor$ is odd, otherwise $\psi_{3}$ will be negative and decreasing. As a result, it leads to
\begin{align*}
	&\max_{y\in[-1,1]}|\psi_{3}(x,y)|=|\psi_{3}(x,1)|\leq C(1-x)^{-1},\\
 &\int_{-1}^{1}|\partial_y\psi_{3}(x,y)|\mathrm{d}y=\left|\psi_{3}(x,1)-\psi_{3}(x,-1)\right|\leq C(1-x)^{-1},
\end{align*}
so $\psi_{3}(x,\cdot)\in W_{\rm AC}(-1,1)$. Then by Lemma \ref{van der boundary}, it establishes from \eqref{bd1} and \eqref{bd3} that
\begin{equation}\label{bd4}
  a_n^{(\alpha,\beta)}(x;g)=(1-x)^{-1}\cdot{\cal O}(n^{-2\lambda-\alpha-1}),
\end{equation}
which together with (\ref{Jtruerr}) and Corollary \ref{Bineua}  leads to the desired result (\ref{PerrJac1B-2})-(\ref{PerrJac1B}).

When $x=1$, it is difficult to establish \eqref{PerrJac2B} from \eqref{Jtruerr}, but it can be derived from the estimate (see \cite[Theorem 3.1]{Xiang2020})
\begin{equation*}
  a_n^{(\alpha,\beta)}(x;f_1)={\cal O}(n^{-2\lambda-\alpha-1})
\end{equation*}
and Theorem \ref{Thm23}, that is
\begin{equation*}\label{PerrJac2BD}
\begin{split}
e_{f_1}^{(\alpha,\beta)}(n,1)=\Big|\sum_{j=n+1}^{\infty}a^{(\alpha,\beta)}_n(x;f_1)P^{(\alpha,\beta)}_n(1)\Big|\le C\sum_{j=n+1}^{\infty} j^{-2\lambda-\alpha-1}j^{\alpha}\le C n^{-2\lambda}.
\end{split}
\end{equation*}

Finally, we obtain the weighted pointwise error estimate (\ref{PerrJac2C}) by analogous arguments as the proofs Theorem \ref{Thm13} and Corollary \ref{Bineua}.
\end{proof}

Based on \eqref{bd4} and Theorem \ref{thm1}, we may give some improved and optimal bounds for $\|f_i-S_n^{(\alpha,\beta)}[f_i]\|_{\infty}$ than those in \cite{XiangLiu}.

\begin{theorem}\label{ThmB2}
Let $f_1(x)=(1- x)^\lambda z(x)\; (\lambda+\alpha>-1)$ and $f_2(x)=(1+ x)^\lambda z(x)\; (\lambda+\beta>-1)$ with  $\lambda>0$  not an integer and $z(x)$ defined as above. Then for  $\min\{\alpha,\beta\}\ge -\frac{1}{2}$, it holds
 \begin{equation*}\label{PerrJac3B}
\|f_1-S_n^{(\alpha,\beta)}[f_1]\|_{\infty}={\cal O}(n^{-2\lambda+\max\{0,\beta-\alpha-1\}});\quad \|f_2-S_n^{(\alpha,\beta)}[f_2]\|_{\infty}={\cal O}(n^{-2\lambda+\max\{0,\alpha-\beta-1\}}).
\end{equation*}
\end{theorem}
\begin{proof}
Note  by $\|P_{n}^{(\alpha,\beta)}\|_{[0,1]}:=\max_{0 \le x \le 1} |P_{n}^{(\alpha,\beta)}|={\cal O}(n^{\alpha})$ (see Theorem \ref{Thm23}) that for $f_1$ and $x\in [0,1]$ it yields
$$
e_{f_1}^{(\alpha,\beta)}(n,x)=\Big|\sum_{j=n+1}^{\infty}a^{(\alpha,\beta)}_{j}(x;f_1)P^{(\alpha,\beta)}_{j}(x)\Big|\le C\sum_{j=n+1}^{\infty} j^{-2\lambda-\alpha-1}\|P^{(\alpha,\beta)}_{j}\|_{[0,1]}\le C_1 n^{-2\lambda}
$$
for some constants $C$ and $C_1$ independent of $x\in [0,1]$ and $n$.
While for $x\in[-1,0]$, by \eqref{bd4} and $\|P_n^{(\alpha,\beta)}\|_{[-1,0]}:=\max_{-1 \le x \le 0} |P_n^{(\alpha,\beta)}| ={\cal O}(n^{\beta})$ (see Theorem \ref{Thm23}) and  \eqref{Jtruerr}, it implies
\begin{equation*}\begin{split}
e_{f_1}^{(\alpha,\beta)}(n,x)&=\Big|A_n^{(\alpha, \beta)} {a_{n}^{(\alpha,\beta)}(x;g)P_{n+1}^{(\alpha,\beta)}(x)}- B_n^{(\alpha, \beta)} {a_{n+1}^{(\alpha,\beta)}(x;g)P_n^{(\alpha,\beta)}(x)}\Big|\\
&\le Cn^{-2\lambda-\alpha-1}\|P^{(\alpha,\beta)}_{n}\|_{[-1,0]}\\
&\le C_2 n^{-2\lambda-\alpha+\beta-1}\end{split}
\end{equation*}
for some constants $C_2$ independent of $x\in [0,1]$ and $n$. These together leads to the desired results for $f_1$. Analogously, the bound for $f_2$ is also satisfied.
\end{proof}

 In Fig. \ref{figure61}, we show the maximum error of $S_{n}^{(\alpha,\beta)}[f]$ and $p_{n}^{*}[f]$ as a function of $n$ for two functions $f_{1}(x)=(1-x)^{1/2}$ and $f_{2}(x)=(1+x)^{2/3}$. In order to consider the boundary behaviors around $x=\pm 1$, we also show a weighted maximum error defined in (\ref{PerrJac2C}). Clearly, numerical results are in good agreement with our theoretical estimates in Theorem \ref{ThmB} and Theorem \ref{ThmB2}, which implies as well the optimality of our estimates in the sense that the orders derived can no more be improved.

\begin{figure}[!t]
\centering
\includegraphics[width=.341\textwidth]{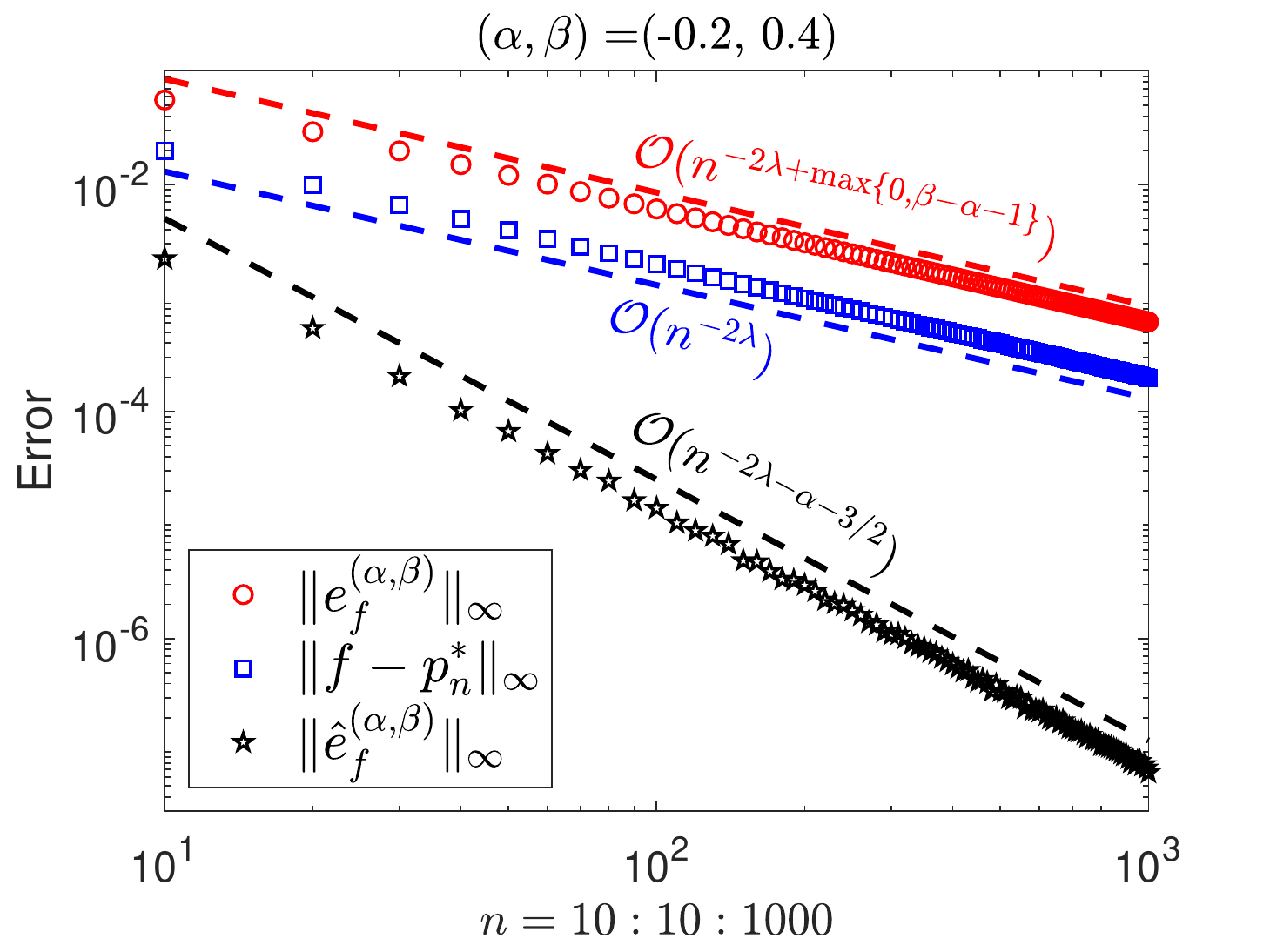}  %\hspace{-3pt}
\includegraphics[width=.325\textwidth]{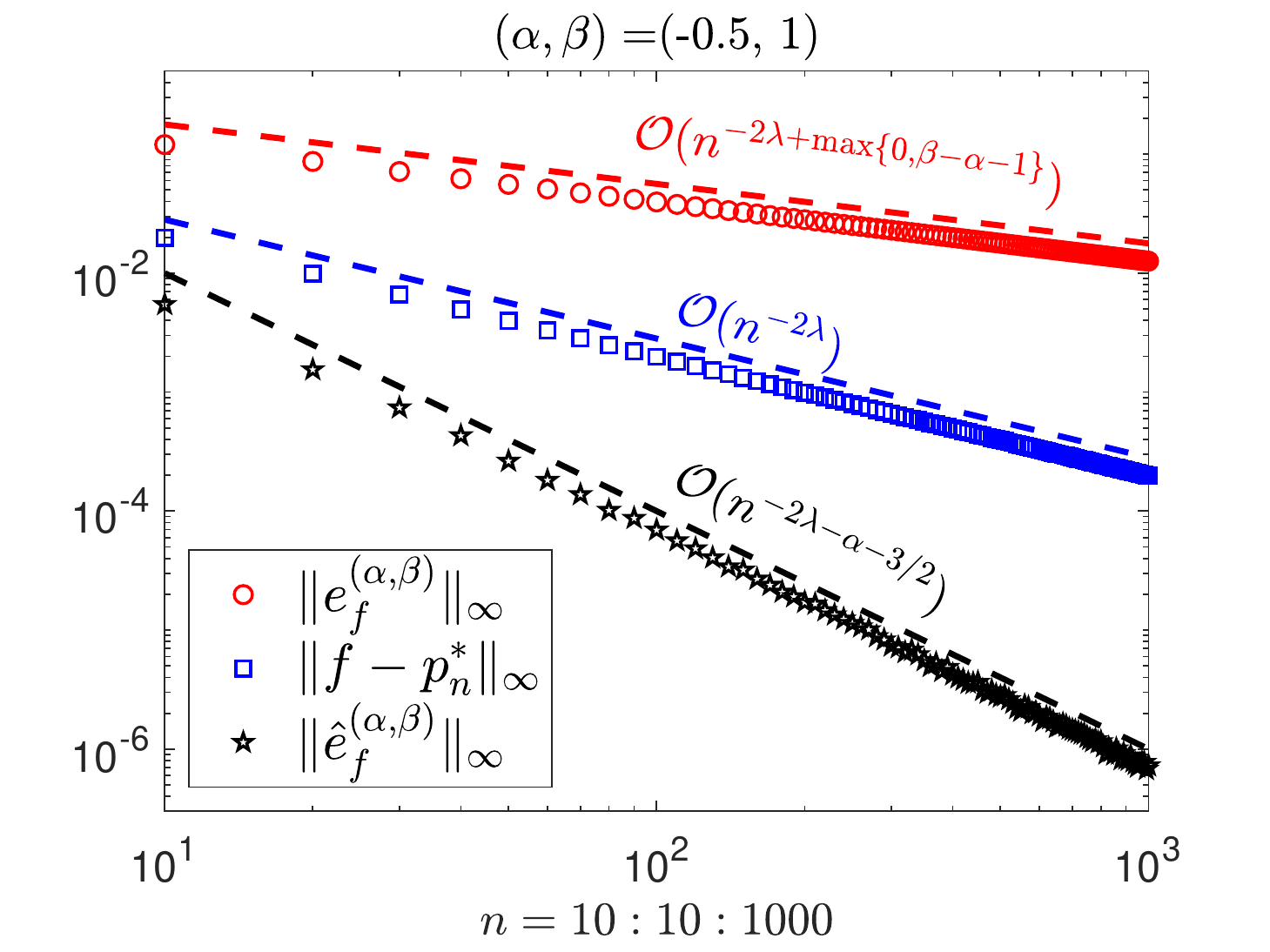}%\hspace{-3pt}
\includegraphics[width=.325\textwidth]{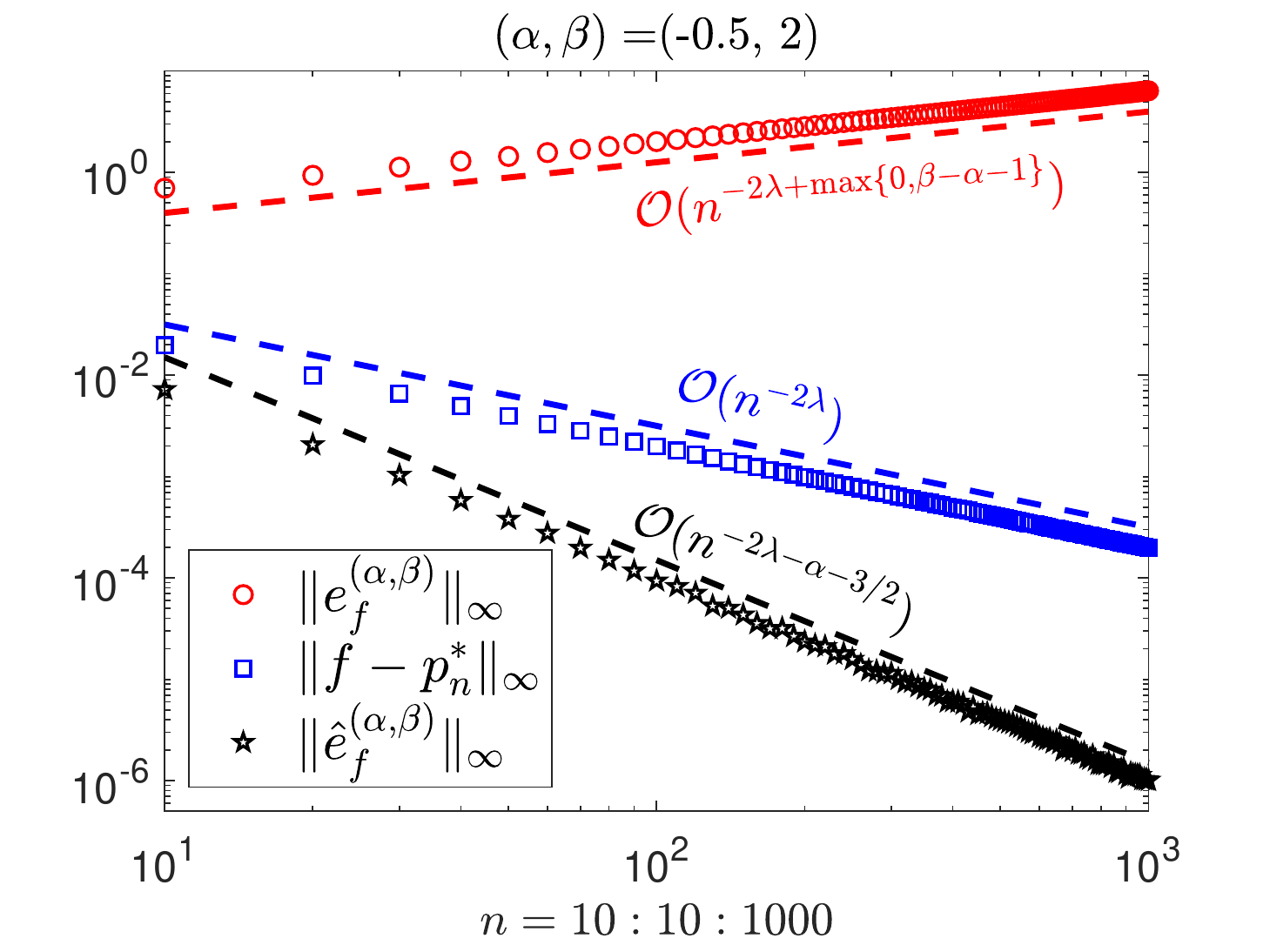} \\[10pt] % \vspace{10pt}
%\medskip
\includegraphics[width=.341\textwidth]{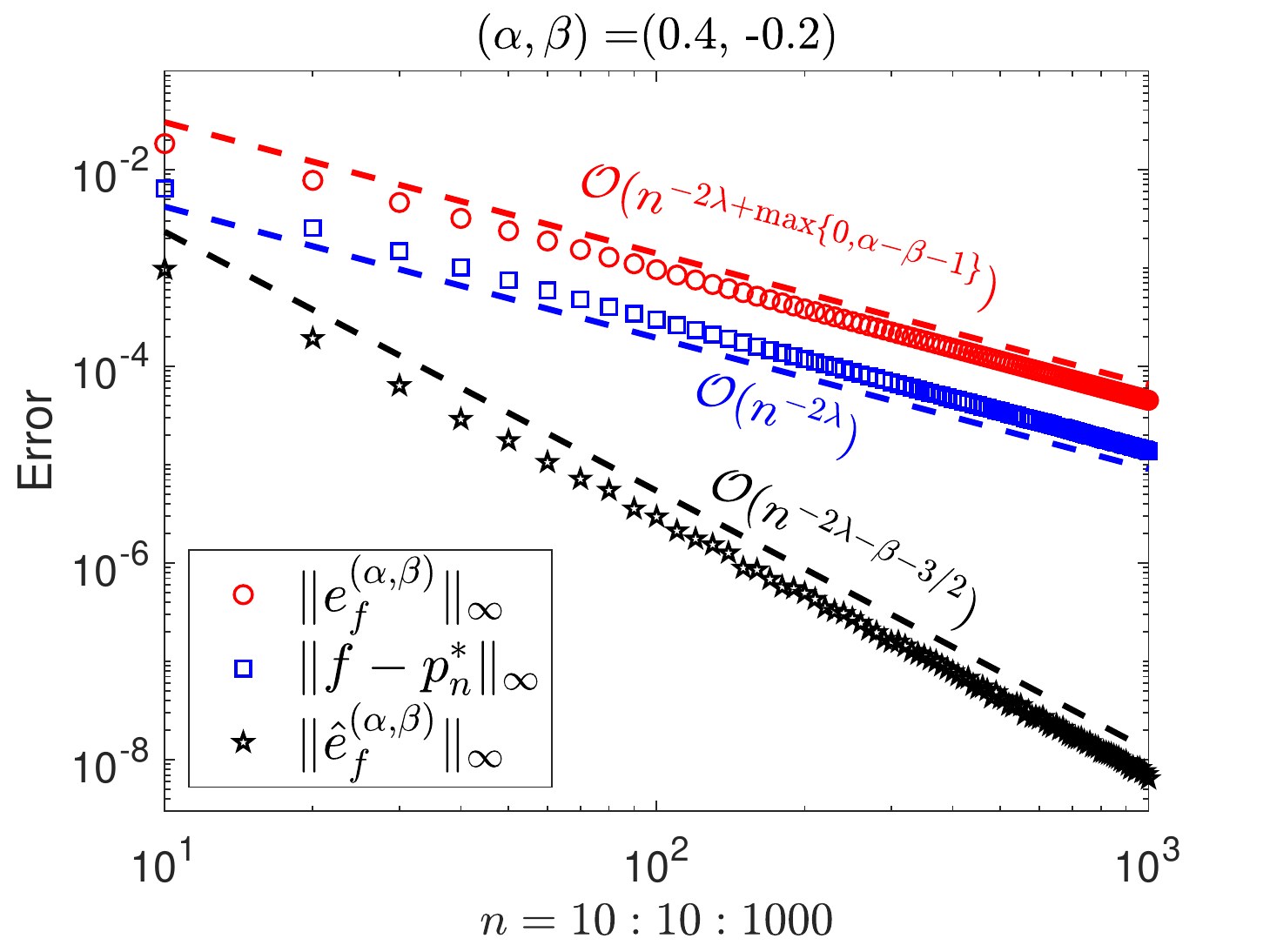}%\hspace{-3pt}
\includegraphics[width=.325\textwidth]{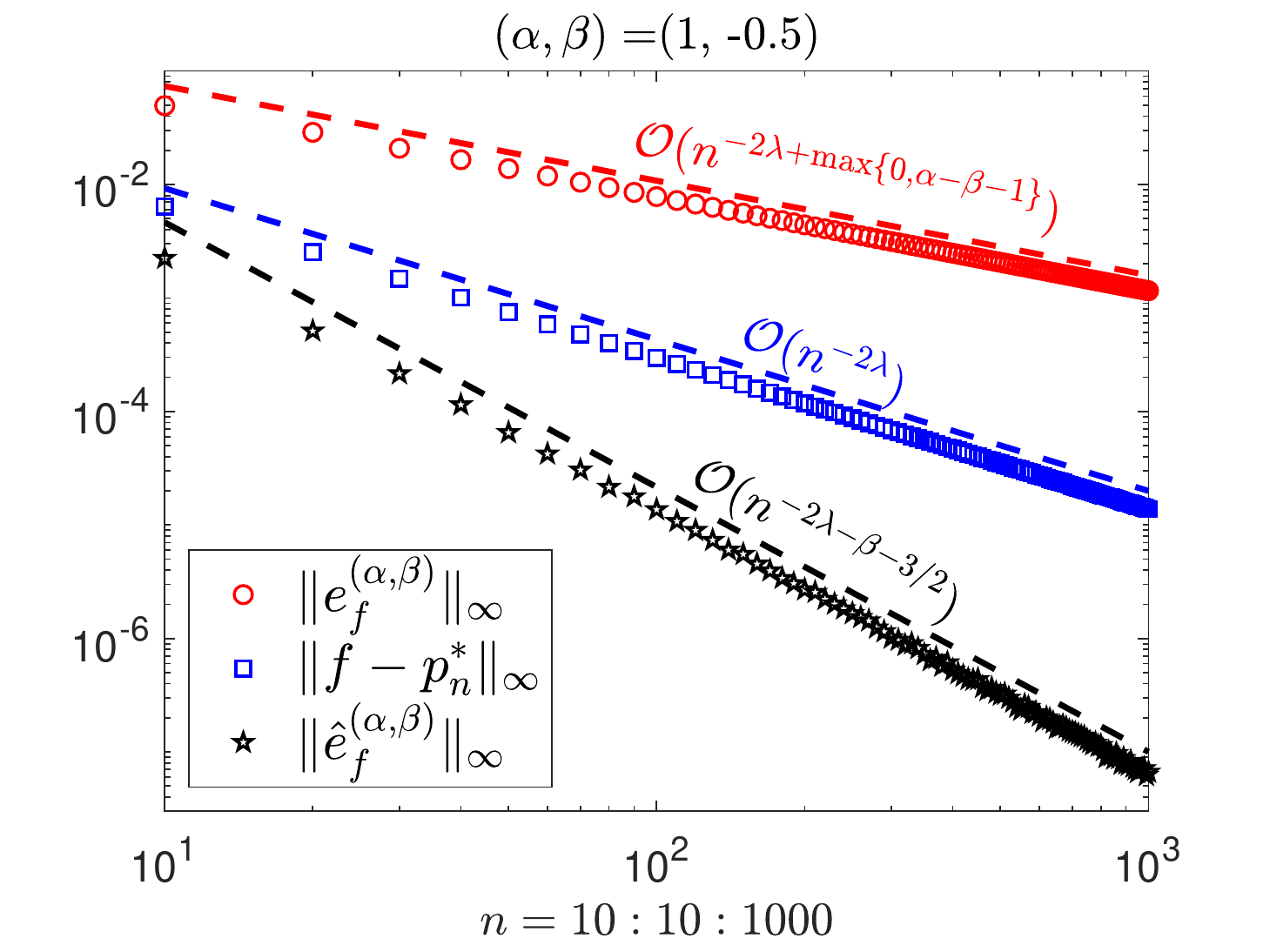}%\hspace{-3pt}
\includegraphics[width=.325\textwidth]{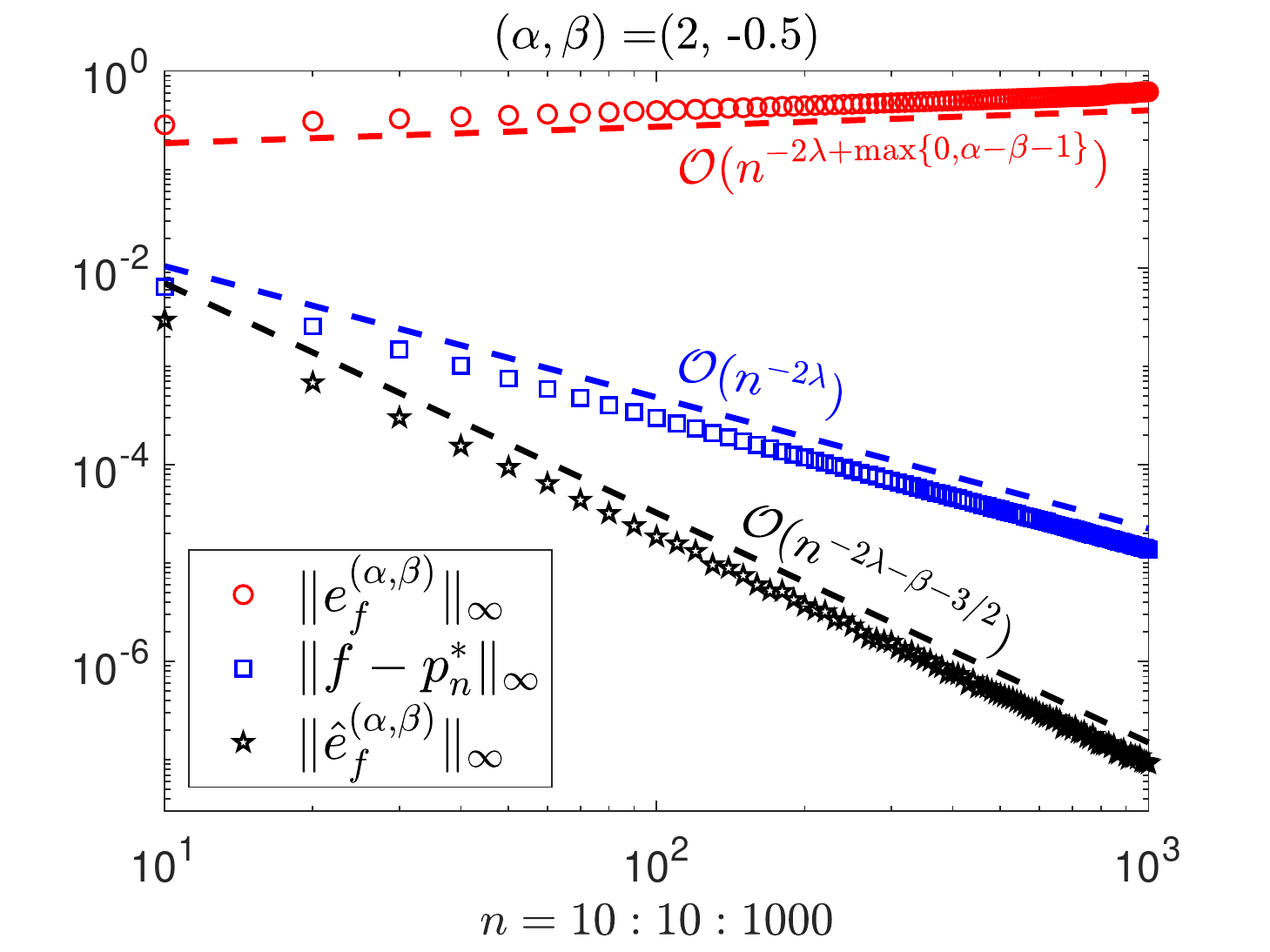}
\caption{Plots of $\|e_{f}^{(\alpha,\beta)}\|_{\infty}$ {\rm(}red circle{\rm)}, $\|f-p_{n}^{*}\|_{\infty}$ {\rm(}blue square{\rm)} and  the weighted maximum error $\|\hat{e}_{f}^{(\alpha,\beta)}\|_{\infty}$ {\rm(}black pentagram{\rm)} for $f_{1}(x)=(1-x)^{1/2}$ {\rm(}first row{\rm)} and $f_{2}(x)=(1+x)^{2/3}$ {\rm(}second row{\rm)}.}\label{figure61}
\end{figure}

\medskip
\begin{remark}
{\rm (i)} The local behaviors  of Legendre series have been extensively studied in  Wahlbin \cite[Theorem 3.3]{Wahlbin},
and was illustrated by numerical examples 6.2.b, 6.3.a-b and 6.4.a-b in \cite{Wahlbin} for some specific $x_0,$
\begin{equation}\label{Wal0}
    H(x)=\begin{dcases}
    0,&-1\le x\le 0,\\
    1,&0< x\le 1\end{dcases} \mbox{ with\quad } e_f(n,\frac{\sqrt{2}}{2})\le C \frac{\ln n}{n},\,\,  e_f(n,1)\le C \frac{\ln n}{\sqrt{n}},
  \end{equation}
or for $x_0$ a ``unit'' distance away from $0$
\begin{subequations}\label{Wal}
\begin{equation}\label{Wala}
    f(x)=|x|^{\frac{1}{2}} \mbox{\quad with\quad } e_f(n,x_0)\le C\sigma_n(x_0)n^{-\frac{3}{2}}\left(\ln n\right)^{\frac{3}{2}},
  \end{equation}
\begin{equation}\label{Walb}
    f(x)=\sqrt{1-x} \mbox{\quad with\quad } e_f(n,x_0)\le C\sigma_n(x_0)n^{-\frac{5}{2}}\left(\ln n\right)^{\frac{5}{2}},
  \end{equation}
\end{subequations}
\begin{subequations}\label{Wal1}
\begin{equation}\label{Wal1a}
    f(x)=|x|^{-\frac{1}{2}} \mbox{\quad with\quad } e_f(n,x_0)\le C\sigma_n(x_0)n^{-\frac{1}{2}}\left(\ln n\right)^{\frac{1}{2}},
  \end{equation}
\begin{equation}\label{Wal1b}
    f(x)=\frac{1}{\sqrt{1-x^2}} \mbox{\quad with\quad } e_f(n,x_0)\le C\sigma_n(x_0)n^{-\frac{1}{2}}\left(\ln n\right)^{\frac{1}{2}},
  \end{equation}
\end{subequations}
where $C$ is a constant independent of $n$, and $\sigma_n(x_0)=\min\{(1-x_0^2)^{-\frac{1}{4}},n^{\frac{1}{2}}\}$ (see \cite{Wahlbin} for more details).

\medskip
 {\rm (ii)} For  $f(x)=(x-a)_+^{\lambda}$, in more recently work by Babu\u{s}ka and  Hakula \cite{Babuska2019}, the above log-term in \eqref{Wal0} is omitted without the assumption $\delta\ge C_4 \frac{\ln n}{n}$ in Theorem 3.3 \cite{Wahlbin} if $\lambda=0$, and the log-term $\left(\ln n\right)^{\frac{3}{2}}$ in \eqref{Wala} and $\left(\ln n\right)^{\frac{1}{2}}$ in \eqref{Wal1a} is replaced by $\ln n$ respectively if $\lambda\not=0$.

\medskip
{\rm (iii)}
Following Wahlbin \cite[Theorem 3.3]{Wahlbin},	from Theorem \ref{Thm23} and  Corollary \ref{Bineua}, we may define $\widetilde{C}(n;x)=|x-a|^{-1}\sigma_n^{(\alpha,\beta)}(x)$ related to $n$ and $x$ instead of $C(x)$ in \eqref{PerrJac100} of Theorem \ref{Thm13}, while for boundary singularities, we define $\widetilde{C}_1(n;x)=(1-x)^{-1}\sigma_n^{(\alpha,\beta)}(x)$ and $\widetilde{C}_2(n;x)=(1+x)^{-1}\sigma_n^{(\alpha,\beta)}(x)$ instead of $C_1(x)$ and $C_2(x)$ in \eqref{PerrJac1B} in Theorem \ref{ThmB} respectively,
where
\begin{equation*}\begin{split}
&\sigma_n^{(\alpha,\beta)}(x)=\begin{dcases}\min\left\{(1-x)^{-\max\{\frac \alpha2+\frac 14,0\}},(n+1)^{\frac12}\right\},&\alpha<-\frac12\\
\min\left\{(1-x)^{-\max\{\frac \alpha2+\frac 14,0\}},(n+1)^{\alpha+\frac12}\right\},&\alpha\ge-\frac12
\end{dcases}\quad x\in [0,1];\\
&\sigma_n^{(\alpha,\beta)}(x)=\begin{dcases}\min\left\{(1+x)^{-\max\{\frac \beta2+\frac 14,0\}},(n+1)^{\frac12}\right\},&\beta<-\frac12\\
\min\left\{(1+x)^{-\max\{\frac \beta2+\frac 14,0\}},(n+1)^{\beta+\frac12}\right\},&\beta\ge-\frac12
\end{dcases}\quad x\in [-1,0]
\end{split}\end{equation*}
satisfied that $\sigma_n^{(0,0)}(x)\sim  \sigma_n(x)$, i.e., with the same order on $x$ and $n$ for Legendre series. Then from Theorems \ref{Thm13} and \ref{ThmB}, all the logarithmic factors in \eqref{Wal0}, \eqref{Wala}, \eqref{Walb}, \eqref{Wal1a} and \eqref{Wal1b} can be removed, improvement of \cite{Wahlbin} in Legendre series.
  As mentioned in Wahlbin \cite{Wahlbin}, these bounds  without
the logarithmic factors are sharp.
\end{remark}

\begin{remark}
From Theorem \ref{ThmB2}, we see that if $\beta-\alpha-1\leq 0$ {\rm(}resp. $\alpha-\beta-1\leq 0${\rm)} for $f_{1}(x)$ {\rm(}resp. $f_{2}(x)${\rm)}, $\|f_i-S_n^{(\alpha,\beta)}[f_i]\|_{\infty}={\cal O}(n^{-2\lambda})$ has the same order as $\|f_i-p_n^*\|_{\infty}$, $(i=1,2)$.
 \end{remark}

%%%%%%%%%%%%%%%%%%%%%%%%%%%%%%%%%%%%%

\begin{appendices}
\section{Proof of Case (iii) in Theorem \ref{Thm32}}\label{AppendixA}
  In Section 3, we presented  detailed proofs of (\ref{coe1}) for $x<a$ and $x=a$, respectively. In the following, we sketch the proof for Case (iii).

\medskip
  Case (iii) $x>a$: A routine computation from (\ref{Jcoe}) gives rise to
\begin{equation}\label{case31}
  \begin{split}
	&a_{n}^{(\alpha,\beta)}(x;g)\\
	=&\frac{1}{\sigma_{n}^{(\alpha,\beta)}}\left[\int_{-1}^{a}\!\frac{z(x)(a-y)^{\lambda}}{x-y}P_{n}^{(\alpha,\beta)}(y)\omega^{(\alpha,\beta)}(y)\,\mathrm{d}y
+\int_{a}^{1}\!\omega^{(\alpha,\beta)}(y)\frac{z(x)-z(y)}{x-y}(y-a)^{\lambda}P_{n}^{(\alpha,\beta)}(y)\,\mathrm{d}y\right.\\
	&\qquad\quad\left.+\int_{-1}^{1}\!\frac{(x-a)^{\lambda}-|y-a|^{\lambda}}{x-y}z(x)P_{n}^{(\alpha,\beta)}(y)\omega^{(\alpha,\beta)}(y)\,\mathrm{d}y\right].
\end{split}
\end{equation}
Similar to the proof in case (i), it is not difficult to verify that the first term in the right hand side of \eqref{case31} satisfies \eqref{coe1} and \eqref{coe11}. Since $z(x)$ is smooth on $[-1,1]$, and for any integer $k$,
 $ \partial_{y}^{k}\Big(\frac{z(x)-z(y)}{x-y}\Big)$
  is uniformly bounded by a constant independent of $x$. As a result, the second term in \eqref{case31} satisfies \eqref{coe1}-\eqref{coe11} as well by Lemma \ref{JEC}.
While for the third term, recalling the definition of $m$ ($m=\lambda-1$ if $\lambda$ is an integer, otherwise $m=\lfloor\lambda\rfloor$), and applying the Rodrigues' formula (\ref{Rodrigues}), we have
\begin{equation}\label{thirdpart}
\begin{split} &\int_{-1}^{1}\!\frac{(x-a)^{\lambda}-|y-a|^{\lambda}}{x-y}z(x)P_{n}^{(\alpha,\beta)}(y)\omega^{(\alpha,\beta)}(y)\,\mathrm{d}y\\
=&\frac{z(x)}{2^{m+1}(n)_{m+1}\sigma_{n}^{(\alpha,\beta)}} \\
&\quad \times \left\{\int_{-1}^{a}\!(a-y)^{\lambda-m-1}P_{n-m-1}^{(\alpha+m+1,\beta+m+1)}(y)\psi_{1}(x,y)\omega^{(\alpha+m+1,\beta+m+1)}(y)\,\mathrm{d}y\right.\\
	&\quad +\left.\int_{a}^{1}\!(y-a)^{\lambda-m-1}P_{n-m-1}^{(\alpha+m+1,\beta+m+1)}(y)\psi_{2}(x,y)\omega^{(\alpha+m+1,\beta+m+1)}(y)\,\mathrm{d}y\right\},
	\end{split}
\end{equation}
where 
\begin{equation*}
\begin{split} &{\displaystyle\psi_{1}(x,y)=(a-y)^{m+1-\lambda}\partial_{y}^{m+1}\Big(\frac{(x-a)^{\lambda}-(a-y)^{\lambda}}{x-y}\Big),\;\;\quad y\in[-1,a]},\\
	&{\displaystyle\psi_{2}(x,y)=(y-a)^{m+1-\lambda}\partial_{y}^{m+1}\Big(\frac{(x-a)^{\lambda}-(y-a)^{\lambda}}{x-y}\Big),\;\; \quad y\in[a,1]},\\
\end{split}
\end{equation*}
and smooth for $x>a$, and $-1\le y \le \frac {-1+a}{2}$ (resp. $\frac {1+a}{2}\le y \le 1$).
Similar to \eqref{g1}, $\psi_1(x,y)$ can be written as
\begin{equation*}
\psi_{1}(x,y)=\frac{(m+1)!} {x-y}\bigg[\frac{(x-a)^{\lambda}-(a-y)^{\lambda}}{(x-y)^{\lambda}}\Big(\frac{a-y}{x-y}\Big)^{m+1-\lambda}
+\sum_{k=1}^{m+1}\frac{(-1)^{m+1}(\lambda)_{k}}{k!}\Big(\frac{a-y}{x-y}\Big)^{m+1-k}\bigg]\\
\end{equation*}
and satisfies
\begin{equation*}
  |\psi_{1}(x,y)|\leq\frac{C_{1}}{x-y}, \quad
  |\partial_{y}\psi_{1}(x,y)|\leq\frac{C_{2}}{(x-y)^{2}},\;\; \forall\, y \in [(-1+a)/2,a],
\end{equation*}
for some constants $C_{1}$ and $C_{2}$ independent of $x$. Thus we have  $\psi_{1}(x,\cdot) \in W_{\rm AC}(\frac{-1+a}{2},a)$  and $\|\psi_{1}(x,\cdot)\|_{W^{1}_{\rm AC}(-1,\frac{-1+a}{2})}$ is uniformly bounded for all  $x$, which, together with  Lemma \ref{JEC},  leads to the desired result.

Now we turn to $\psi_{2}(x,y)$, after a calculation by the Leibniz's formula, leads to
\begin{equation}\label{A4a}
	\psi_{2}(x,y)=(m+1)!\, (y-a)^{m+1-\lambda}\frac{h(x,y)}{(x-y)^{m+2}},
\end{equation}
where
\begin{equation*}\label{A4ah}
  h(x,y)=(x-a)^{\lambda}-\sum_{k=0}^{m+1}\frac{(\lambda)_{k}}{k!}(y-a)^{\lambda-k}(x-y)^{k}.
\end{equation*}
 Similar to \eqref{Taylor}, we have for some $\xi$ between $x$ and $y$ that
\begin{equation}\label{sgnA4a}
 \frac{h(x,y)}{(x-y)^{m+2}}
 =\frac{(x-a)^{\lambda}-(y-a)^{\lambda}-\sum_{k=1}^{m+1}\frac{((y-a)^{\lambda})^{(k)}}{k!}(x-y)^{k}}{(x-y)^{m+2}} =\frac{(\lambda)_{m+2}(\xi-a)^{\lambda-m-2}}{(m+2)!}
 <0.
\end{equation}
Moreover, we find that $\psi_{2}(x,y)$ is monotonically increasing w.r.t $y$ (see the proof at the end of this section and numerical illustration in Fig. \ref{figureA1}). Therefore, we get from (\ref{A4a}) and (\ref{sgnA4a}) by letting $y\rightarrow a$ that
 \begin{equation*}\label{}
  \max_{y\in[a,\frac{1+a}{2}]}|\psi_{2}(x,y)|=|\psi_{2}(x,a)|=(\lambda)_{m+1}(x-a)^{-1},
\end{equation*}
 and
 \begin{equation*}\label{}
  \int_{a}^{\frac{1+a}{2}}\!\left|\partial_{y}\psi_{2}(x,y)\right|\mathrm{d}y=\left|\psi_{2}\Big(x,\frac{1+a}{2}\Big)-\psi_{2}(x,a)\right|\leq 2(\lambda)_{m+1}(x-a)^{-1}.
\end{equation*}
So the second integral in the right hand side of (\ref{thirdpart}) also satisfies (\ref{coe1}) as $\psi_{2}(x,\cdot)\in W_{\rm AC}(a,\frac{1+a}{2})$  and $\|\psi_{2}(x,\cdot)\|_{W^{1}_{\rm AC}(\frac{1+a}{2},1)}$ is uniformly bounded for all $x$. To sum up all the results above, we complete the proof of (\ref{coe1}).

In order to obtain the uniformly estimate (\ref{coe11}), we conduct integration by parts till $m$ instead of $m+1$. The proof is analogous to that in Case (i) and omitted here.

We end this section after providing a rigorous proof of the monotonicity of $\psi_{2}(x,y)$ with respect to $y$. For any fixed $x\in(a,1)$, we have
\begin{equation*}\label{diffpsi2}
  \partial_{y}\psi_{2}(x,y)=(m+1)!\frac{u(y)}{(x-y)^{m+3}},
\end{equation*}
where %$u(y)$ is defined as
\begin{equation*}
  u(y)=(x-y)\partial_{y}\left((y-a)^{m+1-\lambda}h(x,y)\right)+(m+2)(y-a)^{m+1-\lambda}h(x,y).
\end{equation*}
Interestingly, we can show that $u^{(k)}(x)=0$ for any $k\in\left\{0,1,\cdots,m+1\right\}$, that is,
\begin{equation*}
  \begin{aligned}
    & u^{(k)}(x)=(m+2-k)\partial_{y}^{k}\left((y-a)^{m+1-\lambda}h(x,y)\right)\Big|_{y=x}\\
	&=(m+2-k)\partial_{y}^{k}
      \Big[(x-a)^{\lambda}(y-a)^{m+1-\lambda}
      %\\&\qquad
      -\sum_{j=0}^{m+1}\frac{(\lambda)_{j}}{j!}(y-a)^{m+1-j}(x-y)^{j}\Big]_{y=x}\\
    &=(m+2-k)\Big[(x-a)^{\lambda}(m+1-\lambda)_{k}(y-a)^{m+1-\lambda-k}\\
    &\qquad-\sum_{j=0}^{k}\frac{(\lambda)_{j}}{j!}\binom{k}{j}((y-a)^{m+1-j})^{(k-j)}((x-y)^{j})^{(j)}\Big]_{y=x}\\
	&=(m+2-k)(x-a)^{m+1-k}\Big[(m+1-\lambda)_{k}-\sum_{j=0}^{k}\binom{k}{j}(-1)^{j}(\lambda)_{j}(m+1-j)_{k-j}\Big]\\
	&=0,
	\end{aligned}
\end{equation*}
and further
\begin{equation*}
  \begin{aligned}
	 u^{(m+2)}(y)
    &=(x-y)\partial_{y}^{m+3}\big((y-a)^{m+1-\lambda}h(x,y)\big)\\
    &=(x-y)\partial_{y}^{m+3}\Big((x-a)^{\lambda}(y-a)^{m+1-\lambda}-\sum_{j=0}^{m+1}\frac{(\lambda)_{j}}{j!}(y-a)^{m+1-j}(x-y)^{j}\Big)\\
	&=(m+1-\lambda)_{m+3}(x-y)(x-a)^{\lambda}(y-a)^{-\lambda-2}.
  \end{aligned}
\end{equation*}
 Obviously, it follows from $0\leq m+1-\lambda<1$ that
 $ \mathrm{sgn}\big((-1)^{m+2}u^{(m+2)}(y)\big)=\mathrm{sgn}(x-y).$

As a consequence, it deduces from Taylor's theorem that
\begin{equation}\label{monoton}
  \begin{aligned}
    \partial_{y}\psi_{2}(x,y)
    &=\frac{(m+1)!}{(x-y)^{m+3}}\Big(\sum_{k=0}^{m+1}\frac{u^{(k)}(x)}{k!}(y-x)^{k}+\frac{u^{(m+2)}(\xi)}{(m+2)!}(y-x)^{m+2}\Big)\\
    &=\frac{(-1)^{m+2}u^{(m+2)}(\xi)}{(m+2)(x-y)}\geq 0,\\
  \end{aligned}
\end{equation}
where $\xi=x+\eta(y-x)$ and $0<\eta<1$. While if $y=x$, we have
\begin{equation*}
  \begin{aligned}
  \partial_{y}\psi_{2}(x,x) 
  &=\lim_{y\to x}\partial_y\psi_2(x,y)
  = \lim_{y\to x} \Big\{ (m+1)! \frac{ u(y)}{(x-y)^{m+3}} \Big\} \\
  &=\frac{(-1)^{m}(m+1-\lambda)_{m+3}}{(m+3)(m+2)} (x-a)^{-2} \geq0,
  \end{aligned}
\end{equation*}
which together with (\ref{monoton}), complete the proof.
\begin{figure}[hpbt]
\centering{\includegraphics[height=4.5cm,width=12cm]{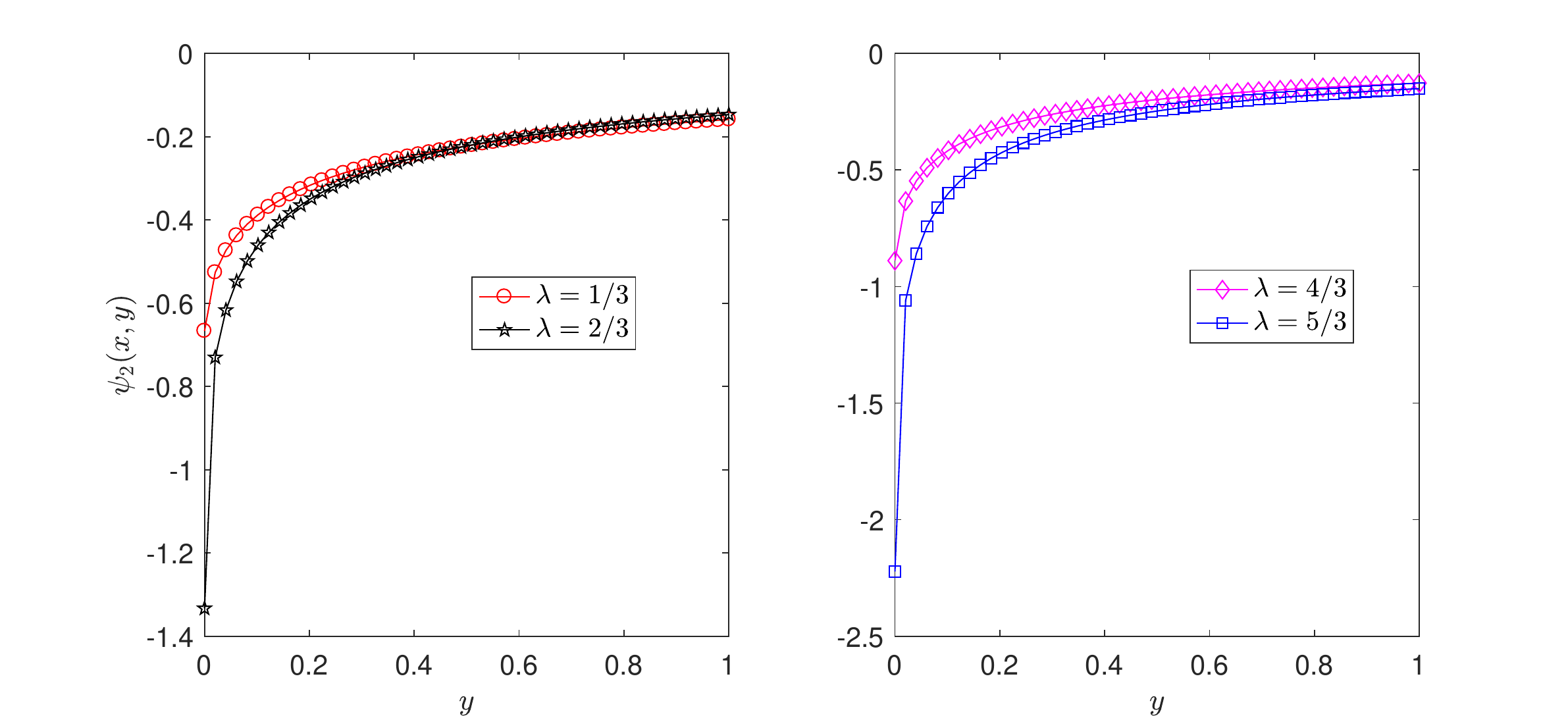}}
\caption{The monotonicity of $\psi_{2}(x,y)$ with respect to the argument $y$ on $[a,1]$, where $a=0$ and $x=1/2$.}\label{figureA1}
\end{figure}

%%%%%%%%%%%%%%%%%%%%%%%%%%
\section*{Acknowledgement}
The authors would like to thank the anonymous referees for the valuable comments and bringing several references to our attention, which led to significant improvement for this work. We also appreciate the referee for providing us useful details for the proof of Lemma 3.2 in the report, which indeed completed its proof.
The work  is  partially supported by the  NSF of China (No.12271528, No.12001280) for the first three authors;
the Fundamental Research Funds for the Central Universities of Central South University (No. 2020zzts031) for the second author;  the Natural Science Foundation of the Jiangsu Higher Education Institutions of China (No. 20KJB110012) and Natural Science Foundation of Jiangsu Province (No. BK20211293) for the third author;  and   Singapore MOE AcRF Tier 1 Grant: RG15/21 for the last author.
\end{appendices}
%%%%%%%%%%%%%%%%%%%%%%%%%%%%%%%%%%%%%%%%%%%%%%%%%%%%%%%


\begin{thebibliography}{10}

\bibitem{Abram}
{M. Abramowitz and I.~A. Stegun}, \emph{Handbook of Mathematical Functions}, National Bureau of Standards, Washington, D.C., 1964.

\bibitem{BaGuo1}
{I. Babu\u{s}ka and B.~Guo},  \emph{Optimal estimates for lower and upper bounds of approximation errors in the p-version of the finite element method in two dimensions}, Numer. Math., 85 (2000), pp. 219--255.

\bibitem{BaGuo2}
{I. Babu\u{s}ka and B. Guo}, \emph{Direct and inverse approximation theorems for the p-version of the finite element method in the framework of weighted Besov spaces. Part I: Approximability of functions in the weighted Besov spaces}, SIAM J. Numer. Anal., 39 (2001), pp. 1512--1538.

\bibitem{BaGuo3}
{I. Babu\u{s}ka and B. Guo}, \emph{Direct and inverse approximation theorems for the p-version of the finite element method in the framework of weighted Besov spaces. Part II: Optimal rate of convergence of the p-version finite element solutions}, Math. Models Methods Appl. Sci., 12 (2002), pp. 689--719.

\bibitem{Babuska2019}
{I. Babu\u{s}ka and H. Hakula}, \emph{Pointwise error estimate of the Legendre expansion: The known and unknown features}, Comput.  Methods Appl. Mech. Engrg., 345 (2019), pp. 748--773.

\bibitem{Bary}
{N. K. Bary}, \emph{A Treatise on Trigonometric Series}, Vol. 2, Macmillan, 1964.

\bibitem{Bernstein1}
{S. Bernstein}, \emph{On the best approximation of $|x-c|^{p}$}, Doll. Akad. Nauk SSSR, 18 (1938), pp. 374-384.

\bibitem{Bernstein2}
{S. Bernstein}, \emph{Sur l' les polynomes orthogonaux relatifs \`{a} un segment fini},
J. Math. Memories de l'. Academie Royale de Belgique, 9 (1930), pp. 127--177; 10 (1931), pp. 219--286.


\bibitem{CSYZ}		
{W.~Cao, Q. Shu, Y. Yang and Z. Zhang}, \emph{Superconvergence of discontinuous Galerkin methods for two-dimensional hyperbolic equations}, SIAM J. Numer. Anal., 53 (2015), pp. 1651--1671.

 \bibitem{CSYZ2}	
{W.~Cao, Q. Shu, Y. Yang and Z. Zhang}, \emph{Superconvergence of discontinuous Galerkin method for scalar nonlinear hyperbolic equations}, SIAM J. Numer. Anal., 56 (2018),  pp. 732--765.

\bibitem{CZZ}	
{W.~Cao, Z. Zhang and Q. Zou}, \emph{Superconvergence of discontinuous Galerkin methods for linear hyperbolic equations}, SIAM J. Numer. Anal., 52 (2014), pp. 2555--2573.


 \bibitem{CCS}
{P. Castillo, B. Cockburn, D. Sch\"{o}tzau and C. Schwab}, \emph{Optimal a priori error estimates for the hp-version of the local discontinuous Galerkin method for convection-diffusion problems}, Math. Comp., 71 (2002), pp. 238, 455--478.


\bibitem{Darboux}
{G.~Darboux}, \emph{M\'{e}moire sur l'approximation des fonctions de tr\`{e}s-grands nombres et sur une classe \'{e}tendue de d\'{e}veloppements en s\'{e}rie}, J. Math. Purer Appl., 4 (1978), pp. 5--56, 377--416.

\bibitem{Forster}
{F.~F\"{o}rster}, \emph{Inequalities for ultraspherical polynomials and application to quadrature}, J. Comput. Appl. Math., 49 (1993), pp. 59--70.

\bibitem{Hesthaven}
{J. Hesthaven, S. Gottlieb and D. Gottlieb}, \emph{Spectral Methods for Time-Dependent Problems}, Cambridge University Press, 2007.

\bibitem{Kruglov}
{E. Kruglov}, \emph{Pointwise convergence of Jacobi polynomials}, Master's Thesis, Aalto University, 2018.

\bibitem{LZ}
{R. Lin and Z. Zhang}, \emph{Natural superconvergence points in three-dimensional finite elements}, SIAM J. Numer. Anal., 46 (2008),  pp. 1281--1297.

\bibitem{LW}
{W. Liu and L. Wang}, \emph{Asymptotics of generalized Gegenbauer functions of fractional degree},  J. Appr. Theory, 253 (2020), 105378.


\bibitem{LWL}
{W. Liu, L. Wang and H. Li},
\emph{Optimal error estimates for Chebyshev approximations of functions with limited regularity in fractional Sobolev-type spaces}, Math. Comp., 88 (2019), pp. 2857--2895.

\bibitem{LWW}
{W. Liu, L. Wang and B. Wu}, \emph{
Optimal error estimates for Legendre approximation of singular functions with limited regularity},   Adv. Comput. Math., (2021), pp. 47--79. 

\bibitem{Lubinsky}
{D. S. Lubinsky}, \emph{A new approach to universality limits involving orthogonal polynomials}, Ann. Math., 170 (2009), pp. 915--939.

\bibitem{Muckenhoupt}
{B. Muckenhoupt}, \emph{Transplantation theorems and multiplier theorems}, Mem. Am. Math. Soc., 64 (1986), 356.

\bibitem{Olver}
{F.~W.~J. Olver, D.~W. Lozier, R.~F. Boisvert and C.~W. Clark}, \emph{NIST Handbook of Mathematical Functions}, Cambridge University Press, 2010.

\bibitem{st}
{E. Saff and  V. Totik}, \emph{ Polynomial approximation of piecewise analytic-functions}, J. Lond. Math. Soc., 39 (1989), pp. 487--498.

\bibitem{STW}
{J. Shen, T. Tang and L. Wang}, \emph{Spectral Methods: Algorithms, Analysis and Applications}, Springer, Berlin, 2011.

\bibitem{Stein}
{E. Stein}, \emph{Harmonic Analysis: Real-Variable Methods, Orthogonality, and Oscillatory Integrals}, Princeton University Press, Princeton, 1993.

\bibitem{Stein2003}
{E. Stein and R. Shakarchi}, \emph{Real Analysis: Measure Theory, Integration, and Hilbert Spaces},
Princeton University Press, Princeton, NJ, 2005

\bibitem{Szego}
{G.  Szeg\"{o}}, \emph{Orthogonal Polynomial}, Academic
Mathematical Society, 1939.

\bibitem{Tao}
{T. Tao}, \emph{An Introduction to Measure Theory}, AMS, Providence, RI, 2011.

\bibitem{Trefethen}
{L.~N. Trefethen}, \emph{Six myths of polynomial interpolation and quadrature}, Maths. Today, 47 (2011), pp. 184--188.


\bibitem{Trefethen1}
{\sc L.~N. Trefethen}, \emph{Approximation Theory and Approximation Practice}, SIAM, Philadelphia, 2013.


\bibitem{Tuan}
P. D. Tuan and D. Elliott, \emph{Coefficients in series expansions for certain classes of functions}, Math. Comp., 26 (1972), pp. 213--232.


\bibitem{Wahlbin}
{L. B. Wahlbin}, \emph{A comparison of the local behavior of spline $L^2$-projections, Fourier series and Legendre series}, in: P. Grisvard,W.Wendland, J. Whiteman (Eds.), Singularities and Constructive Methods for their Treatment, in: Lecture Notes in Mathematics, vol. 1121, Springer Berlin Heidelberg, 1985, pp. 319--346.

\bibitem{Wang}
{H. Wang}, \emph{How fast does the best polynomial approximation converge than Legendre projection}? Numer. Math., 147 (2021),  pp. 481--503.

\bibitem{Wang1}
{H. Wang}, \emph{On the optimal estimates and comparison of Gegenbauer expansion coefficients}, SIAM, J. Numer. Anal., 34 (2016), pp. 1557--1580.

\bibitem{Xiang2020}
{S. Xiang}, \emph{Convergence rates on spectral orthogonal projection approximation for functions of algebraic and logarithmatic regularities}, SIAM J. Numer. Anal., 59 (2021), pp. 1374--1398.


\bibitem{XiangLiu}
{S. Xiang and G. Liu}, \emph{Optimal decay rates on the asymptotics of orthogonal polynomial expansions for functions of limited regularities}, Numer. Math., 145 (2020), pp. 117--148.



\bibitem{Zhang}	
{Z.  Zhang}, \emph{Superconvergence points of polynomial spectral interpolation}, SIAM J. Numer. Anal., 50 (2012),  pp. 2966--2985.

\bibitem{ZhangZ}
{Z.  Zhang and A. Naga}, \emph{A new finite element gradient recovery method: superconvergence property}, SIAM J. Sci. Comput., 26 (2005),  pp. 1192--1213.

\bibitem{ZZ}
{X.  Zhao and Z. Zhang}, \emph{Superconvergence points of fractional spectral interpolation}, SIAM J. Sci. Comput., 38 (2016),  pp. A598--A613.


\end{thebibliography}
\end{document}